\documentclass{siamart171218}

\usepackage{lipsum}
\usepackage{amsfonts}
\usepackage{amssymb}
\usepackage{bm}
\usepackage{graphicx,epstopdf} % <- Preamble
\usepackage[caption=false]{subfig} % <- Preamble
\usepackage{tikz}
\usetikzlibrary{arrows.meta, positioning}
\usepackage{epstopdf}
\usepackage{algorithmic}
\ifpdf
  \DeclareGraphicsExtensions{.eps,.pdf,.png,.jpg}
\else
  \DeclareGraphicsExtensions{.eps}
\fi

\newsiamremark{remark}{Remark}
\newsiamremark{hypothesis}{Hypothesis}
\crefname{hypothesis}{Hypothesis}{Hypotheses}
\newsiamthm{claim}{Claim}

\headers{}{Xukun Wang, Suyash Shrestha, Oscar A. Marino and Esteban Ferrer}

\title{Projection-Based Reconstruction for Achieving High-Order Accuracy from Low-Order DGSEM Simulations\thanks{Submitted to the editors on XXX.
\funding{Xukun Wang acknowledges the financial support of the China Scholarship Council(CSC, project number: 202406290060). Esteban Ferrer and Oscar Marino acknowledge the funding from the European Union (ERC, Off-coustics, project number 101086075). Esteban Ferrer acknowledges the funding received by the Grant DeepCFD (Project No. PID2022-137899OB-I00) funded by MICIU/AEI/10.13039/501100011033 and by ERDF, EU.}}}

\author{
    Xukun Wang\thanks{ ETSIAE-UPM-School of Aeronautics, Plaza Cardenal Cisneros 3, E-28040 Madrid, Spain (\email{xukun.wang@alumnos.upm.es,oscar.marino@upm.es,esteban.ferrer@upm.es}).}
    \and
    Suyash Shrestha\thanks{ Delft University of Technology, Delft, Netherlands (\email{S.Shrestha-1@tudelft.nl}).}
    \and
    Oscar A. Marino\footnotemark[2]
    \and
    Esteban Ferrer\footnotemark[2]
}

\usepackage{amsopn}

\makeatletter
\newcommand*{\addFileDependency}[1]{% argument=file name and extension
  \typeout{(#1)}% latexmk will find this if $recorder=0 (however, in that case, it will ignore #1 if it is a .aux or .pdf file etc and it exists! if it doesn't exist, it will appear in the list of dependents regardless)
  \@addtofilelist{#1}% if you want it to appear in \listfiles, not really necessary and latexmk doesn't use this
  \IfFileExists{#1}{}{\typeout{No file #1.}}% latexmk will find this message if #1 doesn't exist (yet)
}
\makeatother

\newcommand*{\myexternaldocument}[1]{%
    \externaldocument{#1}%
    \addFileDependency{#1.tex}%
    \addFileDependency{#1.aux}%
}
\ifpdf
\hypersetup{
  pdftitle={Projection-Based Reconstruction for Achieving High-Order Accuracy from Low-Order DGSEM Simulations},
  pdfauthor={Xukun Wang, Suyash Shrestha, Oscar A. Marino and Esteban Ferrer}
}
\fi

\myexternaldocument{ex_supplement}

\begin{document}

\maketitle

% REQUIRED
\begin{abstract}
High-order discontinuous Galerkin spectral element methods (DGSEM) based on Legendre-Gauss-Lobatto (LGL) nodes provide accurate and efficient discretizations for conservation laws. However, their cost, memory footprint, and time-step restrictions increase rapidly when the degree of the polynomial increases. This paper develops a corrected $\mathbb{P}_n\mathbb{P}_m$ ($c\mathbb{P}_n\mathbb{P}_m$) approach for DGSEM-LGL discretizations that aims to recover the accuracy of an $m^{th}$-order approximation while evolving only the degrees of freedom associated with an $n^{th}$-order representation, with $n<m$. The projected evolution of the high-order components is derived first at the continuous level and then in the fully discrete DGSEM-LGL setting. The discrete analysis shows that because LGL quadrature is not exact for the highest Legendre mode, a correction term for that mode is required to preserve the order of convergence. A compact projection-based reconstruction operator is then introduced to recover high-order components without solving the enlarged constrained least-squares systems used in standard reconstruction procedures. For sufficiently smooth solutions, the resulting $c\mathbb{P}_n\mathbb{P}_m$ scheme is shown to achieve the expected $m+1^{th}$ convergence order. Numerical experiments for one- and two-dimensional conservation laws, including Euler, viscous Burgers, and 2D decaying homogeneous isotropic turbulence, confirm theoretical convergence behavior and demonstrate competitive accuracy relative to computational cost, with particularly clear efficiency gains for viscous flows.
\end{abstract}

% REQUIRED

% REQUIRED

%\tableofcontents

\section{Introduction}
The Discontinuous Galerkin (DG) method, since first introduced in 1973 by Reed and Hill \cite{osti_4491151}, has established itself as a powerful framework
for solving partial differential equations (PDEs) and has shown its distinguished capability in solving conservative systems in complex geometries, combining the high-order accuracy property of the finite element method (FEM) with the physical insights of Godunov-type methods (finite volume method, FVM) by applying a Riemann solver on the interfaces\cite{cockburn_tvb_1989, BR1, ferrer_high_2012}. Among alternative types, the collocation-type nodal DG \cite{hesthaven2008nodal} reduces computational complexity by representing the solution and fluxes using interpolation polynomials of the same degree. Furthermore, as one of the most efficient variants, the so-called discontinuous Galerkin collocation spectral element method (DGSEM) combines the flexibility of the nodal DG with spectral methods \cite{kopriva2009implementing, karniadakis_spectralhp_2005}, leading to outstanding performance in resolving multiscale phenomena in computational fluid dynamics (CFD)\cite{beck_highorder_2014,ferrer_horses3d_2023,kurz_galaexi_2025}, particularly when combined with Legendre–Gauss–Lobatto (LGL) quadrature nodes. The DGSEM-LGL approach \cite{gassner_comparison_2011} leverages collocation at quadrature points, leading to diagonal mass matrices and the discrete summation-by-parts (SBP) property, which is essential for building entropy-stable and energy-conservative schemes\cite{kopriva_quadrature_2010,fisher_high-order_2013, gassner_skew_symmetric_2013,kopriva_energy_2014, gassner_split_2016}. These features have made DGSEM-LGL a popular choice for large-scale simulations of compressible flows and turbulence.

Despite these advantages, the classic DG and DGSEM-LGL methods also face some challenges. First, achieving high-order accuracy typically requires increasing the order $p$ of the polynomial, leading to a rapid increase in the number of degrees of freedom per element. This, even worse, increases the computational cost and memory usage, which are bottlenecks in modern large-scale simulations. Second, a more restrictive CFL condition for the high-order scheme leads to an extremely small time step, which severely impairs efficiency. To address these issues, alternative approaches have been proposed to achieve high-order accuracy in low-order solutions. Dumbser et al.\cite{dumbser_unified_2008,dumbser_very_2009,dumbser_arbitrary_2010} proposed the $\mathbb{P}_n\mathbb{P}_m$ method, incorporating FVM and DG into a unified framework, where $\mathbb{P}_n$ polynomials are used to represent the solutions and $\mathbb{P}_m$ polynomials are reconstructed to compute the flux. Setting $n=0,\;m\geq2$ and $n=m\geq1$, it leads to the FVM and standard DG methods, respectively. Following a similar idea, Luo et al. \cite{luo_reconstructed_2009,luo_reconstructed_2010,lou_reconstructed_2018} developed the so-called reconstructed DG (rDG). Both are designed to improve the accuracy of the DG method and successfully achieve convergence of $m+1$-orders by evolving only solutions of $n^{th}$-orders. However, they share the same inevitable problem that high-order fluxes are directly involved when computing the spatial derivatives and that an additional high-order quadrature law needs to be implemented. 

To alleviate these problems, we combine the idea of the $\mathbb{P}_n\mathbb{P}_m$ method with DGSEM-LGL, and propose a novel framework: a corrected $\mathbb{P}_n\mathbb{P}_m$ method ($c\mathbb{P}_n\mathbb{P}_m$), where the $m+1$-order accuracy can be reached while all computations are implemented on the same low-order discrete LGL nodes. The projected high-order time-derivative is formulated exactly using projected high-order fluxes and low-order operators. The derivation shows that a correction term of highest Legendre mode has to be added to complement the inaccuracy of LGL quadrature and to maintain the designed convergence order.

In $\mathbb{P}_n\mathbb{P}_m$ and rDG methods, the reconstruction operator has a major influence on the accuracy of the resulting scheme. Unlike high-order finite-volume methods, in which each element typically contains only one degree of freedom, namely the cell average, discontinuous Galerkin methods provide several local degrees of freedom for $p \geq 1$. This additional information enables the use of highly compact reconstruction stencils. For sufficiently smooth solutions, a linear reconstruction operator is generally sufficient to achieve the target order of accuracy, and several reconstruction strategies have been proposed in the literature. For example, Dumbser et al.\cite{dumbser_arbitrary_2007,dumbser_unified_2008} developed a $L^2-$projection-based reconstruction method, while Luo et al. \cite{luo_reconstructed_2009,luo_reconstructed_2010} established the over-determined system by imposing constraints on the mean of the solution and its derivatives. Alternatively, Wang et al. \cite{wang_pnpm-cpr_2011} rebuilt the high-order solution based on the nodal values on the stencil within the framework of the flux reconstruction (FR) method \cite{huynh_flux_2007}. More recently, a global reconstruction method based on a variational formulation has been proposed by Li et al.\cite{li_reconstructed_2022}. As for the discontinuous solution, nonlinear reconstructors are indispensable to maintain the desired accuracy, which has been a research hot-spot in the fields of CFD for the last century. Namely, to capture local discontinuities in the numerical solution (e.g., shocks), a large number of advanced non-linear reconstruction methods have been well researched in FVM, among which the most popular is the family of essentially non-oscillatory (ENO) schemes \cite{harten_uniformly_1987}, such as weighted ENO (WENO)\cite{jiang_efficient_1996,qiu_hermite_2004} and targeted ENO (TENO)\cite{fu_low-dissipation_2019}. 

In this paper, we focus on enhancing the accuracy of smooth solutions instead of capturing discontinuities. Following the idea of general conservation and local $L^2-$projection, we develop a novel projection-based reconstruction algorithm. Instead of solving a constrained least-squares (LS) problem in an enlarged system in \cite{dumbser_unified_2008}, high-order components are directly computed from an equivalent small-scale (short stencil) problem, which increases the efficiency of the reconstruction. Coupling this reconstructor with the proposed corrected $\mathbb{P}_n\mathbb{P}_m$ framework leads to a highly-efficient new scheme, which is comparable to standard DGSEM in terms of accuracy per Flop with fewer memory requirements. The numerical results demonstrate its potential for large-scale simulation, especially in the case of limited memory and bandwidth. What's more, a brief error estimate is also provided, which demonstrates how the error of reconstruction affects the evolution of the total error.

The remainder of the paper is organized as follows. Section 2 introduces some background knowledge, and Section 3 provides an analysis of scale separation in the continuous case.  In Section 4, the simplified form of the projected time-derivative is carefully derived term by term. The novel modal reconstructor is introduced, and a brief error estimate is given in Section 5. Section 6 details the construction and convergence proof of the corrected $\mathbb{P}_n\mathbb{P}_m$ scheme. Section 7 provides the numerical results on one-dimensional linear advection, the viscous Burgers' equation, the Euler equation, two-dimensional isentropic vortex, and decaying homogeneous isotropic turbulence. Finally, the conclusion and future outlook are given in Section 8.

% The outline is not required, but we show an example here.
%The paper is organized as follows. Our main results are in
%\cref{sec:main}, our new algorithm is in \cref{sec:alg}, experimental
%results are in \cref{sec:experiments}, and the conclusions follow in
%\cref{sec:conclusions}.

\section{Preliminaries and Background}
\label{sec:basic_background}

For clarity, we first present the derivation in one dimension. The extension to two-dimensional cases is introduced in a later section, but the same methodology can be generalized straightforwardly to three dimensions.
\subsection{Problem formulation} Considering the general one-dimensional conservative system:
\begin{align}
    \label{eq:1d_conservation_law}
    \begin{split}
        \partial_t\vec{q} + \partial_x\vec{f}(\vec{q})&=0,\;(x,t)\in \Omega \times(0,+\infty),\\
        \vec{q} &= \vec{q}_0,\;\text{on } \Omega \times\{t=0\}\\
        \vec{q} &= \vec{q}_b,\;\text{on } \partial\Omega \times(0,+\infty)
    \end{split}        
\end{align}
where $\Omega \subseteq \mathbb{R}$, $x\in \Omega$ is the spatial coordinate, $t\in(0,+\infty)$ is the time, $\vec{q}:\Omega\times(0,+\infty)\to \mathbb{R}^s$, where $s$ is the number of components in the vector of conservative variables $\vec{q}$. $\vec{f}(\vec{q})$ is the vector of fluxes, $\vec{q}_0$ and $\vec{q}_b$ are the initial and boundary conditions, respectively. 

Assuming that $\mathcal{T}_h$ is a tessellation of the domain $\Omega$ into non-overlapping $K$ elements: $\Omega_i = [x_{i-1},x_{i}],i=1,2,\dots,K$(mesh), by transforming the physical element $\Omega_i$ into the reference one $E$, taking the inner product with the test function $\phi$, applying the integration by parts twice and replacing the boundary flux by a numerical one $\vec{f}^{\ast}$, the general strong-form DG formulation reads:
\begin{equation}
\label{eq:strong_form}
    \mathcal{J}_i\langle \partial_t\vec{q}\,,\phi \rangle + \left(\vec{f}^{\ast}-\vec{f}\right)\phi\Big\vert_{-1}^1+\langle\vec{f}_{\xi}\,,\phi\rangle = 0,
\end{equation}
where $\mathcal{J}_i$ is the Jacobian of the transfinite mapping from $E$ to $\Omega_i$ and $\langle \cdot\,,\cdot\rangle$ denotes the inner product of two functions defined in $E$.

\subsection{Discretization} We follow the standard discretization process in DGSEM-LGL\cite{gassner_comparison_2011,gassner_skew_symmetric_2013} and take only one component of the conservative system $u$ for simplicity:
\begin{equation}  
\label{eq:DGSEM-LGL_0}
\mathcal{J}_i\underline{\dot{u}}+\underline{\underline{M}}^{-1}\underline{\underline{B}}(\underline{f}^{\ast} - \underline{f})+\underline{\underline{D}}\,\underline{f} = 0,
\end{equation}
where $\underline{u}=[u_0,\,u_1,\dots,u_n]^{\top}\in \mathbb{R}^{n+1}$ and $\underline{f}=[f_0,\,f_1,\dots,f_n]^{\top}\in \mathbb{R}^{n+1}$ are the vectors of the nodal values of the solution and the corresponding flux in the element of $n^{th}$-order. $\underline{\underline{M}} = diag([\omega_0,\omega_1,\dots,\omega_n])$ is the so-called mass matrix (diagonal matrix of GL quadrature weights $w_i$). $\underline{\underline{B}} = diag([-1,\,0,\dots,0,\,1])$, $\underline{f}^{\ast} = [f^{\ast}_L,\, 0,\dots,\,0,\,f^{\ast}_R]^{\top}$ and $\underline{\underline{D}}$ is the derivative matrix. Alternatively, a more explicit form reads:
\begin{equation}
\label{eq:DGSEM-LGL_1}
    \mathcal{J}_i\underline{\dot{u}}+\frac{(f^{\ast}_L - f_0)}{\omega_0}\underline{b}_L + \frac{(f^{\ast}_R - f_n)}{\omega_n}\underline{b}_R+\underline{\underline{D}}\,\underline{f} = 0,
\end{equation}
where $\underline{b}_L = [-1,\,0,\dots,0]^{\top}$ and $\underline{b}_R = [0,\dots,0,\,1]^{\top}$. For a detailed description of the discretization process that leads to \cref{eq:DGSEM-LGL_0} and \cref{eq:DGSEM-LGL_1}, the reader is referred to \cref{sec: AppendixA}.

\subsection{Notations of high-order and low-order approximations}
Since approximations of two different polynomial orders are used throughout the paper, we first clarify the notation. To remain consistent with the conventional $\mathbb{P}_n\mathbb{P}_m$ framework, the low- and high-order approximations are denoted by $n$ and $m$, respectively. Consequently, vectors and matrices with superscript $(m)$ refer to their high-order forms, whereas those with superscript $(n)$ denote the corresponding low-order forms. In modal space, filtering is performed by truncating the higher-order modes. Therefore, the vector of high-order modal coefficients can be decomposed into low- and high-order components:
\begin{equation}
\label{eq:u_f_vectors_decomposition}
    \underline{\hat{u}}^{(m)} = \left[\underline{\hat{u}}^{(n)\top},\underline{\tilde{\hat{u}}}^{\top}\right]^{\top}\in \mathbb{R}^{m+1},
\end{equation}
where $\underline{\hat{u}}^{(n)} = \overline{\underline{\hat{u}}^{(m)}} = [\hat{u}_0,\,\hat{u}_1,\dots,\hat{u}_n]^{\top}\in \mathbb{R}^{n+1}$ and $\underline{\tilde{\hat{u}}} = [\hat{u}_{n+1},\,\hat{u}_{n+2},\dots,\hat{u}_m]^{\top}\in \mathbb{R}^{r}$ contain only $r=m-n$ high-order components higher than $n$ and the same for the flux $f$. Equivalently, we decompose the high-order solution, $u^{(m)}$, into a filtered part and high-order components:
\begin{equation}
\label{eq:u_decomposition}
    u^{(m)} = \overline{u^{(m)}} + \tilde{u} = u^{(n)} + \tilde{u}.
\end{equation}
Without specific instructions, the nodal representations of solution are expressed on the LGL nodes of the same order, i.e, 
\begin{equation}
    \underline{u}^{(m)} = \left[ u^{(m)}(\xi_0^{(m)}),\, u^{(m)}(\xi_1^{(m)}), \dots, u^{(m)}(\xi_m^{(m)})  \right]^{\top}\in \mathbb{R}^{m+1},
\end{equation}
where $\{\xi^{(m)}_i \}_{i=0}^m$ denotes the set of $m^{th}$-order LGL nodes and the same for $n^{th}$-order case. 

\section{Scale separation in the continuous case}
Starting from the integrated strong form in \cref{eq:strong_form}, we split the approximation space into two mutually $L^2$-orthogonal subspaces (low-order and higher-order space) using a $L^2$-projector $P$:
\begin{equation}    \mathcal{V}=\overline{\mathcal{V}}\oplus\widetilde{\mathcal{V}},
\end{equation}
where $\overline{\mathcal{V}} = \text{Range}\left( P \right)$ and $\widetilde{\mathcal{V}} = \text{ker}\left( P \right)$. The uniqueness of this decomposition follows from the boundedness of the projector $P$.
The same can be done for the solution and test function:
\begin{equation}
\label{eq:decomposition_VMS}
    q=\overline{q}+\tilde{q},\;\phi = \overline{\phi}+\tilde{\phi},
\end{equation}
where $\overline{q},\overline{\phi}\in\overline{\mathcal{V}}$ and $\tilde{q},\tilde{\phi}\in\widetilde{\mathcal{V}}$. This decomposition is similar to that of the variational multiscale method (VMS)\cite{hughes_variational_1998}, where $\overline{q}$ and $\tilde{q}$ are typically called \emph{coarse scale} and \emph{fine scale}, respectively. Taking a set of $L^2$-orthogonal basis $\{\phi_i\}_{i=0}^{\infty}$ that span the space $\mathcal{V}$, these two subspaces can be expanded by the following:
\begin{equation}
    \overline{\mathcal{V}} = span\{ \phi_0,\phi_1,\dots,\phi_n\}, \;\widetilde{\mathcal{V}}=span\{\phi_{n+1},\phi_{n+2},\dots\}.
\end{equation}
It should be noted that the dimension of $\overline{\mathcal{V}}$ is $n+1$ while the dimension of $\widetilde{\mathcal{V}}$ is infinite. To obtain an exact evolution of low-order components, (the coarse scale $\overline{q})$, we just take the weight function as the low-order one, $\overline{\phi}$, in the integrated strong form \cref{eq:strong_form} and insert the decomposition \cref{eq:decomposition_VMS}:
\begin{equation}
    \mathcal{J}_i\langle \partial_t\overline{q} + \partial_t\tilde{q}, \overline{\phi}\rangle+\left( f^{\ast}-f\right)\overline{\phi}\Big\vert_{-1}^1+\langle \overline{\partial_{\xi}f}+\widetilde{\partial_{\xi}f}, \overline{\phi}\rangle=0,
\end{equation}
where the same decomposition is also applied to the derivative of the flux:
\begin{equation}
    \partial_{\xi}f = \overline{\partial_{\xi}f}+\widetilde{\partial_{\xi}f},\;\overline{\partial_{\xi}f}\in\overline{\mathcal{V}},\,\widetilde{\partial_{\xi}f}\in\widetilde{\mathcal{V}}.
\end{equation}
Using the orthogonality of $\{\phi_i\}_{i=0}^{\infty}$, i.e., $\langle\phi_i,\phi_j\rangle=\delta_{ij}$ (here the basis can be considered as the global Legendre modes and the $L_2$-orthogonality is ensured), we obtain the following
\begin{equation}
\label{eq:strong_form_VMS_0}
    \mathcal{J}_i\langle \partial_t\overline{q}, \overline{\phi}\rangle+\left( f^{\ast}-f\right)\overline{\phi}\Big\vert_{-1}^1+\langle \overline{\partial_{\xi}f}, \overline{\phi}\rangle=0.
\end{equation}
We emphasize that projection and differentiation operators do not commute with each other, i.e., $\overline{\partial_{\xi}f}\neq\partial_{\xi}\overline{f}$. However, we will show that $\overline{\partial_{\xi}f}$ can be expressed in terms of $\partial_{\xi}\overline{f}$ plus an additional surface term.
Considering
\begin{equation}
    \overline{\partial_{\xi}f} = \partial_{\xi}\overline{f} + \left( \overline{\partial_{\xi}f}- \partial_{\xi}\overline{f} \right),
\end{equation}
where
\begin{equation}
    \overline{\partial_{\xi}f}- \partial_{\xi}\overline{f} = \overline{\partial_{\xi}f- \partial_{\xi}\overline{f}} = \overline{\partial_{\xi}\tilde{f}}.
\end{equation}
Although $\tilde{f}\in\widetilde{\mathcal{V}}$, the differentiation operator, $\partial_{\xi}$, will drop some components of $\tilde{f}$ into the low-order space, thus $\overline{\partial_{\xi}\tilde{f}}\neq0$. Therefore, the last term in \cref{eq:strong_form_VMS_0} can be reformulated
\begin{multline}
\label{eq: partial_xi_phi}
    \langle \overline{\partial_{\xi}f}, \overline{\phi}\rangle = \langle \partial_{\xi}\overline{f}, \overline{\phi}\rangle + \langle \overline{\partial_{\xi}\tilde{f}}, \overline{\phi}\rangle=\langle \partial_{\xi}\overline{f}, \overline{\phi}\rangle + \langle \partial_{\xi}\tilde{f}, \overline{\phi}\rangle\\
    =\langle \partial_{\xi}\overline{f}, \overline{\phi}\rangle + \tilde{f}\,\overline{\phi}\Big\vert_{-1}^1-\langle \tilde{f}, \partial_{\xi}\overline{\phi}\rangle= \langle \partial_{\xi}\overline{f}, \overline{\phi}\rangle + \tilde{f}\,\overline{\phi}\Big\vert_{-1}^1,
\end{multline}
where integration by parts has been used. Substituting \cref{eq: partial_xi_phi} into \cref{eq:strong_form_VMS_0} and using $f = \overline{f}+\tilde{f}$ produces
\begin{equation}
\label{eq: exact_low_order_evolution}
    \mathcal{J}_i\langle \partial_t\overline{q}, \overline{\phi}\rangle+\left( f^{\ast}-\overline{f}\right)\overline{\phi}\Big\vert_{-1}^1+\langle \partial_{\xi}\overline{f}, \overline{\phi}\rangle=0.
\end{equation}
    
\begin{remark}
    From \cref{eq: exact_low_order_evolution} we can conclude that the exact evolution of the coarse scale is governed by the flux in low-order space, $\overline{f}$, and the exact interfacial numerical fluxes $f^{\ast}$, and it is equivalent to taking $\overline{\phi}$ as the test function in \cref{eq:strong_form}.
\end{remark}

\section{Derivation on projected high-order time-derivative}
Instead of discretizing \cref{eq: exact_low_order_evolution} directly, we derive the projected high-order time derivatives by applying the projector to the entire semi-discrete form of DGSEM-LGL. All discrete operators used in the following are well defined, including the transformation operators between the nodal and modal spaces, i.e., $\underline{\underline{V}}$ and $\underline{\underline{V}}^{-1}$, and the projection and interpolation matrices, i.e., $\underline{\underline{\Pi}}$ and $\underline{\underline{\Pi}}^{\ast}$. For readers unfamiliar with these operators, details are included in \cref{sec: AppendixB}. We will show that the numerical scheme obtained using this approach mimics the continuous one \cref{eq: exact_low_order_evolution} discretely. The projected high-order time derivative reads as follows:
\begin{equation}
\label{eq:projected_DGESM_0}    \mathcal{J}_i\overline{\underline{\dot{u}}^{(m)}}+\underline{\underline{\Pi}}\left[\frac{\left(f^{\ast(m)}_L - f_0^{(m)}\right)}{\omega_0^{(m)}}\underline{b}_L^{(m)} + \frac{\left(f^{\ast(m)}_R - f_m^{(m)}\right)}{\omega_m^{(m)}}\underline{b}_R^{(m)}+\underline{\underline{D}}^{(m)}\,\underline{f}^{(m)}\right] = 0.
\end{equation}

For clarity, we deal with the surface and volume terms separately in the following.

\subsection{Projected surface term}
\label{sec:projected_surface_term}
Using the definition of the projection operator in \cref{eq:projection}, the projected surface terms are:
\begin{equation}
    \frac{\left(f^{\ast(m)}_L - f_0^{(m)}\right)}{\omega_0^{(m)}}\underline{\underline{\Pi}}\,\underline{b}^{(m)}_L,\;\frac{\left(f^{\ast(m)}_R - f_m^{(m)}\right)}{\omega_m^{(m)}}\underline{\underline{\Pi}}\,\underline{b}^{(m)}_R.
\end{equation}
Using the explicit form of the inverse Vandermonde matrix \cref{eq: inverse_Vandermonde} together with the boundary values of Legendre polynomials($L_k(\xi_0) = (-1)^{k}$, $L_k(\xi_m) = 1,\,k=0,1,\dots,m$), one obtains
\begin{equation}
\label{eq: Pi_b_L}
    \frac{1}{\omega_0^{(m)}}\underline{\underline{\Pi}}\,\underline{b}^{(m)}_L = \underline{\underline{V}}^{(n)}\left[\underline{\underline{\hat{M}}}^{(n)}\right]^{-1}{\underline{\underline{V}}^{(n)}}^{\top}\underline{b}_L^{(n)},
\end{equation}
and similarly,
\begin{equation}
\label{eq: Pi_b_R}
    \frac{1}{\omega_m^{(m)}}\underline{\underline{\Pi}}\,\underline{b}^{(m)}_R = \underline{\underline{V}}^{(n)}\left[\underline{\underline{\hat{M}}}^{(n)}\right]^{-1}{\underline{\underline{V}}^{(n)}}^{\top}\underline{b}_R^{(n)}.
\end{equation}
Furthermore, the following Lemma provides a convenient simplification.
\begin{lemma}
\label{lemma:VNVT_VtildeNVT}
For the Legendre Vandermonde matrix evaluated in the $n^{th}$-order LGL nodes, $\underline{\underline{V}}^{(n)}$, together with the corresponding nodal mass matrix, $\underline{\underline{M}}^{(n)}$ and the modal mass matrix, $\underline{\hat{\underline{M}}}^{(n)}$, the following relation holds:
\begin{displaymath}    \underline{\underline{V}}^{(n)}\left[\underline{\underline{\hat{M}}}^{(n)}\right]^{-1}{\underline{\underline{V}}^{(n)}}^{\top} = \left[\underline{\underline{M}}^{(n)}\right]^{-1} + \frac{n+1}{2}{\underline{v}_n^{(n)}}{{\underline{v}_n^{(n)}}}^{\top},
\end{displaymath}
where $\underline{v}_n^{(n)}$ is the last column of $\underline{\underline{V}}^{(n)}$.
\end{lemma}
\begin{proof}
    Using the relation between the exact and modified modal mass matrices,
    \begin{displaymath}
        \underline{\underline{\hat{M}}}^{(n)} = \underline{\underline{\hat{M}^{\prime}}}^{(n)} + \left( \left[ \underline{\hat{M}}^{(n)} \right]_{nn} - \left[ \underline{\underline{\hat{M}^{\prime}}}^{(n)} \right]_{nn} \right){\underline{e}_n^{(n)}}{\underline{e}_n^{(n)}}^{\top},
    \end{displaymath}
    together with
    \begin{displaymath}
        \left[\underline{\underline{M}}^{(n)}\right]^{-1} =\underline{\underline{V}}^{(n)}\left[\underline{\underline{\hat{M}^{\prime}}}^{(n)}\right]^{-1}{\underline{\underline{V}}^{(n)}}^{\top},
    \end{displaymath}
    immediately gives the result.
\end{proof}

%Integrating \cref{lemma:VNVT_VtildeNVT} into \cref{eq: Pi_b_L} and \cref{eq: Pi_b_R}, we obtain the following proposition:
\begin{proposition}
\label{prop:projected_surface_terms}
The projected high-order surface terms in \cref{eq:projected_DGESM_0} can be written as its low-order counterpart plus a corrective term expressed in the highest retained Legendre mode: 
\begin{equation}
\label{eq:projected_ho_surface_terms}
\begin{split}
    &\underline{\underline{\Pi}}\left[\frac{\left(f^{\ast(m)}_L - f_0^{(m)}\right)}{\omega_0^{(m)}}\underline{b}_L^{(m)} + \frac{\left(f^{\ast(m)}_R - f_m^{(m)}\right)}{\omega_m^{(m)}}\underline{b}_R^{(m)}\right] \\
    = &\frac{\left(f^{\ast(m)}_L - f_0^{(m)}\right)}{\omega_0^{(n)}}\underline{b}_L^{(n)} + \frac{\left(f^{\ast(m)}_R - f_m^{(m)}\right)}{\omega_n^{(n)}}\underline{b}_R^{(n)}\\
    &+\frac{n+1}{2}\left[ \left(f^{\ast(m)}_R - f_m^{(m)}\right) - (-1)^n\left(f^{\ast(m)}_L - f_0^{(m)}\right) \right]\underline{v}^{(n)}_n.
\end{split}  
\end{equation}
\end{proposition}
\begin{proof}
    The result follows directly from \cref{lemma:VNVT_VtildeNVT} after substitution into the projected surface terms .
\end{proof}

\begin{remark}
Recall that the original low-order ($n-$order) surface terms are just
\begin{displaymath}
    \frac{\left({f^{\ast(n)}_L}-{f_0^{(n)}}\right)}{{\omega_0^{(n)}}}{\underline{b}_L^{(n)}} + \frac{\left({f^{\ast(n)}_R}-{f_n^{(n)}}\right)}{{\omega_n^{(n)}}}{\underline{b}_R^{(n)}},
\end{displaymath}
The error in the low-order approximation relative to the projected high-order solution arises from two sources. The first is associated with the accuracy of the nodal flux values at the element boundaries and, consequently, with the accuracy of the numerical flux. The second can be interpreted as the quadrature error introduced by the LGL rule when integrating polynomials of degree $2n$. Therefore, the terms in the third row of \cref{eq:projected_ho_surface_terms} can be considered as a correction to the inexactness of the LGL quadrature for high-order polynomials .
\end{remark}

\subsection{Projected volume term}
Following a procedure analogous to that described in \cref{sec:projected_surface_term}, the projected volume term can be written as:
\begin{equation}
\label{eq:projected_volume_term_0}
\underline{\underline{\Pi}}\,\underline{\underline{D}}^{(m)}\underline{f}^{(m)} = \underline{\underline{V}}^{(n)}\overline{\left[ \underline{\underline{V}}^{(m)} \right]^{-1}\underline{\underline{D}}^{(m)}\underline{f}^{(m)}}.
\end{equation}
Introducing the modal derivative matrix:
\begin{equation}
\label{eq:modal_derivative_matrix}
    \underline{\underline{\hat{D}}}^{(m)}:=\left[ \underline{\underline{V}}^{(m)} \right]^{-1}\underline{\underline{D}}^{(m)}\underline{\underline{V}}^{(m)},
\end{equation}
we obtain
\begin{equation}
    \left[ \underline{\underline{V}}^{(m)} \right]^{-1}\underline{\underline{D}}^{(m)}\underline{f}^{(m)} = \underline{\underline{\hat{D}}}^{(m)}\underline{\hat{f}}^{(m)}.
\end{equation}
The modal derivative matrix possesses a sparse alternating-parity structure and satisfies the following identity.
\begin{corollary}
\label{col:D=2N^-1Q}
    The spatial differentiation matrix (modal bases) $\underline{\underline{\hat{D}}}:= \underline{\underline{V}}^{-1}\underline{\underline{D}}\,\underline{\underline{V}}$ can be reformulated as:
\begin{displaymath}
    \underline{\underline{\hat{D}}} = 2\underline{\underline{\hat{M}}}^{-1}\underline{\underline{\hat{Q}}}
\end{displaymath}
where $\underline{\underline{\hat{Q}}} $ is an upper-triangular alternating-parity matrix:
\begin{equation}
    \left[ \underline{\underline{\hat{Q}}}\right]_{ij} = \left\{\begin{array}{cl}
         1&,0<i<j,\;j-i\; \text{is odd};  \\ 
         0&, \text{else.}
    \end{array}
    \right.
\end{equation}
\end{corollary}

Partitioning the modal coefficients into retained and truncated modes yields the following result.
\begin{lemma}
\label{lemma:high_low_Df_modal}
Let
\begin{displaymath}
    \underline{\hat{f}}^{(m)}=\left[ \underline{\hat{f}}^{(n)\top}, \underline{\tilde{\hat{f}}}^{\top}\right],
\end{displaymath}
then the filtered modal derivative satisfies
\begin{displaymath}    \overline{\underline{\underline{\hat{D}}}^{(m)}\underline{\hat{f}}^{(m)}} = \underline{\underline{\hat{D}}}^{(n)}\underline{\hat{f}}^{(n)} + 2\left[\underline{\underline{\hat{M}}}^{(n)}\right]^{-1}\underline{\underline{\hat{R}}}\,\underline{\tilde{\hat{f}}},
\end{displaymath}
where $\underline{\underline{\hat{R}}}$ couples retained and truncated modes.
\end{lemma}
\begin{proof}
    The result is directly derived from the block structure of $\underline{\underline{\hat{D}}}^{(m)}$. Detailed expressions are given in \cref{sec: AppendixC1}.
\end{proof}
\Cref{lemma:high_low_Df_modal} immediately produces the following decomposition of the projected volume term.

\begin{proposition}
\label{prop:projected_volume_term}
The projected high-order volume contribution can be decomposed into the low-order volume term plus boundary and quadrature corrections induced by the truncated modes:
\begin{equation}
\label{eq:projected_volume_term_final}
    \begin{split}        \underline{\underline{\Pi}}\,\underline{\underline{D}}^{(m)}\underline{f}^{(m)} = & \underline{\underline{D}}^{(n)}\overline{\underline{f}^{(m)}}+\frac{\tilde{f}_0}{\omega^{(n)}_0}\underline{b}_L^{(n)}+\frac{\tilde{f}_n} {\omega^{(n)}_n}\underline{b}_R^{(n)} \\
    &+ \frac{n+1}{2}\left[ \tilde{f}_n-(-1)^n\tilde{f}_0 \right]\underline{v}_n^{n}.
    \end{split}
\end{equation}
\end{proposition}
\begin{proof}
Using \cref{lemma:high_low_Df_modal}, the projected volume term reads
\begin{equation}    \underline{\underline{\Pi}}\,\underline{\underline{D}}^{(m)}\underline{f}^{(m)} = \underline{\underline{D}}^{(n)}\overline{\underline{f}^{(m)}} + 2\underline{\underline{V}}^{(n)}\left[\underline{\underline{\hat{M}}}^{(n)}\right]^{-1}\underline{\underline{\hat{R}}}\,\underline{\tilde{\hat{f}}}.
\end{equation}
and using the structure of $\underline{\underline{\hat{R}}}$ (see \cref{sec: AppendixC2}),
\begin{equation}  
\label{eq: 2Rf}
2\underline{\underline{\hat{R}}}\,\underline{\tilde{\hat{f}}} = {\underline{\underline{V}}^{(n)}}^{\top}\left[ \tilde{f}_0{\underline{b}_L^{(n)}} + \tilde{f}_n{\underline{b}_R^{(n)}} \right].
\end{equation}
one can apply \cref{lemma:VNVT_VtildeNVT} to immediately obtain
\begin{equation}
\label{eq:correction_of_derivative_terms}
    \begin{split}        2\underline{\underline{V}}^{(n)}\left[\underline{\underline{\hat{M}}}^{(n)}\right]^{-1}\underline{\underline{\hat{R}}}\,\underline{\tilde{\hat{f}}} &= \frac{\tilde{f}_0}{{\omega_0^{(n)}}}{\underline{b}_L^{(n)}} + \frac{\tilde{f}_n}{{\omega_n^{(n)}}}{\underline{b}_R^{(n)}} + \frac{n+1}{2}\left[ \tilde{f}_n - (-1)^n\tilde{f}_0 \right]{\underline{v}_n^{(n)}},
    \end{split}
\end{equation}
which completes the proof.
\end{proof}

\begin{remark}
 It should be noted that $\tilde{f}_0$ and $\tilde{f}_n$ are the contribution of nodal flux values to the boundaries of Legendre polynomials higher than $n$ and less than or equal to $m$
    \begin{equation}
        \tilde{f}_0 = \sum_{i=n+1}^m\hat{f}_iL_i({\xi_0^{(n)}}),\; \tilde{f}_n = \sum_{i=n+1}^m\hat{f}_iL_i({\xi_n^{(n)}}).
    \end{equation}
\end{remark}

\begin{remark}
    Note that if the nodal flux evaluation function, $f$, is nonlinear, then the flux vector computed from the projected solution and the projected high-order flux are different:
    \begin{equation}
        \underline{f}^{(n)} :=f\left(\underline{\underline{\Pi}}\,\underline{u}^{(m)}\right)\neq \underline{\underline{\Pi}}\,f\left( \underline{u}^{(m)} \right):=\overline{\underline{f}^{(m)}}.
    \end{equation}
    However, for linear flux, these two vectors coincide, and the expression in \cref{eq:projected_volume_term_final} can be further simplified. The property is highly attractive when computing viscous fluxes because of their linearity, which means that the projected high-order residual can be computed directly from low-order solutions.
\end{remark}

\subsection{Simplified formulation for the projected high-order time derivative}

By combining the projected surface and volume terms, the projected high-order time derivative is obtained:
\begin{proposition}
\label{prop:projected_ho_time_derivative}
    In nodal DGSEM-LGL, the low-order components of a high-order solution can be recovered exactly on low-order nodes if and only if the high-order flux vectors, i.e., $\underline{f}^{(m)}:=f(\underline{u}^{(m)})$, are given. More explicitly, the projected high-order time derivative is governed by the following equation
    \begin{equation}
    \label{eq:projected_ho_time_derivative}
    \begin{split}
    \mathcal{J}_i\overline{\underline{\dot{u}}^{(m)}} &+\frac{\left({f^{\ast}_L}^{(m)}-\overline{f_0^{(m)}}\right)}{{\omega_0^{(n)}}}{\underline{b}_L^{(n)}} + \frac{\left({f^{\ast}_R}^{(m)}-\overline{f_n^{(m)}}\right)}{{\omega_n^{(n)}}}{\underline{b}_R^{(n)}} +\underline{\underline{D}}^{(n)}\overline{\underline{f}^{(m)}}\\
    &+ \frac{n+1}{2}\left[\left({f^{\ast}_R}^{(m)}-\overline{f_n^{(m)}}\right)-(-1)^{n}\left({f^{\ast}_L}^{(m)}-\overline{f_0^{(m)}}\right)\right] {\underline{v}_n^{(n)}}=0.
    \end{split}
    \end{equation}
\end{proposition}

\begin{proof}
 Taking into account \cref{prop:projected_surface_terms} and \cref{prop:projected_volume_term} and using the properties ${f_0^{(m)}} = \overline{f_0^{(m)}}+\tilde{f}_0$ and ${f_n^{(m)}} =\overline{ {f_n^{(m)}}}+\tilde{f}_n$ complete the proof.
\end{proof}

\begin{remark}
 It should be noted that the formulation of the projected high-order time derivative in \cref{eq:projected_ho_time_derivative} mimics the continuous formulation \cref{eq: exact_low_order_evolution} term by term, except for a correction term to address the inaccuracy of the LGL quadrature.
\end{remark}

\begin{remark}
    The two-dimensional case of \cref{eq:projected_ho_time_derivative} can be easily derived using the tensor product of the one-dimensional formulation. The detailed expression is given in \cref{sec: Appendix_2d_formulation}.
\end{remark}

By subtracting the low-order time derivative, $\mathcal{J}_i\underline{\dot{u}}^{(n)}$, from \Cref{eq:projected_ho_time_derivative}, the corrective forcing term can be explicitly defined
\begin{equation}
    \mathcal{J}_i\overline{\underline{\dot{u}}^{(m)}} = \mathcal{J}_i\underline{\dot{u}}^{(n)} + \underline{c}^{(n,m)},
\end{equation}
where $\underline{c}^{(n,m)}\in\mathbb{R}^{n+1}$ is the correction in the following form:
\begin{equation}
\label{eq:correction}
\begin{split}
    \underline{c}^{(n,m)} = &- \frac{\left({f^{\ast(m)}_L}-{f^{\ast(n)}_L}-\overline{f_0^{(m)}}+f^{(n)}_0\right)}{{\omega_0^{(n)}}}{\underline{b}_L^{(n)}} - \frac{\left({f^{\ast(m)}_R}-{f^{\ast(n)}_R}-\overline{f_n^{(m)}}+f^{(n)}_n\right)}{{\omega_n^{(n)}}}{\underline{b}_R^{(n)}} \\
    &-  \underline{\underline{D}}^{(n)}\left( \overline{\underline{f}^{(m)}}-\underline{f}^{(n)} \right)\\
    &- \frac{n+1}{2}\left[\left({f^{\ast(m)}_R}-\overline{f_n^{(m)}}\right)-(-1)^{n}\left({f^{\ast(m)}_L}-\overline{f_0^{(m)}}\right)\right] {\underline{v}_n^{(n)}}.
\end{split}
\end{equation}

\subsection{Viscous flux}
When viscous effects are considered, the viscous flux, $g$, needs to be added to the inviscid flux. Let us explain the procedure by taking the viscous flux of Burgers' equation, as an example 
\begin{equation}
    g = \frac{\nu}{\mathcal{J}_i} u_{\xi},
\end{equation}
and follow the BR1 scheme \cite{bassi_high-order_1997}
\begin{equation}
    \underline{u_{\xi}}^{(n)} = \frac{\left(u^{\ast(n)}_L - u_0\right)}{\omega_0^{(n)}}\underline{b}_L^{(n)} + \frac{\left(u^{\ast(n)}_R - u_n^{(n)}\right)}{\omega_n^{(n)}}\underline{b}_R^{(n)}+\underline{\underline{D}}^{(n)}\,\underline{u}^{(n)},
\end{equation}
where $u^{\ast(n)}_L$ and $u^{\ast(n)}_R$ are the numerical fluxes using the central scheme. Analogously to the projected high-order inviscid flux term, the corresponding viscous flux term can be written as
\begin{equation}
\label{eq:projected_D^m_g^m}
\begin{split}
    \underline{\underline{\Pi}}\,\underline{\underline{D}}^{(m)}\underline{g}^{(m)}
    = &\frac{\left(g^{\ast(m)}_{L} - \overline{g_{0}^{(m)}}\right)}{\omega_0^{(n)}}\underline{b}_L^{(n)} + \frac{\left(g^{\ast(m)}_R - \overline{g_n^{(m)}}\right)}{\omega_n^{(n)}}\underline{b}_R^{(n)}+\underline{\underline{D}}^{(n)}\overline{\underline{g}^{(m)}}\\
    &+\frac{n+1}{2}\left[\left({g^{\ast(m)}_R}-\overline{g_n^{(m)}}\right)-(-1)^{n}\left({g^{\ast(m)}_L}-\overline{g_0^{(m)}}\right)\right] {\underline{v}_n^{(n)}}
\end{split}
\end{equation}
where $g^{\ast(m)}_L$ and $g^{\ast(m)}_R$ are the numerical fluxes computed from the high-order viscous flux $\underline{g}^{(m)}$, which are calculated using a central scheme
\begin{equation}
\label{eq:central_g}
g^{\ast(m)}_L =\frac{1}{2}\left(g_m^{(m)-}+g_0^{(m)}\right),\,g^{\ast(m)}_R =\frac{1}{2}\left(g_m^{(m)}+g_0^{(m)+}\right).
\end{equation}
For the projected high-order viscous flux, $\overline{\underline{g}^{(m)}}$, it can be computed as follows:
\begin{multline}
\label{eq:projected_g^m}
    \overline{\underline{g}^{(m)}} = \frac{\nu}{\mathcal{J}_i}\Biggl\{ \frac{\left(u^{\ast(m)}_L - u_0^{(n)}\right)}{\omega_0^{(n)}}\underline{b}_L + \frac{\left(u^{\ast(m)}_R - u_n^{(n)}\right)}{\omega_n^{(n)}}\underline{b}_R+\underline{\underline{D}}^{(n)}\,\underline{u}^{(n)} \\
    + \frac{n+1}{2}\left[\left({u^{\ast(m)}_R}-{u_n^{(n)}}\right)-(-1)^{n}\left({u^{\ast(m)}_L}-{u_0^{(n)}}\right)\right] {\underline{v}_n^{(n)}}\Biggl\}.
\end{multline}

\section{Compact reconstruction of high-order solutions}
\label{sec:reconstruction}
Assuming that the target element of reconstruction is the mesh element $\Omega_i$, we choose a compact reconstruction stencil that contains only the direct neighboring elements
\begin{equation}
    \mathcal{S}_i = \{ \Omega_{i-1},\Omega_{i},\Omega_{i+1}  \},
\end{equation}
where three elements are involved. We follow the ideas in \cite{dumbser_arbitrary_2007,dumbser_unified_2008}, to define a reconstruction based on $L^2$-projection, and consequently a generalized conservation property must be satisfied for all degrees of freedom up to degree $n$. In the physical coordinate system, we have the following condition
\begin{multline}
\label{eq: conditions_physical_space}
    \int_{\Omega_l}\left[w^{(m)}\right]_i\left(\xi(\Omega_i,x)\right)L_kdx = \int_{\Omega_l}\left[u^{(n)}\right]_l\left(\xi(\Omega_i,x)\right)L_kdx,\\
    \forall\Omega_l\in\mathcal{S}_{i},\; k\in [ 0,1,\dots,n],
\end{multline}
where $[u]_l$ denotes the solution in $\Omega_l$ and $w^{(m)}$ denotes the reconstructed high-order solution. As before, it is more convenient to implement the reconstruction in the reference element by applying the mapping in \cref{eq: geo_mapping} to all elements in the stencil, and the transformed elements are indicated by $\widetilde{\Omega}_l$. It should be noted that for all elements $\Omega_l\in S_i$ the mapping with respect to $\Omega_i$ is applied. In particular, for the target element $\Omega_i$, its transformed element is just the reference element, e.g., $\widetilde{\Omega}_i=E$. The stencil transformed following this procedure  is denoted as $\widetilde{\mathcal{S}}_i =\{ \widetilde{\Omega}_{i-1},\widetilde{\Omega}_i,\widetilde{\Omega}_{i+1}  \}$.

After the transformation, and considering that the solution can be expanded using Legendre polynomials, \cref{eq: conditions_physical_space} becomes
\begin{multline}
    \label{eq: conditions_xi}
    \mathcal{J}_i\left[ \hat{w}^{(m)}_p\right]_i\int_{\widetilde{\Omega}_l}L_kL_p^{S}d\xi=\mathcal{J}_i\left[ \hat{u}^{(n)}_q\right]_l\int_{\widetilde{\Omega}_l}L_kL_qd\xi,\;0,1,\dots,n],\forall\widetilde{\Omega}_l\in\widetilde{\mathcal{S}}_{i},\\ \forall k \in [0,1,\dots,n],
\end{multline}
and the Jacobian appearing on both sides can be canceled out. To compute the integrals of the above equations, we transform $\widetilde{\Omega}_l$ into the reference element $E$ by another mapping of the $\xi$-coordinates to the $\tilde{\xi}$-coordinates. This second mapping and its inverse are indicated by $\tilde{\xi}=\tilde{\xi}\left(\widetilde{\Omega}_l,\xi \right)$ and $\xi=\xi\left(\widetilde{\Omega}_l,\tilde{\xi} \right)$, respectively. Thus, \cref{eq: conditions_xi} becomes after the second transformation
\begin{multline}
     \left[ \hat{w}_p^{(m)}\right]_i\int_EL_k\left(\xi(\widetilde{\Omega}_l,\tilde{\xi})\right)L_p^{S}\left(\xi(\widetilde{\Omega}_l,\tilde{\xi})\right)d\tilde{\xi} \\= \left[ \hat{u}_q^{(n)}\right]_l\int_EL_k\left(\xi(\widetilde{\Omega}_l,\tilde{\xi})\right)L_q\left(\xi(\widetilde{\Omega}_l,\tilde{\xi})\right)d\tilde{\xi},
     \forall\widetilde{\Omega}_l\in\widetilde{\mathcal{S}}_{i},\;\forall k \in [0,1,\dots,n],
\end{multline}
as the Jacobian of the second transform also cancels out. Note that in the target element $\Omega_i$, $L_p^S = L_p$, and the above equation is reduced to:
\begin{equation}
\label{eq: conditions_n}
    \left[\hat{w}_k^{(m)}\right]_i = \left[\hat{u}_k^{(n)}\right]_i,\; \forall k \in [0,1,\dots,n].
\end{equation}
This last equation states that the first $n+1$ modal coefficients remain unchanged and guaranty generalized conservation up to the $n^{th}$-order. Taking into account orthogonality $\int_EL_kL_qd\tilde{\xi} = \delta_{kq}\hat{\omega}_k$, for all other elements of the stencil, we have
\begin{equation}
\label{eq: w_hat = u_hat}
\left[\hat{w}_p^{(m)}\right]_i\int_EL_k\left(\xi(\widetilde{\Omega}_l,\tilde{\xi})\right)L_p^{S}\left(\xi(\widetilde{\Omega}_l,\tilde{\xi})\right)d\tilde{\xi} =  \hat{\omega}_k\left[\hat{u}_k^{(n)}\right]_l, \forall\widetilde{\Omega}_l\in\widetilde{\mathcal{S}}_{i}\backslash\widetilde{\Omega}_i.
\end{equation}
In our case, there are three elements in the stencil, which provide $3(n+1)$ equations in \cref{eq: conditions_xi} with $n+1$ additional constraints in \cref{eq: conditions_n}. A classical approach is to use a constrained least-squares technique \cite{dumbser_arbitrary_2007,dumbser_unified_2008}. It leads to a linear system size of $4(n+1)\times4(n+1)$, which is computationally expensive. 

We emphasize that since the first $n+1$ modal coefficients remain unchanged in \cref{eq: conditions_n}, the reconstruction task is equivalent to computing the high-order components, $\underline{\widetilde{\hat{w}}}=[\hat{w}_{n+1},\dots,\hat{w}_{m}]^{\top}\in \mathbb{R}^r$ (the index denoting the target element is dropped for clarity), by imposing another $2(n+1)$ condition from the neighboring elements. To derive the new reconstruction algorithm, we keep the modal form on the left side of \cref{eq: w_hat = u_hat} and write the right side in nodal form:
\begin{equation}
\label{eq: w = u_hat}   \int_{\Omega_l}w^{(m)}\left(\xi(\widetilde{\Omega}_l,\tilde{\xi})\right)L_k\left(\xi(\widetilde{\Omega}_l,\tilde{\xi})\right) = \hat{\omega}_k\left[\hat{u}_k^{(n)}\right]_l, \forall\widetilde{\Omega}_l\in\widetilde{\mathcal{S}}_{i}\backslash\widetilde{\Omega}_i.
\end{equation}
\subsection*{Discrete algorithm}The use of LGL-quadrature of required accuracy leads to the discrete form. Taking the left-hand side element as an example, its discrete form is:
\begin{equation}    \left[\underline{\underline{N}}^{(n)}\right]^{-1}\overline{\underline{\underline{V}}^{(n^{\prime})}}^{\top}\underline{\underline{M}}^{(n^{\prime})}\left[\underline{w}^{(n^{\prime})}\right]_{i-1} = \left[ {\underline{\hat{u}}^{(n)}}\right]_{i-1},
\end{equation}
which is equivalent to:
\begin{equation}
\label{eq:Piwu}
\underline{\underline{\Pi}}^{n}_{n^{\prime}}\left[\underline{w}^{(n^{\prime})}\right]_{i-1} = \left[{\underline{u}^{(n)}}\right]_{i-1},
\end{equation}
where $\underline{\underline{\Pi}}^{n}_{n^{\prime}}:=\underline{\underline{V}}^{(n)}\overline{\left[ \underline{\underline{V}}^{(n^{\prime})} \right]^{-1}}$ is the projection operator from $\mathbb{R}^{n^{\prime}}$ to $\mathbb{R}^{n}$ and $n^{\prime}$ is the order of LGL quadrature needed to compute the left hand in \cref{eq: w = u_hat}. Noting that $\left[\underline{w}^{(n^{\prime})}\right]_{i-1}$ is the extension of $w$ to the left element $\Omega_{i-1}$ evaluated on the LGL quadrature nodes of $n^{\prime}$-order, it can be written as
\begin{equation}
\label{eq:extension_left}
    \left[\underline{w}^{(n^{\prime})}\right]_{i-1} = \left[\overline{\underline{\underline{V}}^{(n^{\prime})}}\right]_{i-1}\underline{\hat{w}}^{(n)}+\left[\widetilde{\underline{\underline{V}}^{(n^{\prime})}}\right]_{i-1}\underline{\widetilde{\hat{w}}},
\end{equation}
where $\left[\overline{\underline{\underline{V}}^{(n^{\prime})}}\right]_{i-1}$ and $\left[\widetilde{\underline{\underline{V}}^{(n^{\prime})}}\right]_{i-1}$ are the low- and high-order parts of the Legendre polynomials extended on $n^{\prime}+1$ LGL nodes in $\Omega_{i-1}$. Inserting \cref{eq:extension_left} into \cref{eq:Piwu} and using conditions \cref{eq: conditions_n} lead to
\begin{multline}
\underline{\underline{\Pi}}^{n}_{n^{\prime}}\left[\widetilde{\underline{\underline{V}}^{(n^{\prime})}}\right]_{i-1}\underline{\widetilde{\hat{w}}} = \left[{\underline{u}^{(n)}}\right]_{i-1} - \underline{\underline{\Pi}}^{n}_{n^{\prime}}\left[\overline{\underline{\underline{V}}^{(n^{\prime})}}\right]_{i-1}\underline{\hat{u}}^{(n)} \\= \left[{\underline{u}^{(n)}}\right]_{i-1} -\underline{\underline{\Pi}}^{n}_{n^{\prime}}\left[\overline{\underline{\underline{V}}^{(n^{\prime})}}\right]_{i-1}\left[\underline{\underline{V}}^{(n)}\right]^{-1}\underline{u}^{(n)}.
\end{multline} 
The similar approach can be applied to the right element $\Omega_{i+1}$ and assembling them together produces a least-squares (LS) problem:
\begin{equation}
    \underline{\underline{\Phi}}\,\underline{\widetilde{\hat{w}}} = \underline{\mu},
\end{equation}
where $\underline{\underline{\Phi}}\in \mathbb{R}^{(2n+2)\times r}$ and $\underline{\mu} \in \mathbb{R}^{(2n+2)}$ are composed of contributions from the left and right elements:
\begin{equation*}
    \underline{\underline{\Phi}}:=\left[ \begin{array}{l}
          \underline{\underline{\Pi}}^{n}_{n^{\prime}}\left[\widetilde{\underline{\underline{V}}^{(n^{\prime})}}\right]_{i-1} \\
         \underline{\underline{\Pi}}^{n}_{n^{\prime}}\left[\widetilde{\underline{\underline{V}}^{(n^{\prime})}}\right]_{i+1}
    \end{array} \right],\;\underline{\mu} = \left[ \begin{array}{l}
          \left[{\underline{u}^{(n)}}\right]_{i-1} -\underline{\underline{\Pi}}^{n}_{n^{\prime}}\left[\overline{\underline{\underline{V}}^{(n^{\prime})}}\right]_{i-1}\left[\underline{\underline{V}}^{(n)}\right]^{-1}\underline{u}^{(n)} \\
         \left[{\underline{u}^{(n)}}\right]_{i+1} -\underline{\underline{\Pi}}^{n}_{n^{\prime}}\left[\overline{\underline{\underline{V}}^{(n^{\prime})}}\right]_{i+1}\left[\underline{\underline{V}}^{(n)}\right]^{-1}\underline{u}^{(n)}
    \end{array} \right].
\end{equation*}
The LS problem can be solved by:
\begin{equation}
    \underline{\widetilde{\hat{w}}} = \underline{\underline{\Phi}}^{+}\underline{\mu},
\end{equation}
and then compute the nodal values of the high-order components:
\begin{equation}
    \underline{\tilde{w}}^{(m)} = \widetilde{\underline{\underline{V}}^{(m)}} \underline{\underline{\Phi}}^{+}\underline{\mu}.
\end{equation}
Finally, the reconstructed high-order nodal solution is as follows:
\begin{equation}
    \underline{w}^{(m)}  = \underline{\underline{\Pi}}^{\ast}\underline{u}^{(n)}+\underline{\tilde{w}}^{(m)}.
\end{equation}

When implementing the reconstruction step, the products of the matrices, i.e., $\widetilde{\underline{\underline{V}}^{(m)}} \underline{\underline{\Phi}}^{+}$, $\underline{\underline{\Pi}}^{n}_{n^{\prime}}\left[\overline{\underline{\underline{V}}^{(n^{\prime})}}\right]_{i+1}\left[\underline{\underline{V}}^{(n)}\right]^{-1}$, and $\underline{\underline{\Pi}}^{n}_{n^{\prime}}\left[\overline{\underline{\underline{V}}^{(n^{\prime})}}\right]_{i-1}\left[\underline{\underline{V}}^{(n)}\right]^{-1}$, are calculated and stored before simulation to improve efficiency.

\subsection*{Error estimates}
By defining the local transformed broken polynomial space attached to $\Omega_i$:
\begin{equation}
    \mathcal{V}^{(n)}_{\widetilde{\mathcal{S}}_i}:=\{ v\in L^2(\widetilde{\omega}_i)\,\big\vert\,v\vert_{\widetilde{\Omega}_j}\in\mathbb{P}_n(\widetilde{\Omega}_j),\forall\widetilde{\Omega}_j\in\widetilde{\mathcal{S}}_i\},
\end{equation}
the local reconstruction operator can be denoted by:
\begin{equation}
    \widetilde{\mathcal{R}}_i: \mathcal{V}^{(n)}_{\widetilde{\mathcal{S}}_i} \to \mathbb{P}_m(E),
\end{equation}
where $\widetilde{\omega}_i:= \bigcup_{\widetilde{\Omega}_j\in \widetilde{S}_i}\widetilde{\Omega}_j$. Then, the following lemma summarizes the local consistency and stability properties of the reconstruction operator on the transformed stencil.
\begin{lemma}
\label{lemma:projector_property_local}
Let $\widetilde{\mathcal{R}}_i$ be the $m$-exact local reconstruction operator defined on the transformed stencil $\widetilde{\mathcal{S}}_i$ attached to the element $\Omega_i$. Assume that the reconstruction procedure is well-posed and uniformly bounded. Then $\widetilde{\mathcal{R}}_i$ satisfies the following properties:
\begin{enumerate}
    \item \textbf{Consistency}:
    \begin{multline}
        \bigg\Vert v^{(m)}\big\vert_E
        - \widetilde{\mathcal{R}}_i\left(v^{(n)}\right)
        \bigg\Vert_{L^2(E)}
        \leq
        \widetilde{C}_{\mathcal{R}_c}
        \vert v \vert_{H^{m+1}(\widetilde{\omega}_i)}, \\
        v^{(m)}=\widetilde{\Pi}^{(m)}v,\quad
        v^{(n)}=\widetilde{\Pi}^{(n)}v,\quad
        \forall v\in H^{m+1}(\widetilde{\omega}_i).
    \end{multline}

    \item \textbf{Stability}:
    \begin{multline}
        \bigg\Vert
        \widetilde{\mathcal{R}}_i\left(v^{(n)}\right)
        -
        \widetilde{\mathcal{R}}_i\left(w^{(n)}\right)
        \bigg\Vert_{L^2(E)}
        \leq
        \widetilde{C}_{\mathcal{R}_s}
        \Vert v^{(n)}-w^{(n)} \Vert_{L^2(\widetilde{\omega}_i)}, \\
        \forall v^{(n)},w^{(n)}
        \in
        \mathcal{V}^{(n)}_{\widetilde{\mathcal{S}}_i}.
    \end{multline}
\end{enumerate}
Here $C_{\mathcal{R}_c}^{\mathrm{loc}}$ and $C_{\mathcal{R}_s}^{\mathrm{loc}}$ are constants independent of $h$; $\widetilde{\omega}_i := \bigcup_{\widetilde{\Omega}_j\in \widetilde{\mathcal{S}}_i}\widetilde{\Omega}_j$; and
$\widetilde{\Pi}^{(p)}:L^2(\widetilde{\omega}_i)\to \mathcal{V}^{(p)}_{\widetilde{\mathcal{S}}_i}$, $p=n,m$, denotes the piecewise $L^2$-projection operator onto polynomial spaces of degree $p$.
\end{lemma}

Using \cref{lemma:projector_property_local}, the extension to the physical and global space is quite straightforward.

\begin{lemma}[Physical-space reconstruction estimates]
\label{lemma:projector_property_global}
Let $\mathcal{T}_h=\{\Omega_i\}_{i=1}^{K}$ be a shape-regular mesh with mesh size $h$, and let $\mathcal{S}_i$ denote the physical reconstruction stencil associated with the element $\Omega_i$. Define the stencil domain
\begin{equation}
    \omega_i:=\bigcup_{\Omega_j\in\mathcal{S}_i}\Omega_j.
\end{equation}

Assume that the transformed reconstruction operator $\widetilde{\mathcal{R}}_i$ satisfies the consistency and stability properties stated in \cref{lemma:projector_property_local}. Let $\mathcal{R}_i:\mathcal{V}_h^{(n)}\big\vert_{\omega_i}\to \mathbb{P}_m(\Omega_i)$ denote the corresponding physical reconstruction operator obtained through the affine mapping between $\omega_i$ and $\widetilde{\omega}_i$.

Then there exist constants $C_{\mathcal{R}_c}^{\mathrm{loc}},\;C_{\mathcal{R}_s}^{\mathrm{loc}}>0$, independent of $h$, such that:
\begin{enumerate}
    \item \textbf{Local physical consistency:} 
    \begin{equation}
        \left\|v^{(m)}\big\vert_{\Omega_i}-\mathcal{R}_i\left(v^{(n)}\right)\right\|_{L^2(\Omega_i)}
        \le
        C_{\mathcal{R}_c}^{\mathrm{loc}}h^{m+1}\vert v \vert_{H^{m+1}(\omega_i)},\;\forall v\in H^{m+1}(\omega_i)
    \end{equation}
    where
    \[v^{(m)}=\left(\Pi_h^{(m)}v\right)\Big\vert_{\omega_i},\qquad v^{(n)}=\left(\Pi_h^{(n)}v\right)\Big\vert_{\omega_i}.\]

    \item \textbf{Local physical stability:} 
    \begin{multline}
        \left\|\mathcal{R}_i\left(v^{(n)}\right)-\mathcal{R}_i\left(w^{(n)}\right)\right\|_{L^2(\Omega_i)}
        \le
        C_{\mathcal{R}_s}^{\mathrm{loc}}
        \left\|v^{(n)}-w^{(n)}\right\|_{L^2(\omega_i)},\\ 
        \forall v^{(n)},w^{(n)}\in\mathcal{V}_h^{(n)}\big\vert_{\omega_i}.
    \end{multline}
\end{enumerate}

Moreover, assume that the reconstruction stencils possess uniformly bounded overlap. Define the global reconstruction operator $\mathcal{R}_h$ by
\begin{equation}
    \left(\mathcal{R}_h v_h^{(n)}\right)\big\vert_{\Omega_i}
    =
    \mathcal{R}_i\left(v_h^{(n)}\big\vert_{\omega_i}\right).
\end{equation}

Then the following global estimates hold:

\begin{enumerate}
    \setcounter{enumi}{2}
    \item \textbf{Global consistency:}
    \begin{equation}
        \left\|v_h^{(m)}-\mathcal{R}_h\left(v_h^{(n)}\right)\right\|_{L^2(\Omega)}
        \le
        C_{\mathcal{R}_c}h^{m+1}\vert v\vert_{H^{m+1}(\Omega)},\;\forall v\in H^{m+1}(\Omega).
    \end{equation}
    where
    \[v^{(m)}_h=\Pi_h^{(m)}v,\qquad v^{(n)}_h=\Pi_h^{(n)}v.\]

    \item \textbf{Global stability:}
    \begin{multline}
    \label{eq:global_stability}
        \left\|\mathcal{R}_h\left(v_h^{(n)}\right)-\mathcal{R}_h\left(w_h^{(n)}\right)
        \right\|_{L^2(\Omega)}
        \le
        C_{\mathcal{R}_s}\left\|v_h^{(n)}-w_h^{(n)}\right\|_{L^2(\Omega)},\\ 
        \forall v_h^{(n)},w_h^{(n)}\in \mathcal{V}_h^{(n)}.
    \end{multline}
\end{enumerate}

The constants $C_{\mathcal{R}_c}$ and $C_{\mathcal{R}_s}$ are independent of $h$.
\end{lemma}

\begin{proof}
We first derive the local physical estimates from the transformed-stencil estimates. Consider the affine mapping \cref{eq: geo_mapping}, and denote by $\breve{v}$ the transformation of $v$ onto the transformed stencil $\widetilde{\omega}_i$.

By Lemma~\ref{lemma:projector_property_local},
\begin{equation}
    \left\|
    \breve{v}^{(m)}|_E
    -
    \widetilde{\mathcal{R}}_i\left(\breve{v}^{(n)}\right)
    \right\|_{L^2(E)}
    \le
    \widetilde{C}_{\mathcal{R}_c}
    |\breve{v}|_{H^{m+1}(\widetilde{\omega}_i)}.
\end{equation}

Using standard affine-scaling relations for Sobolev semi-norms,
\begin{equation}
    |\breve{v}|_{H^{m+1}(\widetilde{\omega}_i)}
    \lesssim h^{m+1/2}
    |v|_{H^{m+1}(\omega_i)},
\end{equation}
while the $L^2$-norm scales as
\begin{equation}
    \|e\|_{L^2(\Omega_i)}
    \lesssim h^{1/2}
    \|\breve{e}\|_{L^2(E)}.
\end{equation}

Combining the above estimates yields
\begin{equation}
    \left\|v^{(m)}\big\vert_{\Omega_i}-\mathcal{R}_i\left(v^{(n)}\right)\right\|_{L^2(\Omega_i)}
    \le
    C_{\mathcal{R}_c}^{\mathrm{loc}}h^{m+1}|v|_{H^{m+1}(\omega_i)}.
\end{equation}

Similarly, the local transformed stability estimate together with the equivalence of norms under affine mappings implies
\begin{equation}
    \left\|\mathcal{R}_i\left(v^{(n)}\right)-\mathcal{R}_i\left(w^{(n)}\right)\right\|_{L^2(\Omega_i)}
    \le
    C_{\mathcal{R}_s}^{\mathrm{loc}}\left\|v^{(n)}-w^{(n)}\right\|_{L^2(\omega_i)}.
\end{equation}

We now derive the global estimates. Summing the local consistency estimate over all elements gives
\begin{equation}
    \sum_{i=1}^{K}
    \left\|
    v_h^{(m)}\big\vert_{\Omega_i}
    -
    \mathcal{R}_i\left(v_h^{(n)}\big\vert_{\omega_i}\right)
    \right\|_{L^2(\Omega_i)}^2
    \lesssim h^{2m+2}
    \sum_{i=1}^{K}
    |v|_{H^{m+1}(\omega_i)}^2.
\end{equation}

Since the stencil overlap is uniformly bounded, there exists a constant $C_{\mathrm{ov}}$ such that
\begin{equation}
    \sum_{i=1}^{K}
    |v|_{H^{m+1}(\omega_i)}^2
    \leq C_{\mathrm{ov}}
    |v|_{H^{m+1}(\Omega)}^2.
\end{equation}

Therefore,
\begin{equation}
    \left\|
    v_h^{(m)}
    -
    \mathcal{R}_h\left(v_h^{(n)}\right)
    \right\|_{L^2(\Omega)}
    \le
    C_{\mathcal{R}_c}
    h^{m+1}
    |v|_{H^{m+1}(\Omega)}.
\end{equation}

The proof of the global stability estimate follows analogously by summing the local stability estimate over all elements and applying the bounded-overlap property of the stencils.
\end{proof}

\section{Corrected \texorpdfstring{\(\mathbb{P}_n\mathbb{P}_m\)} scheme and error estimates} Coupling the simplified formulation for high-order time derivatives in \cref{eq:projected_ho_time_derivative} with any $m^{th}$-order reconstructor leads to the corrected $\mathbb{P}_n\mathbb{P}_m$ scheme ($c\mathbb{P}_n\mathbb{P}_m$). In this paper, the local  $L^2$-projection-based reconstructor in \cref{sec:reconstruction} is adopted. For simplicity we will give here the proof only for the semi-discrete $c\mathbb{P}_n\mathbb{P}_m$ scheme for a scalar conservation law:
\begin{equation}
\label{eq:1d_conservation_law_scalar}
    \partial_t u + \partial_xf(u)=0
\end{equation}
in the domain $\Omega$ with periodic boundary conditions. 

We emphasize that the objective of this section is to provide a streamlined convergence proof for the $c\mathbb{P}_n\mathbb{P}_m$ scheme, relying on the assumed convergence properties of the standard DGSEM-LGL method. A complete proof based on the detailed structure of the DGSEM-LGL operator, including all derivative contributions, is therefore not pursued here.

Assuming that $\mathcal{T}_h=\{\Omega_i\}_{i=1}^{K}$ is a tessellation of the domain $\Omega$ into non-overlapping $K$ elements with element size of $h$, then the local semi-discrete $c\mathbb{P}_n\mathbb{P}_m$ operator in the reference element reads:
\begin{equation}
\label{eq:cPnPm}
\begin{split}
    \mathcal{J}_i\frac{d}{dt}\underline{u}^{(n)}_{\mathcal{R},i} = &-\frac{\delta f^{\ast(n,m)}_{L,i}}{{\omega_0^{(n)}}}{\underline{b}_L^{(n)}} - \frac{\delta f^{\ast(n,m)}_{R,i}}{{\omega_n^{(n)}}}{\underline{b}_R^{(n)}} -\underline{\underline{D}}^{(n)}\overline{\underline{f}^{(m)}}\left(  \underline{u}^{\prime(m)}_{\mathcal{R},i}\right)
    \\
    &- \frac{n+1}{2}\left(\delta f^{\ast(n,m)}_{R,i}-(-1)^{n}\delta f^{\ast(n,m)}_{L,i}\right) {\underline{v}_n^{(n)}}.
\end{split}
\end{equation}
where $\underline{u}^{\prime(m)}_{\mathcal{R},i} := \underline{\underline{\mathcal{\widetilde{\mathcal{R}}}_i}}\left( \underline{u}^{(n)}_{\mathcal{R},i-1},\underline{u}^{(n)}_{\mathcal{R},i},\underline{u}^{(n)}_{\mathcal{R},i+1} \right)$ is the locally reconstructed high-order solution; $\underline{\underline{\mathcal{\widetilde{\mathcal{R}}}_i}} $ is the discrete reconstruction operator; and
\begin{align*}
    \delta f^{\ast(n,m)}_{L,i} :&= {f^{\ast}_L}^{(m)}\left( \underline{u}^{\prime(m)}_{\mathcal{R},i-1}, \underline{u}^{\prime(m)}_{\mathcal{R},i}\right)-\overline{f_0^{(m)}}\left( \underline{u}^{\prime(m)}_{\mathcal{R},i}\right)  \\
    \delta f^{\ast(n,m)}_{R,i}:&= {f^{\ast}_R}^{(m)}\left( \underline{u}^{\prime(m)}_{\mathcal{R},i}, \underline{u}^{\prime(m)}_{\mathcal{R},i+1}\right)-\overline{f_n^{(m)}}\left( \underline{u}^{\prime(m)}_{\mathcal{R},i}\right),
\end{align*}
and $\underline{u}^{(n)}_{\mathcal{R},i} \in \mathbb{P}_n$ denotes the solution of the $c\mathbb{P}_n\mathbb{P}_m$ scheme. Similarly, the two-dimensional version can be obtained by coupling \cref{eq:2d_projected_ho_time_derivative} with the corresponding reconstruction.

\subsection*{Projection error}Assuming $u(x,t)\in L^{\infty}(0,T;H^{m+1}(\Omega))$ denotes the global smooth exact solution of \cref{eq:1d_conservation_law_scalar}, then the exact numerical solution of $m^{th}$-order is defined as the broken $L^2$-projection of $u$. Using approximation theory\cite{doi:10.1137/1.9780898719208}, we introduce the following lemma.
\begin{lemma}
\label{lemma:projection_error}
Let $\Pi^{(m)}_h:L^2(\Omega)\to\mathcal{V}^{(m)}_h$
be the global piecewise $L^2$-projection operator up to order $m$.
Assume that $u\in H^{m+1}(\Omega)$. Then the projection error satisfies
\begin{equation}
    \left\| \epsilon_{\Pi}\right\|
    =
    \left\| u-u^{(m)}_{h,\ast}\right\|
    \leq
    C_{\Pi}h^{m+1}\vert u\vert_{H^{m+1}(\Omega)},
\end{equation}
where
\[
u^{(m)}_{h,\ast}:=\Pi^{(m)}_h u
\]
is the projected exact solution of order $m$,
$\Vert\cdot\Vert=\Vert\cdot\Vert_{L^2(\Omega)}$,
and $C_{\Pi}>0$ is independent of $h$.
\end{lemma}
\begin{remark}
    In this paper, the projected exact solution
    $u^{(m)}_{h,\ast}=\Pi_h^{(m)}u$
    is referred to as the exact numerical solution of order $m$.
\end{remark}

\subsection*{Evolution error of DGSEM-LGL}For the sake of proof, the following lemma on the convergence of DGSEM-LGL is introduced later:
\begin{lemma}[Convergence of DGSEM-LGL]
\label{lem:DGSEM_convergence}
Assume that the exact solution satisfies
$u\in L^\infty(0,T;H^{m+1}(\Omega))$.
Let $\mathcal{P}^{(m)}_h:\mathcal{V}^{(m)}_h\to \mathcal{V}^{(m)}_h$ be the global $m^{th}$-order semi-discrete DGSEM-LGL operator (element-wise formulation is given in \cref{eq:DGSEM-LGL_0}) defined in $\mathcal{T}_h$ and $u_h^{(m)}\in\mathcal V_h^{(m)}$ denote the DGSEM-LGL solution of polynomial degree $m$:
\begin{equation}
    \frac{d}{dt}u_h^{(m)} = \mathcal{P}^{(m)}_h\left( u^{(m)}_h \right),
\end{equation}
and assume that the truncation error of $\mathcal{P}^{(m)}_h$:
\begin{equation}
    \tau_h^{(m)} := \frac{d}{dt}u_{h,\ast}^{(m)} - \mathcal{P}^{(m)}_h\left(u_{h,\ast}^{(m)}\right)
\end{equation}
is bounded by:
\begin{equation}
    \left\| \tau_h^{(m)} \right\|\leq C^{(m)}_{\tau}h^{m+1},
\end{equation}
where $C^{(m)}_{\tau}$ is a constant independent of $h$.
Define the evolution error by
\[
\epsilon_h^{(m)}
:=
u_h^{(m)}-u_{h,*}^{(m)},
\]
then the evolution error satisfies
\begin{equation}
\label{eq:convergence_rate}
\|\epsilon_h^{(m)}\|
\le
C^{(m)} h^{m+1},
\end{equation}
where $C^{(m)}>0$ is independent of $h$. 
\end{lemma}

\begin{figure}[ht]
\centering
\begin{tikzpicture}[node distance=2.6cm, every node/.style={font=\small}]

% ---------------- Nodes (top level m) ----------------
\node (um0) {$u^{(m)}_{h,\mathcal{R}}(\cdot,0)$};
\node (Run0) [right of=um0] {$\mathcal{R}_h\left(u^{(n)}_{h,\mathcal{R}}(\cdot,0)\right)$};
\node (umdt) [right of=Run0, xshift=1.6cm] {$u^{(m)}_{h,\mathcal{R}}(\cdot,\Delta t)$};
\node (Rumdt) [right of=umdt] {$\mathcal{R}_h\left(u^{(n)}_h(\cdot,\Delta t)\right)$};

\node (u_exact) [above of=um0,yshift=-1cm] {$u_0\in H^{m+1}(\Omega)$};

% ---------------- Nodes (bottom level n) ----------------
\node (un0) [below of=um0, yshift=-1.2cm] {$u^{(n)}_{h,\mathcal{R}}(\cdot,0)$};
\node (unRdt) [below of=umdt, yshift=-1.2cm] {$u^{(n)}_{h,\mathcal{R}}(\cdot,\Delta t)$};

% ---------------- Top arrows ----------------
\draw[->, thick,dashed] (Run0) -- node[above] {$\mathcal{P}_h^{(m)}$} (umdt);

% ---------------- Bottom evolution ----------------
\draw[->, thick] (un0) -- node[above, yshift=0.2cm] {$\mathcal{P}^{(n,m)}_{h,\mathcal{R}}$} (unRdt);

% ---------------- Vertical projections ----------------
\draw[->, thick,dash dot] (um0) -- node[left] {$\Pi_h$} (un0);
\draw[->, thick,dash dot] (umdt) -- node[right] {$\Pi_h$} (unRdt);
\draw[->, thick,dash dot] (u_exact) -- node[left] {$\Pi^{(m)}_h$} (um0);

% ---------------- Cross operator ----------------
%\draw[->, thick] (Run0) to[bend left=20] node[left] {} (unRdt);
\draw[->, thick] (Run0) to[bend left=-45] (unRdt);

% ---------------- Reverse reconstruction ----------------
\draw[->, thick] (un0) to[bend right=20] node[left] {$\mathcal{R}_h$} (Run0);
\draw[->, thick] (unRdt) to[bend right=20] node[left] {$\mathcal{R}_h$} (Rumdt);

%\node[right=of Rumdt] {$\dots$};
\node[right=of unRdt] {$\dots$};

\end{tikzpicture}
\label{fig:schematic_diagram}
\caption{Schematic diagram of auxiliary evolution, $u^{(m)}_{h,\mathcal{R}}$, the real $c\mathbb{P}_n\mathbb{P}_m$ evolution, $u^{(n)}_{h,\mathcal{R}}$ and the reconstructed solution, $\mathcal{R}_h\left( u^{(n)}_{h,\mathcal{R}}\right)$. The dash-dot line denotes the pre-processing. The solid and dash lines denotes the process of auxiliary and $c\mathbb{P}_n\mathbb{P}_m$ evolution, respectively.}
\end{figure}
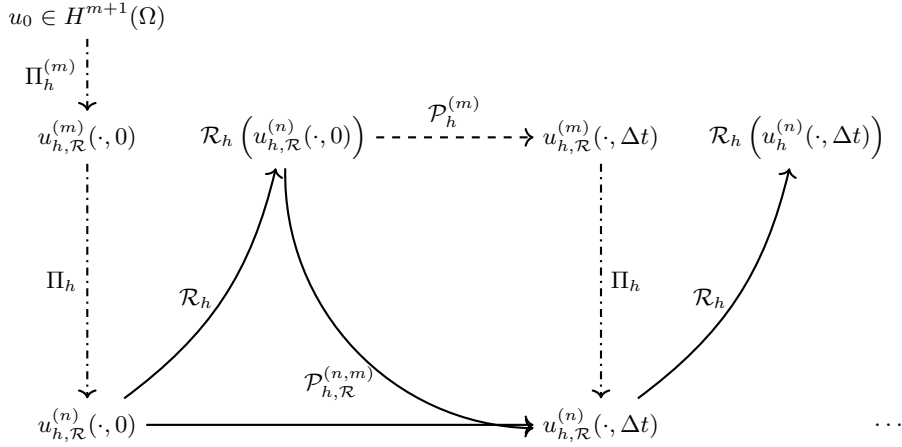 

\subsection*{Evolution error of \(c\mathbb{P}_n\mathbb{P}_m\) scheme}

With the reconstructor in \cref{sec:reconstruction} at hand, an auxiliary scheme has to be introduced here as an ingredient of the proof:
\begin{equation}
\label{eq:du_R_mdt}
    \frac{d}{dt}u^{(m)}_{h,\mathcal{R}} = \mathcal{P}^{(m)}_h\left(  \mathcal{R}_h\left(\Pi_h\,u^{(m)}_{h,\mathcal{R}} \right)\right),
\end{equation}
where $u^{(m)}_{h,\mathcal{R}} \in \mathcal{V}^{(m)}_h$ is the solution of the auxiliary scheme and $\Pi_h:\mathcal{V}^{(m)}_h\to \mathcal{V}^{(n)}_h$ is the global projection operator from order $m$ to $n$. Considering that the $c\mathbb{P}_n\mathbb{P}_m$ scheme we derived in \cref{eq:projected_ho_time_derivative} is mathematically equivalent to the projected high-order time-derivative according to \cref{eq:projected_DGESM_0}, the solution of $c\mathbb{P}_n\mathbb{P}_m$, $u^{(n)}_{h,\mathcal{R}}$, is governed by:
\begin{equation}
\label{eq:du_R_ndt}
    \frac{d}{dt}u^{(n)}_{h,\mathcal{R}} = \mathcal{P}^{(n,m)}_{h,\mathcal{R}}\left( u^{(n)}_{h,\mathcal{R}}\right) = \Pi_h\mathcal{P}^{(m)}_h\left(  \mathcal{R}_h\left(\Pi_h\,u^{(m)}_{h,\mathcal{R}} \right)\right),
\end{equation}
where $\mathcal{P}^{(n,m)}_{h,\mathcal{R}}$ is the semi-discrete operator of the $c\mathbb{P}_n\mathbb{P}_m$ method and is locally defined by \cref{eq:cPnPm}.
Then we can obtain the following:
\begin{lemma}
    Let $u^{(n)}_{h,\mathcal{R}}(\cdot,0) = \Pi_hu^{(m)}_{h,\mathcal{R}}(\cdot,0) =\Pi_hu^{(m)}_h(\cdot,0) \in \mathcal{V}^{(n)}_h$. Then the solutions of \cref{eq:du_R_mdt} and \cref{eq:du_R_ndt} satisfy:
    \begin{equation}
        u^{(n)}_{h,\mathcal{R}}(\cdot,t) = \Pi_hu^{(m)}_{h,\mathcal{R}}(\cdot,t),\; \forall t \in [0,T].
    \end{equation}
\end{lemma}
To clarify the notation, a schematic diagram of different methods is given in \cref{fig:schematic_diagram}. 

Instead of directly dealing with the operator $\mathcal{P}^{(n,m)}_{\mathcal{R}}$, we first prove that the auxiliary solution achieves the optimal convergence order of $m+1$ using the lemmas introduced earlier and present the following theory:
\begin{theorem}[Convergence of the auxiliary solution]
\label{thm:convergence_of_the auxiliary_solution}
 Assume that the solution is smooth enough:
    \begin{equation}
        C_u = \sup_{t\in[0,T]}\vert u(\cdot,t) \vert_{H^{m+1}(\Omega)}< \infty, 
    \end{equation}
    and $\mathcal{P}_h^{(m)}$ is locally Lipschitz continuous:
\begin{equation}
\label{eq:def_beta_max}
    \left\| \mathcal{P}^{(m)}_h\left(v\right) - \mathcal{P}_h^{(m)}\left(w\right)\right\| \leq \beta_{\max}\| v-w\|,\;\forall v,w\in\mathcal{V}^{(m)}_h,
\end{equation}
    given the initial condition $u^{(m)}_{h,\mathcal{R}}(\cdot,0) =u^{(m)}_h(\cdot,0)=\Pi^{(m)}u_0$, the solution of the auxiliary scheme \cref{eq:du_R_mdt} achieves an optimal order of $m+1$:
    \begin{equation}
        \left\|u - u^{(m)}_{h,\mathcal{R}}\right\| \leq Ch^{m+1},
    \end{equation}
    where $C$ is a constant independent of $h$. More precisely,
    \begin{equation}
        C = C_{\Pi}C_u + C^{(m)}+C^{(m)}_a,
    \end{equation}
    and two estimates of $C^{(m)}_a$ are:
    \begin{enumerate}
        \item \textbf{Estimate 1:}
            \begin{equation}
            \label{eq:Cma1}
                C_{a,1}^{(m)} = \frac{C_{\mathcal{R}_c} C_u}{C_{\mathcal{R}_s}}\left(e^{ \beta_{\max}C_{\mathcal{R}_s}t}-1\right);
            \end{equation}

        \item \textbf{Estimate 2:}
            \begin{equation}
            \label{eq:Cma2}
                C_{a,2}^{(m)} = \frac{\beta_{\max}C_{\mathcal{R}_c} C_u}{\left( 2\beta_{\max}C_{\mathcal{R}_s}+1\right)^{1/2}}\left( e^{\left( 2\beta_{\max}C_{\mathcal{R}_s}+1 \right)t} -1\right)^{1/2}.
        \end{equation}
    \end{enumerate}
    Here $C_{\Pi}$ is the constant from \cref{lemma:projection_error}, $C^{(m)}$ from \cref{lem:DGSEM_convergence} and $C_{\mathcal{R}_c}, C_{\mathcal{R}_s}$ from \cref{lemma:projector_property_global}.
\end{theorem}

\begin{proof}
    First, by introducing the error $\epsilon^{(m)}_{h,\mathcal{R}}:=u^{(m)}_h-u^{(m)}_{h,\mathcal{R}}$
    we can decompose the total error into three parts:
    \begin{multline}
        \left\| u - u^{(m)}_{h,\mathcal{R}}\right\| 
        \leq 
        \left\| u - u^{(m)}_{h,\ast} \right\| 
        + 
        \left\| u^{(m)}_{h,\ast} - u^{(m)}_h \right\|
        + 
        \left\| u_h^{(m)} - u^{(m)}_{h,\mathcal{R}} \right\| \\
        =\Vert \epsilon_{\Pi}\Vert +\left\| \epsilon^{(m)}_h\right\| + \left\|\epsilon^{(m)}_{h,\mathcal{R}}\right\|,
    \end{multline}
    where the first two errors are bounded using \cref{lemma:projection_error} and \cref{lem:DGSEM_convergence}.
    As for the last term, its evolution is governed by:
    \begin{equation}
    \frac{d}{dt}\epsilon^{(m)}_{h,\mathcal{R}} = \mathcal{P}^{(m)}_h\left( u^{(m)} \right)-\mathcal{P}^{(m)}_h\left( \mathcal{R}_h\left(\Pi_hu^{(m)}_{h,\mathcal{R}}\right)\right).
    \end{equation}

    Taking the inner product with $\epsilon^{(m)}_{h,\mathcal{R}}$, we obtain:
    \begin{equation}
    \label{eq:1_2_dedt^2}
        \frac{1}{2}\frac{d}{dt}\left\| \epsilon^{(m)}_{h,\mathcal{R}} \right\|^2 = \Big\langle\mathcal{P}^{(m)}_h\left( u^{(m)} \right)- \mathcal{P}^{(m)}_h\left( \mathcal{R}_h\left(\Pi_h u^{(m)}_{h,\mathcal{R}}\right)\right),\epsilon^{(m)}_{h,\mathcal{R}}\Big\rangle.
    \end{equation}

    From \cref{eq:def_beta_max}, we have: 
    \begin{equation}
        \left\| \mathcal{P}^{(m)}_h\left( u^{(m)}_h \right)- \mathcal{P}^{(m)}_h\left( \mathcal{R}_h\left(\Pi_hu^{(m)}_{h,\mathcal{R}}\right)\right)  \right\| \leq \beta_{\max} \left\| \epsilon^{\prime(m)}_{h,\mathcal{R}}  \right\|
    \end{equation}
    where $\epsilon^{\prime(m)}_{h,\mathcal{R}}:=u^{(m)}_h - \mathcal{R}_h\left(\Pi_hu^{(m)}_{h,\mathcal{R}}\right)$. 
    Thus,
    \begin{equation}
    \label{eq:norm_innner_product_P_e}
        \Bigg\vert \Big\langle\mathcal{P}^{(m)}_h\left(u_h^{(m)} \right)
        - 
        \mathcal{P}^{(m)}_h\left( \mathcal{R}\left(\Pi_hu^{(m)}_{h,\mathcal{R}}\right)\right),\epsilon^{(m)}_{h,\mathcal{R}}\Big\rangle \Bigg\vert \leq \beta_{\max}\left\|\epsilon^{\prime(m)}_{h,\mathcal{R}} \right\| \cdot
        \left\| \epsilon^{(m)}_{h,\mathcal{R}} \right\|.
    \end{equation}
    Noting that $\epsilon^{\prime(m)}_{h,\mathcal{R}}$ can be split into two parts:
    \begin{multline}
    \label{eq:norm_e_prime_R^m}         \left\|\epsilon^{\prime(m)}_{h,\mathcal{R}}\right\| 
    = 
    \left\|u^{(m)}_h - \mathcal{R}_h\left(\Pi_h\,u^{(m)}_{h,\mathcal{R}}\right) \right\|
    \\=
    \left\|\left(u^{(m)}_h - \mathcal{R}_h\left(\Pi_hu^{(m)}_h\right)\right)+\left(\mathcal{R}_h\left(\Pi_hu_h^{(m)}\right)-\mathcal{R}\left(\Pi_hu^{(m)}_{\mathcal{R}}\right)\right)\right\| \\
        \leq \left\|\epsilon^{(m)}_{h,\mathcal{R}\Pi}\left( u^{(m)}_h\right)\right\| + \left\|\mathcal{R}_h\left(\Pi_hu_h^{(m)}\right)-\mathcal{R}_h\left(\Pi_hu^{(m)}_{h,\mathcal{R}}\right) \right\|,
    \end{multline}
    where $\epsilon^{(m)}_{h,\mathcal{R}\Pi}(\cdot):=(\cdot)-\mathcal{R}_h\left(\Pi_h(\cdot)\right)$ is the projection-reconstruction error. Using \cref{lemma:projector_property_global}, we obtain the following:
    \begin{equation}
    \label{eq:C_Rc_Cu}
        \left\| \epsilon^{(m)}_{h,\mathcal{R}\Pi}(u_h^{(m)}) \right\| \leq C_{\mathcal{R}_c} C_uh^{m+1},
    \end{equation}
    and 
    \begin{equation}
        \label{eq:C_Rs}
        \left\|\mathcal{R}_h\left(\Pi_hu^{(m)}_h\right)-\mathcal{R}_h\left(\Pi_hu^{(m)}_{h,\mathcal{R}}\right) \right\| \leq C_{\mathcal{R}_s}\left\| \epsilon^{(m)}_{h,\mathcal{R}}\right\|.
    \end{equation}

    Combining \cref{eq:1_2_dedt^2}, \cref{eq:norm_innner_product_P_e}, \cref{eq:norm_e_prime_R^m}, \cref{eq:C_Rc_Cu} and \cref{eq:C_Rs} yields:
    \begin{equation}
    \label{eq:dehRdt^2}
        \frac{d}{dt}\left\| \epsilon^{(m)}_{h,\mathcal{R}} \right\|^2 \leq 2\beta_{\max}C_{\mathcal{R}_s}\left\| \epsilon^{(m)}_{h,\mathcal{R}}\right\|^2 + 2\beta_{\max}C_{\mathcal{R}_c} C_uh^{m+1} \cdot \left\| \epsilon^{(m)}_{h,\mathcal{R}}\right\|.
    \end{equation}

    Since $\frac{d}{dt}(\cdot)^2 = 2(\cdot)\frac{d}{dt}(\cdot)$, \cref{eq:dehRdt^2} can be written as:
    \begin{equation}
    \label{eq:dehRdt}
            \frac{d}{dt}\left\| \epsilon^{(m)}_{h,\mathcal{R}} \right\| \leq \beta_{\max}C_{\mathcal{R}_s}\left\| \epsilon^{(m)}_{h,\mathcal{R}}\right\| + \beta_{\max}C_{\mathcal{R}_c} C_uh^{m+1}.
    \end{equation}
    
    Considering that an exact numerical initial condition of $m^{th}$-order is given, i.e., $\left\|\epsilon^{(m)}_{\mathcal{R}}(\cdot,0)\right\|=0$, applying Grönwall's inequality to \cref{eq:dehRdt} yields the the formulation of $C^{(m)}_{a,1}$ in \cref{eq:Cma1}.

    Alternatively, using Young’s inequality leads to:
    \begin{equation}
        2\beta_{\max}C_{\mathcal{R}_c} C_uh^{m+1} \cdot \left\| \epsilon^{(m)}_{h,\mathcal{R}}\right\| \leq \beta_{\max}^2C_{\mathcal{R}_c}^2 C_u^2h^{2m+2} + \left\| \epsilon^{(m)}_{h,\mathcal{R}}\right\|^2,
    \end{equation}
    and we obtain:
    \begin{equation}
        \label{eq:dedt^2_R}
        \frac{d}{dt}\left\| \epsilon^{(m)}_{h,\mathcal{R}} \right\|^2 \leq (2\beta_{\max}C_{\mathcal{R}_s}+1)\left\| \epsilon^{(m)}_{h,\mathcal{R}}\right\|^2 + \beta_{\max}^2C_{\mathcal{R}_c}^2 C_u^2h^{2m+2}.
    \end{equation}
    Similarly, applying $\left\|\epsilon^{(m)}_{\mathcal{R}}(\cdot,0)\right\|=0$ and Grönwall's inequality and taking the square root leads to the formulation of $C^{(m)}_{a,2}$ in \cref{eq:Cma2}.
\end{proof}

\begin{remark}
    Assumption \cref{eq:def_beta_max} should be interpreted as a local regularity condition on the DGSEM-LGL time-derivative operator $\mathcal{P}_h^{(m)}$. In particular, the bound is expected to hold only in regimes where the exact solution and the
    corresponding discrete approximation remain sufficiently smooth. For smooth solutions, the LGL differentiation matrices, quadrature weights, metric terms, and discrete residuals define a bounded finite-dimensional operator on the relevant solution manifold. However, near discontinuities, shocks, or strongly under-resolved gradients, such a uniform Lipschitz constant may fail to exist or may become very large. Thus the analysis based on \cref{eq:def_beta_max} is restricted to the smooth-solution regime.
\end{remark}

\begin{remark}
    The constant $\beta_{\max}$ in \cref{eq:def_beta_max} generally depends on $\mathcal{P}^{(m)}_h$, which is parameterized by mesh size $h$ and the polynomial order $m$. It can be approximated by:
    \begin{equation}
        \beta_{\max} \approx \sup_{v} \left\| \mathcal{J}_{\mathcal{P}^{(m)}_h }\left( v \right)\right\|,\qquad\mathcal{J}_{\mathcal{P}^{(m)}_h }\left( v \right):=\frac{\partial \mathcal{P}^{(m)}_h(v)}{\partial v},
    \end{equation}
    and it measures how strongly the spatial discretization amplifies perturbations in the discrete state. 
\end{remark}

Based on the convergence of the auxiliary solution in \cref{thm:convergence_of_the auxiliary_solution}, it can be easily proved that the reconstructed solution of the $c\mathbb{P}_n\mathbb{P}_m$ scheme also achieves the convergence order of $m+1$:
\begin{theorem}[Convergence of the $c\mathbb{P}_n\mathbb{P}_m$ solution]
\label{thm:convergence_of_the_cPnPm}
    Using the assumptions in \Cref{thm:convergence_of_the auxiliary_solution}, given the initial condition $u^{(n)}_{h,\mathcal{R}}(\cdot,0) = \Pi_hu^{(m)}_{h,\mathcal{R}}(\cdot,0)=\Pi_hu_h^{(m)}(\cdot,0)$, the reconstructed solution from $c\mathbb{P}_n\mathbb{P}_m$ scheme \cref{eq:du_R_ndt}, i.e.,$\mathcal{R}_h\left( u^{(n)}_{h,\mathcal{R}} \right)$, achieves the $m+1$ convergence order:
    \begin{equation}
        \left\| u - \mathcal{R}_h\left( u^{(n)}_{h,\mathcal{R}} \right)\right\|\leq Ch^{m+1},
    \end{equation}
    where $C$ is a constant independent of $h$. More precisely,
    \begin{equation}
        C = C_{\Pi}C_u + C^{(m)}+C^{(m)}_b,
    \end{equation}
    where $C^{(m)}_b=C_{\mathcal{R}_c}C_u+C_{\mathcal{R}_s}C_a^{(m)}$.
    
    Here $C_{\Pi}$ is the constant from \cref{lemma:projection_error}, $C_{(m)}$ from \cref{lem:DGSEM_convergence}, $C_{\mathcal{R}_c},C_{\mathcal{R}_s}$ from \cref{lemma:projector_property_global} and $C_a^{(m)}$ from \cref{thm:convergence_of_the auxiliary_solution}.
\end{theorem}

\begin{proof}
    Following a similar approach as before, the total error can be split:
    \begin{equation}
        \left\| u - \mathcal{R}_h\left( u^{(n)}_{h,\mathcal{R}} \right)\right\| \leq \left\| \epsilon_{\Pi}\right\| +\left\| \epsilon^{(m)}_h\right\| + \left\|\epsilon^{\prime(m)}_{h,\mathcal{R}}\right\|, 
    \end{equation}
    where $\epsilon^{\prime(m)}_{h,\mathcal{R}}:=u_h^{(m)} - \mathcal{R}_h\left(\Pi_hu^{(m)}_{h,\mathcal{R}}\right)$. Combining \cref{eq:norm_e_prime_R^m}, \cref{eq:C_Rc_Cu}, \cref{eq:C_Rs} and the estimates of $C_{a}^{(m)}$ completes the proof.
\end{proof}

\begin{corollary}
\label{corollary:two_error_estimates}
    The difference between the $L^2$-norm error of the $m^{th}$-order solution from standard DGSEM-LGL and the reconstructed solution from $c\mathbb{P}_n\mathbb{P}_m$ is bounded by:
    \begin{equation}
          \bigg\vert\left\|u- u^{(m)}_{h}\right\|
          -
          \left\|u - \mathcal{R}_h\left(u^{(n)}_{h,\mathcal{R}}\right)\right\|\bigg\vert
          \leq \left\|u^{(m)}_h - \mathcal{R}_h\left(u^{(n)}_{h,\mathcal{R}}\right)\right\|
        \leq
        C_b^{(m)}h^{m+1},
    \end{equation}
    where the two estimates of $C_b$ are:
    \begin{enumerate}
        \item \textbf{Estimate 1:}
            \begin{equation}
            \label{eq:Cmb1}
                C_{b,1}^{(m)} = C_{\mathcal{R}_c} C_ue^{ \beta_{\max}C_{\mathcal{R}_s}t};
            \end{equation}

        \item \textbf{Estimate 2:}
            \begin{equation}
            \label{eq:Cmb2}
                C_{b,2}^{(m)} = C_{\mathcal{R}_c} C_u\left(\frac{\beta_{\max}C_{\mathcal{R}_s}}{\left( 2\beta_{\max}C_{\mathcal{R}_s}+1\right)^{1/2}}\left( e^{\left( 2\beta_{\max}C_{\mathcal{R}_s}+1 \right)t} -1\right)^{1/2}+1\right).
        \end{equation}
    \end{enumerate}
    
\end{corollary}
It reveals that the reconstruction introduces another source of error. More specifically, there would be a shift between the $h$-convergence lines of $u^{(m)}$ and $\mathcal{R}_h\left(u^{(n)}_{h,\mathcal{R}}\right)$, which is related to the regularity of the solution ($C_u$), the property of the reconstructor $\mathcal{R}$ ($C_{\mathcal{R}_c}$ and $C_{\mathcal{R}_s}$) and the DGSEM operator ($\beta_{\max}$). This shift will be observed in the following numerical results.

\section{Numerical results} 
In this section, we study the performance of the $c\mathbb{P}_n\mathbb{P}_m$ method in a wide range of numerical cases. For the one dimensional cases, the linear advection, the viscous Burgers' equation, and the Euler equation are tested. Furthermore, we extend the method to two-dimensional cases, tested using inviscid isentropic vortex and finally decaying homogeneous isotropic turbulence (DHIT). All simulation results are obtained with an explicit five stage fourth order accurate low storage Runge-Kutta scheme (RK(5,4))\cite{carpenter_fourth-order_1994}. The setting of the time step satisfies the CFL condition for $n^{th}$-order simulation:
\begin{equation}
    \text{CFL} = \frac{\lambda_c \Delta t(2n+1)}{\Delta x},
\end{equation}
where $\lambda_c$ is the maximum global wave speed. The CFL number is set to about $0.25$ for all the convergence tests. 
The DGSEM-LGL and $c\mathbb{P}_n\mathbb{P}_m$ solvers are written using JAX \cite{jax2018github,wang_accelerating_2026}. All cases are run on a laptop computer with an Intel(R) Core(TM) Ultra 9 275HX CPU @ 2.70 GHz and RAM 32,0 GB.

\subsection{Linear advection}
Consider the one-dimensional linear advection problem:
\begin{equation}
\label{eq:linear_advection}
    \begin{split}
        u_t + u_x &= 0,\,x\in(0,1)\\
        u(x,0) &= \sin(2\pi x)+1.5,\,\forall x\in[0,1]\\
        u(0,t)&=u(1,t),\,\forall t\in[0,T].
    \end{split}
\end{equation}
We solve this problem in a series of refined meshes (from 8 to 256) using both standard DGSEM-LGL and $c\mathbb{P}_n\mathbb{P}_m$ methods, and their $L^2$-errors are compared in \cref{fig:convergence_rate_linear_advection}. The detailed experimental order of convergence (EOC) of $c\mathbb{P}_n\mathbb{P}_m$ methods are detailed in \cref{tab:convergence_order_linear_advection}, and all the methods ($n=1,m=3,4,5$) achieve the optimal convergence order of $m+1$. These results demonstrate that the correction procedure effectively enhances the accuracy of low-order DGSEM-LGL discretizations and reproduces the convergence characteristics of higher-order methods without directly increasing the underlying polynomial degree of the base scheme. In addition, from \cref{tab:convergence_order_linear_advection} it can be seen that there is a vertical shift between the convergence lines of the standard methods $\mathbb{P}_m$ and that of the methods $c\mathbb{P}_n\mathbb{P}_m$, which is a predictable outcome of \cref{corollary:two_error_estimates} and a numerical validation is given below.

\begin{figure}[tbhp]
\centering
    \includegraphics[trim=1cm 0cm 1cm 0cm, clip, width=0.7\textwidth]{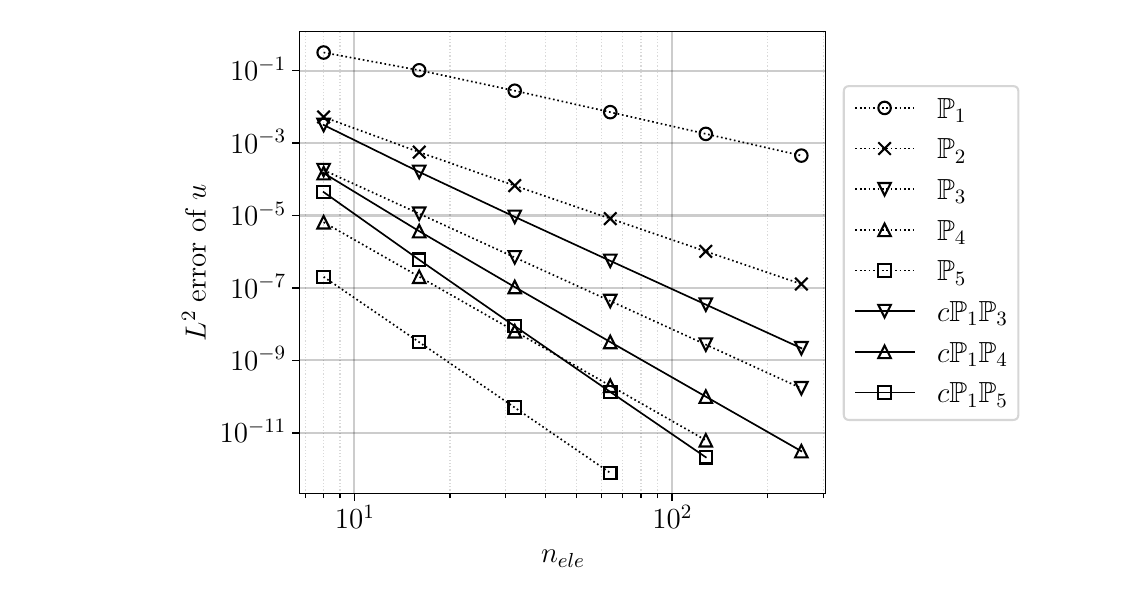}  \\
    \caption{The convergence of $L^2$-error of standard DGSEM-LGL methods and the $c\mathbb{P}_n\mathbb{P}_m$ methods for the one-dimensional linear advection problem.}
\label{fig:convergence_rate_linear_advection}
\end{figure}

\begin{table}[htbp]
{\footnotesize
  \caption{The EOC of $c\mathbb{P}_n\mathbb{P}_m$ for one-dimensional linear advection problem}  \label{tab:convergence_order_linear_advection}
\begin{center}
  \begin{tabular}{ccccccc} \hline
   $n_{ele}$ & $L^1$ & $\mathcal{O}_{L^1}$ & $L^2$ & $\mathcal{O}_{L^2}$ & $L^{\infty}$ & $\mathcal{O}_{L^{\infty}}$ \\ \hline
   $c\mathbb{P}_1\mathbb{P}_3$ &&&&&& \\
    8   & 2.63E-03 & -   & 3.15E-03   & - & 5.99E-03   & - \\ 
    16  & 1.32E-04 & 4.3 & 1.59E-04 & 4.3 & 4.01E-04 & 3.9 \\ 
    32  & 7.42E-06 & 4.1 & 9.17E-06 & 4.1 & 2.93E-05 & 3.8 \\
    64  & 4.35E-07 & 4.1 & 5.60E-07 & 4.0 &	1.95E-06 & 3.9 \\
    128 & 2.62E-08 & 4.0 & 3.48E-09	& 4.0 &	1.26E-07 & 4.0 \\
    256 & 1.61E-09 & 4.0 & 2.17E-09	& 4.0 &	7.96E-09 & 4.0 \\
    \\
    $c\mathbb{P}_1\mathbb{P}_4$ &&&&&& \\
    8   & 1.07E-04 & -   & 1.48E-04 & -   & 3.44E-04 & - \\
    16  & 2.87E-06 & 5.2 & 3.69E-06 & 5.3 & 8.90E-06 & 5.3 \\ 
    32  & 8.46E-08 & 5.1 & 1.06E-07 & 5.1 & 2.41E-07 & 5.2 \\ 
    64  & 2.58E-09 & 5.0 & 3.22E-09 & 5.0 & 6.88E-09 & 5.1 \\
    128 & 7.94E-11 & 5.0 & 1.00E-10 & 5.0 &	2.03E-10 & 5.1 \\
    256 & 2.48E-12 & 5.0 & 3.12E-12	& 5.0 &	6.33E-12 & 5.0 \\ 
    \\
    $c\mathbb{P}_1\mathbb{P}_5$ &&&&&& \\
    8   & 3.49E-05 & -   & 4.42E-05 & -   & 8.97E-05 & - \\
    16  & 4.97E-07 & 6.1 & 5.99E-07 & 6.2 & 1.20E-06 & 6.0 \\ 
    32  & 7.22E-09 & 6.1 & 8.84E-09 & 6.1 & 1.91E-08 & 6.0 \\ 
    64  & 1.08E-10 & 6.1 & 1.35E-10 & 6.0 & 3.27E-10 & 5.9 \\
    128 & 1.68E-12 & 6.0 & 2.12E-12 & 6.0 &	5.39E-12 & 5.9 \\
    \hline
  \end{tabular}
\end{center}
}
\end{table}

From \cref{corollary:two_error_estimates}, it can be known that this vertical shift between the $\mathbb{P}_m$ and $c\mathbb{P}_n\mathbb{P}_m$ methods is bounded by $\left\| \epsilon^{\prime(m)}_{h,\mathcal{R}}\right\|$. In the following, we numerically validate the constants in our error bounds. For the constants associated with the global consistency of the reconstructor, i.e, $C_{\mathcal{R}_c}$ and $C_u$, instead of evaluating them separately, we approximate the term $C_{\mathcal{R}_c}C_uh^{m+1}$ by the initial error:
\begin{equation}
    C_{\mathcal{R}_c}C_uh^{m+1} \approx \left\| \epsilon^{\prime(m)}_{h,\mathcal{R}}(\cdot, 0)\right\| = \left\| u^{(m)}_h(\cdot, 0) - \mathcal{R}_h\left( \Pi_hu^{(m)}_h(\cdot, 0) \right) \right\|.
\end{equation}

\begin{figure}[tbhp]
\centering
    \subfloat[]{\label{fig:convergence_rate_C_Rs}\includegraphics[trim=3cm 0cm 3cm 0cm, clip, width=0.34\textwidth]{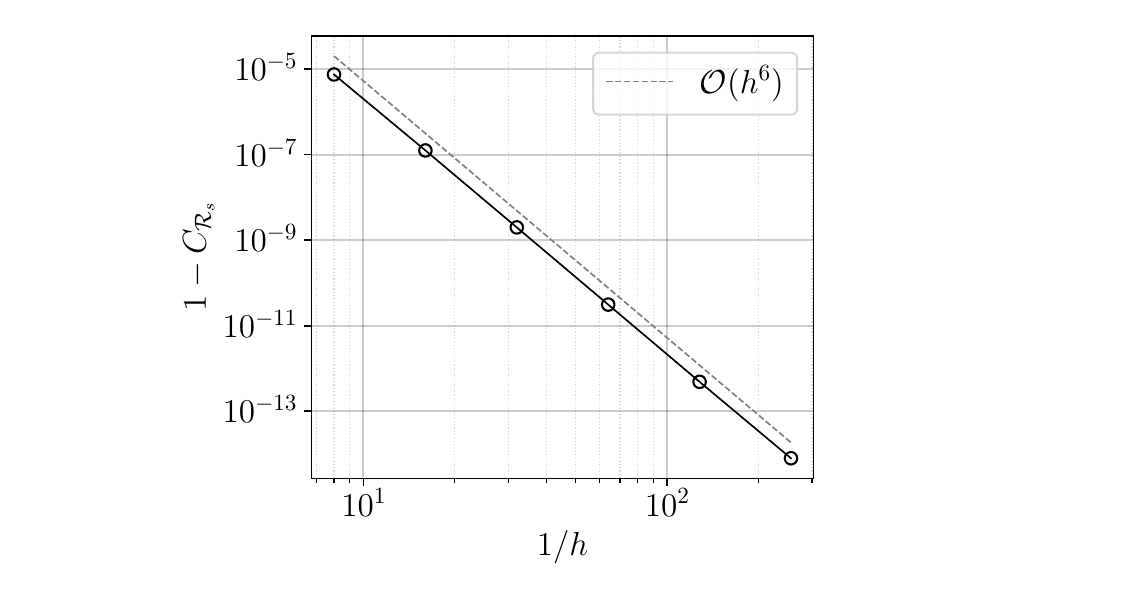}}
    \subfloat[]{\label{fig:error_p1p3_t0-1_2estimates}\includegraphics[trim=0cm 0cm 0cm 1.6cm, clip, width=0.66\textwidth]{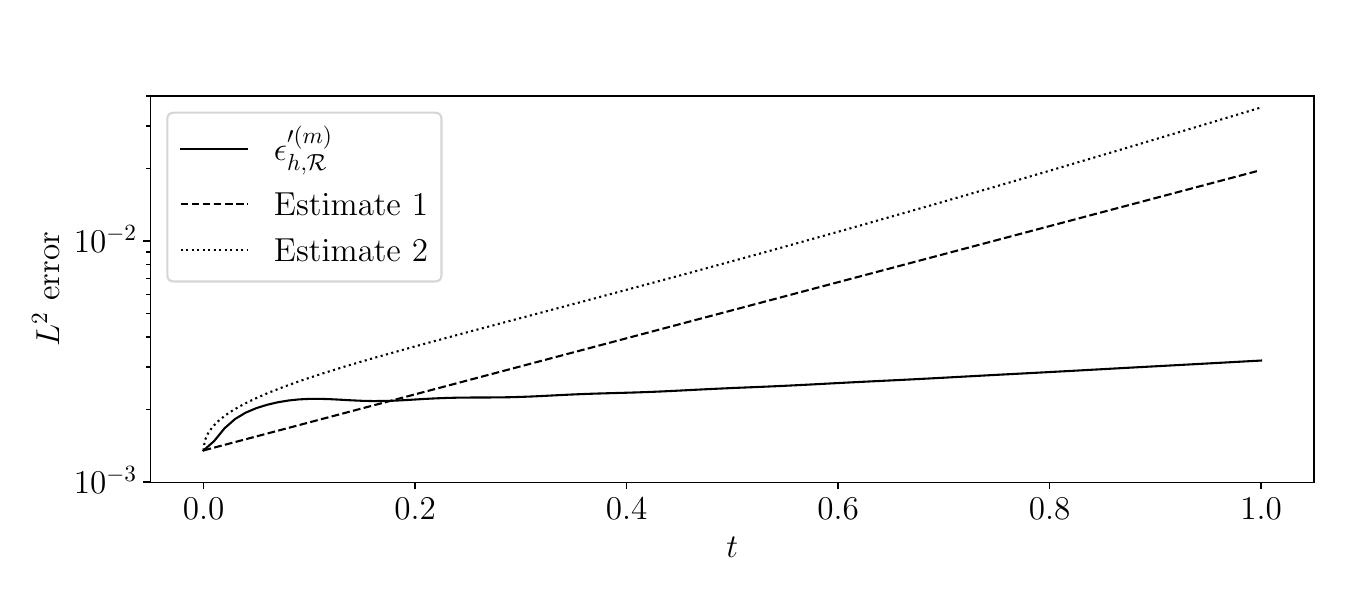}}
    \caption{(a) The convergence of constant $C_{\mathcal{R}_s}$; (b) The comparison of the evolution of $\left\| \epsilon^{\prime(m)}_{h,\mathcal{R}}\right\|$ and two error estimating methods: $\left\| \epsilon^{\prime(m)}_{h,\mathcal{R}}(\cdot,t)\right\|\leq \left\| \epsilon^{\prime(m)}_{h,\mathcal{R}}(\cdot, 0)\right\| \alpha(t)$, where $\alpha(t)$ is taken as $e^{ \beta_{\max}t}$ and $\frac{\beta_{\max}}{\left( 2\beta_{\max}+1\right)^{1/2}}\left( e^{\left( 2\beta_{\max}+1 \right)t} -1\right)^{1/2}+1$ respectively for Estimate 1 and Estimate 2.}
\label{fig:C_Rs_and_2estimates}
\end{figure}

\begin{figure}[tbhp]
\centering
    \includegraphics[trim=1cm 0cm 1cm 0cm, clip, width=0.7\textwidth]{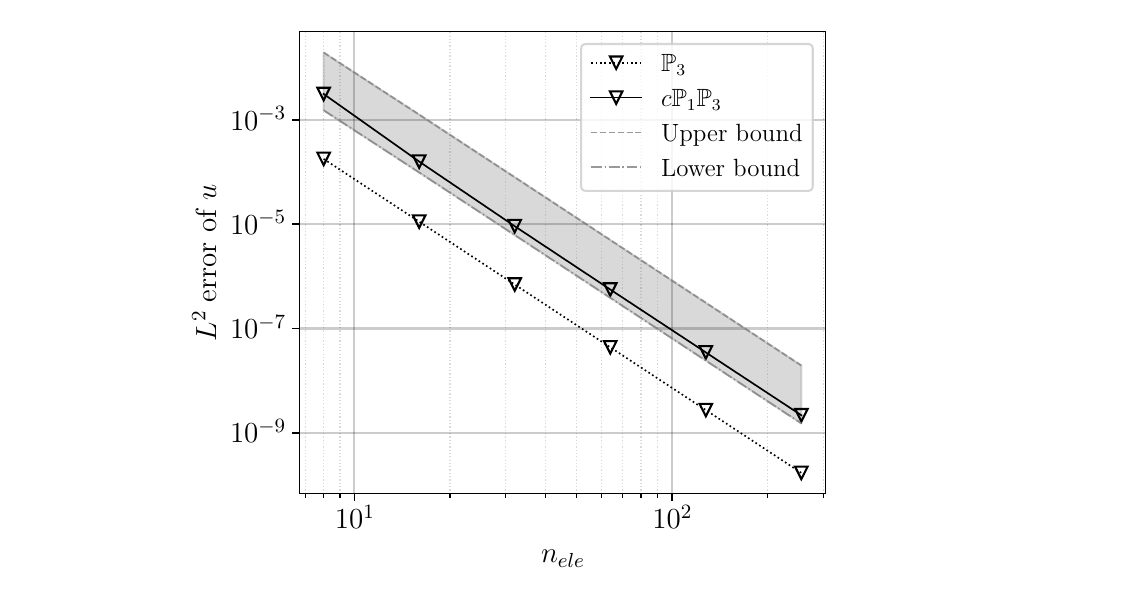}  \\
    \caption{The upper and lower bound estimates of $3^{rd}$-order methods for the one-dimensional linear advection problem.}
\label{fig:convergence_rate_p3_linear_advection}
\end{figure}

As for the constant $C_{\mathcal{R}_s}$ associated with global stability, since $\mathcal{R}_h$ is a linear operator, the inequality \cref{eq:global_stability} can be written as:
\begin{equation}
    \frac{\left\| \mathcal{R}_h\left( v_h^{(m)} \right) \right\|} {\left\|  v_h^{(m)} \right\|}\leq C_{\mathcal{R}_s},\qquad \forall v_h^{(m)} \in \mathcal{V}^{(m)}_h,
\end{equation}
and $C_{\mathcal{R}_s}$ measures how much energy the reconstructor recovers from the projected solution compared to the original high-order one. We compute $C_{\mathcal{R}_s}$ of the reconstructor from $\mathcal{V}^{(1)}_h$ to  $\mathcal{V}^{(3)}_h$ numerically for different $h$ and find that it converges to $1$ from the bottom. The convergence of $C_{\mathcal{R}_s}$ can be seen in \cref{fig:convergence_rate_C_Rs} and it converges to $1$ with the rate of $\mathcal{O}(h^{6})$. When involved in the error estimates, it can be simply approximated by $1$.

Finally, considering that $\mathcal{P}^{(m)}_h$ is linear in this case and the corresponding Lipschitz constant $\beta_{\max}$ can be numerically computed from the $L^2$-norms of the solution and its time derivative. Here we take the value $\beta_{\max} \approx 2.68$.

With these approximations in hands, the evolution of $\left\| \epsilon^{\prime(m)}_{h,\mathcal{R}}\right\|$ can be bounded by:
\begin{equation}
\label{eq:alphat}
    \left\| \epsilon^{\prime(m)}_{h,\mathcal{R}}(\cdot,t)\right\|\leq \left\| \epsilon^{\prime(m)}_{h,\mathcal{R}}(\cdot, 0)\right\| \alpha(t),
\end{equation}
where $\alpha(t)$ controls the evolution and can be computed from \cref{eq:Cmb1} and \cref{eq:Cmb2}:
    \begin{enumerate}
        \item \textbf{Estimate 1:}
            \begin{equation}
            \label{eq:alpha1}
                \alpha_1(t) = e^{ \beta_{\max}t};
            \end{equation}

        \item \textbf{Estimate 2:}
            \begin{equation}
            \label{eq:alpha2}
                \alpha_2(t) = \frac{\beta_{\max}}{\left( 2\beta_{\max}+1\right)^{1/2}}\left( e^{\left( 2\beta_{\max}+1 \right)t} -1\right)^{1/2}+1.
        \end{equation}
    \end{enumerate}
The results of two estimating methods are compared in \cref{fig:error_p1p3_t0-1_2estimates}. It can be seen that \textbf{Estimate 1} provides better bounds. Assuming that the error increases monotonically, taking $\alpha(t)=1$ leads to an estimate of the lower bound for the error. The $L^2$-error of the $3^{rd}$-order methods on different meshes is plotted in \cref{fig:convergence_rate_p3_linear_advection}, where the upper and lower bounds are also given. It shows that the real error of $c\mathbb{P}_n\mathbb{P}_m$ is located in the bounded region (denoted gray), which shows the effectiveness of our error estimating methods.

\subsection{Viscous Burgers' equation}

One-dimensional viscous Burgers' equation is taken here as an example to test the convergence of $c\mathbb{P}_n\mathbb{P}_m$ scheme when the viscous term is involved:
\begin{equation}
    u_t +f_x=g_x,
\end{equation}
where $f= \frac{1}{2}u^2$ and $g=\nu u_x$ are the inviscid and viscous flux, respectively. We study the $h$-convergence property using the following manufactured solution:
\begin{equation}
    u(x,t) = \sin(x)e^{-\nu t},\;x\in[0,2\pi]
\end{equation}
with the source term:
\begin{equation}
    s = \frac{1}{2}\sin({2x})e^{-2\nu t}.
\end{equation}
The periodic boundary conditions are imposed. Here we set $\nu=0.01$. The $L^2$-errors are checked at $t=2$.

\begin{figure}[t]
\centering
    \includegraphics[trim=1cm 0cm 1cm 0cm, clip, width=0.7\textwidth]{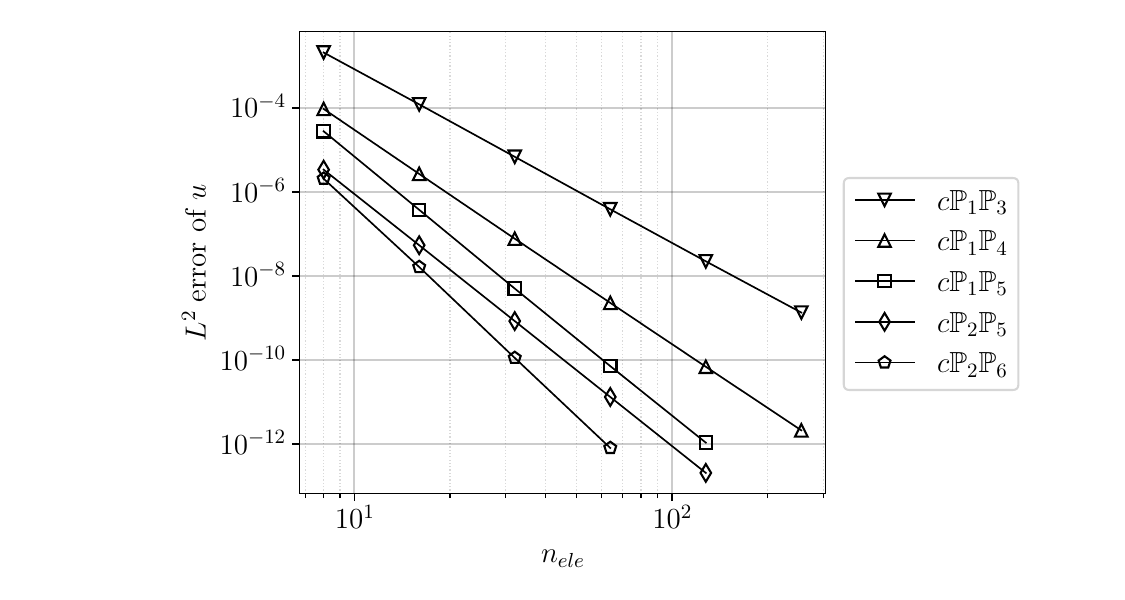}  \\
    \caption{The convergence of $L^2$-error of the standard DGSEM-LGL methods and the $c\mathbb{P}_n\mathbb{P}_m$ methods for the one-dimensional Burgers' equation.}
\label{fig:convergence_rate_burgers}
\end{figure}

\begin{table}[ht]
{\footnotesize
  \caption{The EOC of $c\mathbb{P}_n\mathbb{P}_m$ for one-dimensional Burgers' equation}  \label{tab:convergence_order_burgers}
\begin{center}
  \begin{tabular}{ccccccc} \hline
   $n_{ele}$ & $L^1$ & $\mathcal{O}_{L^1}$ & $L^2$ & $\mathcal{O}_{L^2}$ & $L^{\infty}$ & $\mathcal{O}_{L^{\infty}}$ \\ \hline
   $c\mathbb{P}_1\mathbb{P}_3$ &&&&&& \\
    8   & 1.59E-03 & -   & 2.07E-03 & -   & 6.93E-03 & - \\
    16  & 8.94E-05 & 4.1 & 1.21E-04 & 4.1 & 4.48E-04 & 4.0 \\ 
    32  & 5.08E-06 & 4.1 & 6.85E-06 & 4.1 & 2.63E-05 & 4.1 \\ 
    64  & 2.92E-07 & 4.1 & 3.87E-07 & 4.1 & 1.49E-06 & 4.1 \\
    128 & 1.74E-08 & 4.1 & 2.24E-08 & 4.1 &	8.26E-08 & 4.1 \\
    256 & 1.05E-09 & 4.0 & 1.35E-09	& 4.0 &	4.66E-09 & 4.1 \\ 
    \\
    $c\mathbb{P}_1\mathbb{P}_4$ &&&&&& \\
    8   & 7.90E-05 & -   & 9.41E-05 & -   & 1.97E-04 & - \\
    16  & 2.15E-06 & 5.2 & 2.68E-06 & 5.1 & 6.40E-06 & 4.9 \\ 
    32  & 6.21E-08 & 5.1 & 7.73E-08 & 5.1 & 1.71E-07 & 5.2 \\ 
    64  & 1.81E-09 & 5.1 & 2.30E-09 & 5.1 & 4.75E-09 & 5.1 \\
    128 & 5.48E-11 & 5.0 & 6.96E-11 & 5.0 &	1.40E-10 & 5.1 \\
    256 & 1.68E-12 & 5.0 & 2.14E-12	& 5.0 &	4.26E-12 & 5.0 \\
    \\
    $c\mathbb{P}_1\mathbb{P}_5$ &&&&&& \\
    8   & 2.86E-05 & -   & 2.77E-05 & -   & 6.91E-05 & - \\
    16  & 2.92E-07 & 6.2 & 3.76E-07 & 6.2 & 1.10E-06 & 5.9 \\ 
    32  & 3.97E-09 & 6.2 & 5.04E-09 & 6.2 & 1.52E-08 & 6.2 \\ 
    64  & 5.77E-11 & 6.1 & 7.21E-11 & 6.1 & 2.12E-10 & 6.2 \\
    128 & 8.60E-13 & 6.1 & 1.09E-12 & 6.1 &	3.04E-12 & 6.1 \\
    \\
    $c\mathbb{P}_2\mathbb{P}_5$ &&&&&& \\
    8   & 2.86E-06 & -   & 3.38E-06 & -   & 7.18E-06 & - \\
    16  & 4.54E-08 & 6.0 & 5.40E-08 & 6.0 & 1.21E-07 & 5.9 \\ 
    32  & 7.14E-10 & 6.0 & 8.49E-10 & 6.0 & 1.90E-09 & 6.0 \\ 
    64  & 1.12E-11 & 6.0 & 1.33E-11 & 6.0 & 2.99E-11 & 6.0 \\
    128 & 1.75E-13 & 6.0 & 2.08E-13 & 6.0 &	4.69E-13 & 6.0 \\
    \\
    $c\mathbb{P}_2\mathbb{P}_6$ &&&&&& \\
    4   & 1.44E-06 & -   & 2.06E-06 & -   & 6.09E-06 & - \\
    8   & 1.20E-08 & 6.9 & 1.64E-08 & 7.0 & 5.78E-08 & 6.7 \\ 
    16  & 9.20E-11 & 7.0 & 1.13E-10 & 7.2 & 2.63E-10 & 7.7 \\ 
    32  & 6.75E-13 & 7.1 & 8.16E-13 & 7.1 & 1.62E-12 & 7.3 \\
    \hline
  \end{tabular}
\end{center}
}
\end{table}

The comparison of $L^2$-errors and the detailed EOC of $c\mathbb{P}_n\mathbb{P}_m$ methods (both for $n=1$ and $n=2$) are given in \cref{fig:convergence_rate_burgers} and \cref{tab:convergence_order_burgers} respectively. As expected, all the methods tested ($n=1,m=3,4,5$ and $n=2,m=5,6$) achieve the optimal convergence order of $m+1$, which proves that the discretization for the viscous term in \cref{eq:projected_g^m} maintains high-order convergence on low-order nodes. 

\begin{figure}[tbhp]
\centering
    \subfloat[$3^{rd}$-order methods]{\label{fig:cost_3order_burgers}\includegraphics[trim=3cm 0cm 3cm 0cm, clip, width=0.5\textwidth]{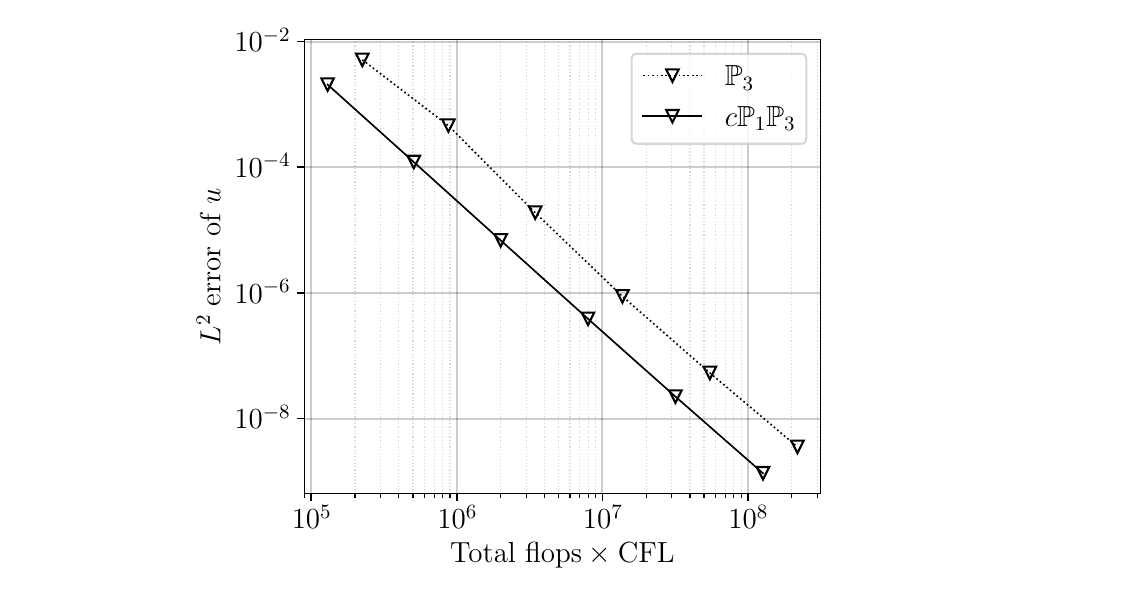}} 
    \subfloat[$4^{th}$-order methods]{\label{cost_4order_burgers}\includegraphics[trim=3cm 0cm 3cm 0cm, clip, width=0.5\textwidth]{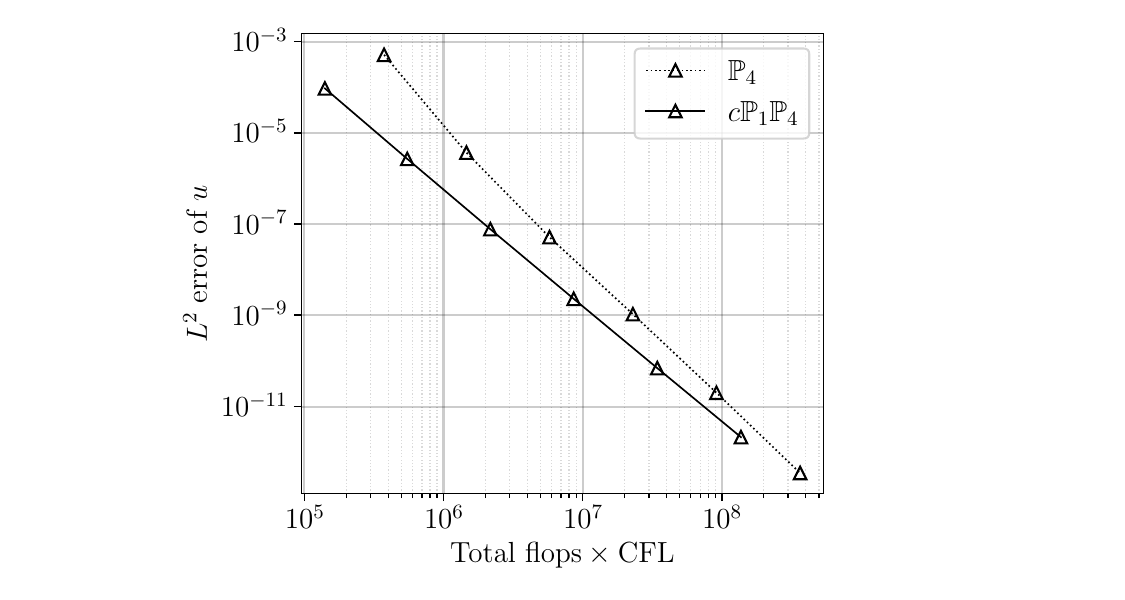}}
    \\
    \subfloat[$5^{th}$-order methods]{\label{fig:cost_5order_burgers}\includegraphics[trim=3cm 0cm 3cm 0cm, clip, width=0.5\textwidth]{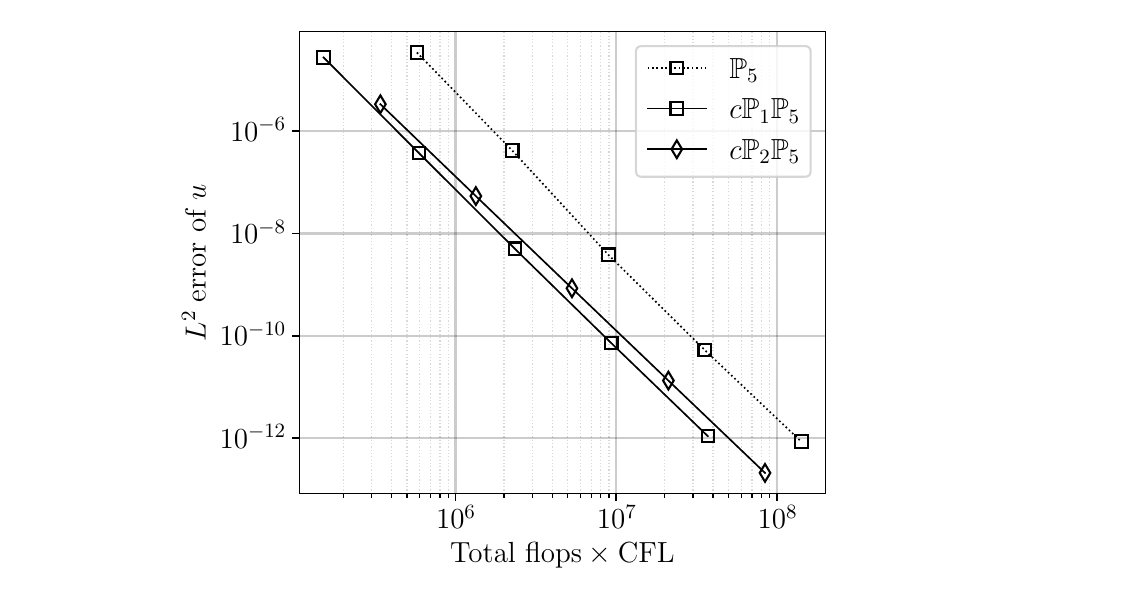}} 
    \subfloat[$6^{th}$-order methods]{\label{cost_6order_burgers}\includegraphics[trim=3cm 0cm 3cm 0cm, clip, width=0.5\textwidth]{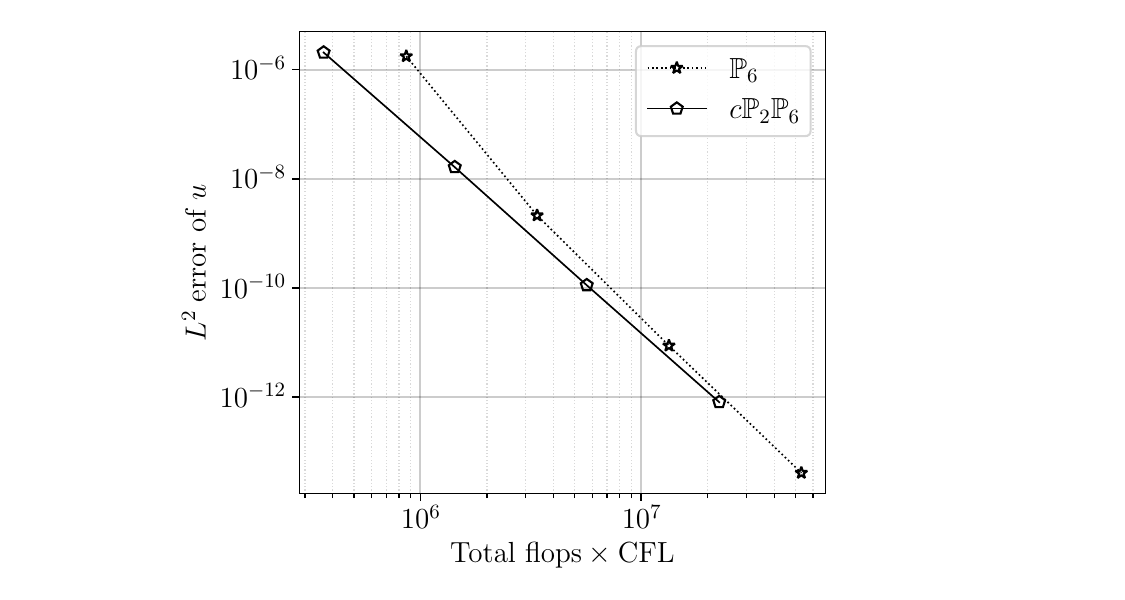}}
    \caption{The comparison of computation efficiency of different methods for solving one-dimensional Burgers equation}
\label{fig:cost_comparison_burgers}
\end{figure}

Furthermore, to compare the actual computational cost of different methods, the $L^2$-errors and the corresponding total flops of the simulations are plotted in \cref{fig:cost_comparison_burgers}. Since the CFL number is not exactly the same for different methods, the total flops are normalized by the CFL number to denote the equivalent computation overhead. It can be seen that to achieve the same level of $L^2$-error, $c\mathbb{P}_n\mathbb{P}_m$ methods need fewer flops than the standard methods, which results from the combined effects of the simplified implementations of projected high-order time derivative \cref{eq:projected_ho_time_derivative}and the modal reconstructor. 

Note that although both $c\mathbb{P}_1\mathbb{P}_5$ and $c\mathbb{P}_2\mathbb{P}_5$ demonstrate $6^{th}$-order convergence, the convergence line of the latter is slightly lower than that of the former in \cref{fig:convergence_rate_burgers}. This phenomenon can also be explained by \cref{eq:alphat}. Although the reconstructors applied in both schemes are $6^{th}$-order, the projection-reconstruction error of $c\mathbb{P}_2\mathbb{P}_5$ is smaller than that of $c\mathbb{P}_1\mathbb{P}_5$ since more DOFs are provided in the $\mathbb{P}_2$ element than in $\mathbb{P}_1$. However, this gain requires an additional computational cost, and from \cref{fig:cost_5order_burgers} it can be found that $c\mathbb{P}_1\mathbb{P}_5$ is slightly more efficient than $c\mathbb{P}_2\mathbb{P}_5$.

\subsection{One-dimensional Euler equation}
Now we extend the $c\mathbb{P}_n\mathbb{P}_m$ methods to conservative systems. Considering the one-dimensional Euler equation:
\begin{equation}
    \vec{q}_t+\vec{f}(\vec{q})_x = \vec{s},
\end{equation}
where $\vec{q}$ and $\vec{f}(\vec{q})$ are the vectors of conservative variables and the corresponding non-linear fluxes:
\begin{equation}
    \vec{q}=\left[\begin{array}{c}
         \rho \\
         \rho u \\
         E
    \end{array}\right],\;
    \vec{f}(\vec{q}) = \left[\begin{array}{c}
         \rho u \\
         \rho u^2+p \\
         u(E + p)
    \end{array}\right],
\end{equation}
and $\vec{s}$ is the source term. $\rho$ is the density; $u$ is the velocity; $p$ is the pressure and $E$ is the energy. The hypothesis of a calorically ideal gas is adopted here to close the system:
\begin{equation}
    p = (\gamma-1)\left( E - \frac{1}{2}\rho u^2 \right),
\end{equation}
where $\gamma$ is the specific heat ratio and is set to $1.4$.

\begin{figure}[tbhp]
\centering
    \subfloat[$n=1$, $m=3,4,5$]{\label{fig:convergence_rate_nonlinear_upwind_p1px}\includegraphics[trim=2cm 0cm 1cm 0cm, clip, width=0.7\textwidth]{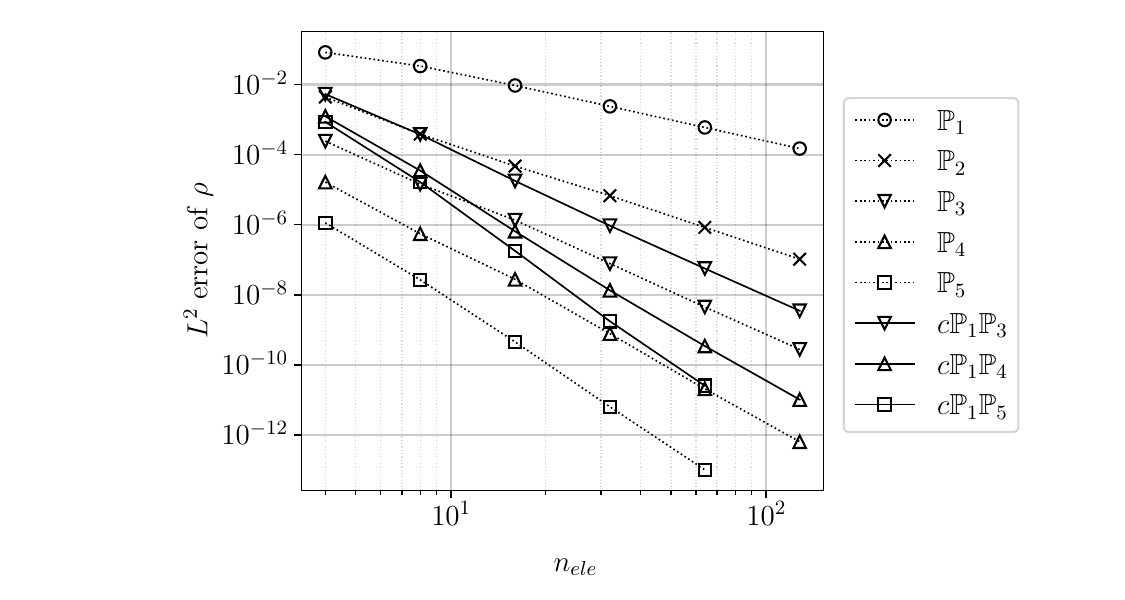}}  \\
    \subfloat[$n=2$, $m=5,6$]{\label{fig:convergence_rate_nonlinear_upwind_p2px}\includegraphics[trim=2cm 0cm 1cm 0cm, clip, width=0.7\textwidth]{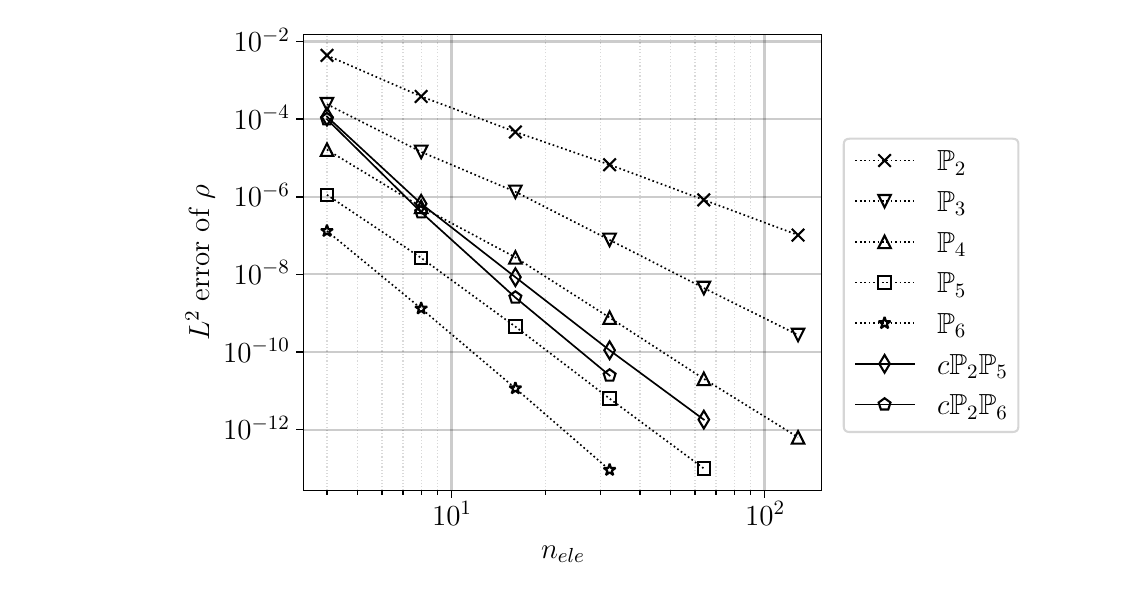}}
    \caption{The convergence of $L^2$-error of standard DGSEM-LGL methods and $c\mathbb{P}_n\mathbb{P}_m$ methods for the one-dimensional Euler equation.}
\label{fig:convergence_rate_nonlinear_upwind}
\end{figure}

\begin{table}[htbp]
{\footnotesize
  \caption{The EOC of $c\mathbb{P}_n\mathbb{P}_m$ for one-dimensional Euler equation}  \label{tab:convergence_order_euler}
\begin{center}
  \begin{tabular}{ccccccc} \hline
   $n_{ele}$ & $L^1$ & $\mathcal{O}_{L^1}$ & $L^2$ & $\mathcal{O}_{L^2}$ & $L^{\infty}$ & $\mathcal{O}_{L^{\infty}}$ \\ \hline
   $c\mathbb{P}_1\mathbb{P}_3$ &&&&&& \\
    4   & 4.38E-03 & -   & 5.31E-03 & -   & 1.44E-02 & - \\
    8   & 3.16E-04 & 3.8 & 3.75E-04 & 3.8 & 1.15E-03 & 3.7 \\ 
    16  & 1.44E-05 & 4.4 & 1.77E-05 & 4.4 & 6.08E-04 & 4.2 \\ 
    32  & 6.99E-07 & 4.4 & 9.31E-07 & 4.2 & 3.63E-06 & 4.0 \\
    64  & 4.07E-08 & 4.1 & 5.59E-08 & 4.0 &	2.21E-07 & 4.0 \\
    128 & 2.51E-09 & 4.0 & 3.47E-09	& 4.0 &	1.35E-08 & 4.0 \\ 
    \\
    $c\mathbb{P}_1\mathbb{P}_4$ &&&&&& \\
    4   & 1.10E-03 & -   & 1.23E-03 & -   & 2.41E-02 & - \\
    8   & 3.09E-05 & 5.1 & 3.49E-05 & 5.1 & 6.80E-05 & 5.1 \\ 
    16  & 5.55E-07 & 5.8 & 6.45E-07 & 5.7 & 1.38E-06 & 5.6 \\ 
    32  & 1.07E-08 & 5.7 & 1.33E-08 & 5.6 & 2.90E-08 & 5.5 \\
    64  & 2.75E-10 & 5.3 & 3.42E-10 & 5.3 &	7.29E-10 & 5.3 \\
    128 & 8.11E-12 & 5.0 & 1.01E-11	& 5.0 &	2.12E-11 & 5.0 \\ 
    \\
    $c\mathbb{P}_1\mathbb{P}_5$ &&&&&& \\
    4   & 7.40E-04 & -   & 8.56E-04 & -   & 1.62E-03 & - \\
    8   & 1.43E-05 & 5.7 & 1.64E-05 & 5.7 & 2.66E-05 & 6.0 \\ 
    16  & 1.58E-07 & 6.5 & 1.78E-07 & 6.5 & 3.22E-07 & 6.4 \\ 
    32  & 1.50E-09 & 6.7 & 1.74E-09 & 6.7 & 3.60E-09 & 6.5 \\
    64  & 2.08E-11 & 6.1 & 2.51E-11 & 6.1 &	5.64E-11 & 6.0 \\
    \\
    $c\mathbb{P}_2\mathbb{P}_5$ &&&&&& \\
    4   & 1.03E-04 & -   & 1.16E-04 & -   & 1.95E-04 & - \\
    8   & 5.28E-07 & 7.6 & 6.61E-07 & 7.4 & 2.34E-06 & 6.4 \\ 
    16  & 6.52E-09 & 6.3 & 8.48E-09 & 6.3 & 2.81E-08 & 6.3 \\ 
    32  & 9.03E-11 & 6.1 & 1.11E-10 & 6.2 & 3.29E-10 & 6.4 \\
    64  & 1.56E-12 & 5.9 & 1.81E-12 & 6.0 &	4.38E-12 & 6.2 \\
    \\
    $c\mathbb{P}_2\mathbb{P}_6$ &&&&&& \\
    4   & 8.66E-05 & -   & 9.93E-05 & -   & 1.92E-04 & - \\
    8   & 3.42E-07 & 8.0 & 3.99E-07 & 8.0 & 1.18E-06 & 7.3 \\ 
    16  & 1.94E-09 & 7.5 & 2.53E-09 & 7.3 & 1.14E-08 & 6.7 \\ 
    32  & 1.86E-11 & 6.7 & 2.48E-10 & 6.7 & 1.04E-10 & 6.8 \\
    \hline
  \end{tabular}
\end{center}
}
\end{table}

Following \cite{gassner_split_2016}, we test the convergence order of different methods using the following nonlinear manufactured solution in the domain $[0,2]$:
\begin{equation}
\label{eq:manufactured_solution_euler}
\begin{split}
    \rho(x,t) &= \frac{1}{10}\sin{\left( \pi(x-t) \right)}+2,\\
    u(x,t) &= 1,\\
    E(x,t) &= \left( \frac{1}{10}\sin{\left( \pi(x-t) \right)}+2 \right)^2,
\end{split}    
\end{equation}
and the corresponding source terms are:
\begin{equation}
\begin{split}
    s_{\rho} &= 0,\\
    s_{\rho u} &= c_1\cos{\left( \pi(x-t) \right)} + c_2\sin{\left( 2\pi(x-t) \right)},\\
    s_{E} &= c_1\cos{\left( \pi(x-t) \right)} + c_2\sin{\left( 2\pi(x-t) \right)},
\end{split}    
\end{equation}
where $c_1 = \frac{7}{20}(\gamma-1)\pi$ and $c_2 = \frac{1}{100}(\gamma-1)\pi$. 

\begin{figure}[tbhp]
\centering
    \subfloat[$3^{rd}$-order methods]{\label{fig:cost_3order_nonlinear}\includegraphics[trim=3cm 0cm 3cm 0cm, clip, width=0.5\textwidth]{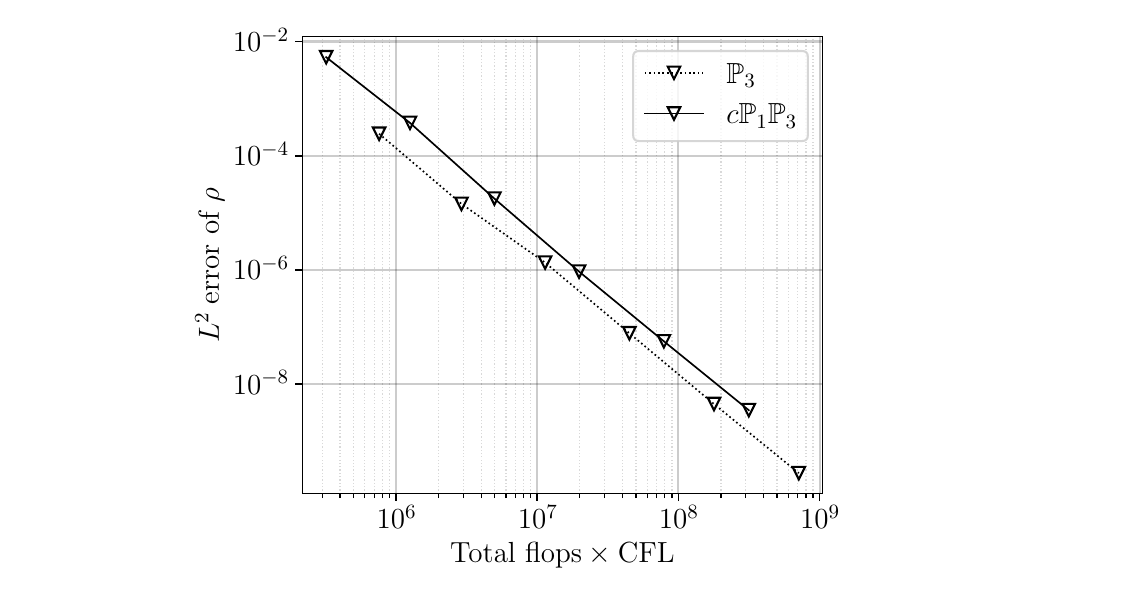}} 
    \subfloat[$4^{th}$-order methods]{\label{cost_4order_nonlinear}\includegraphics[trim=3cm 0cm 3cm 0cm, clip, width=0.5\textwidth]{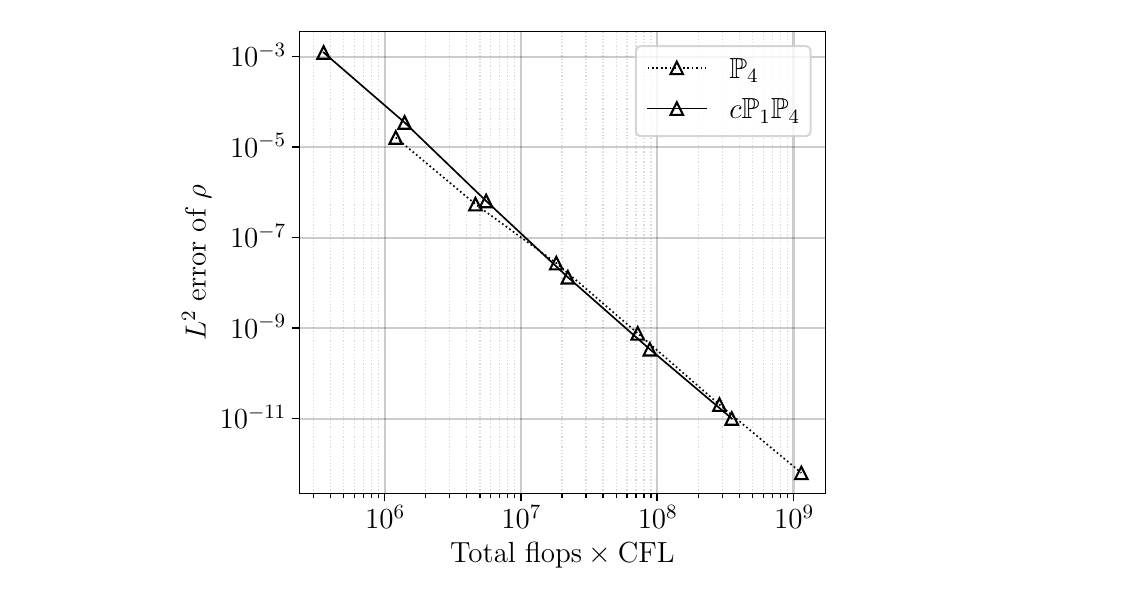}}
    \\
    \subfloat[$5^{th}$-order methods]{\label{fig:cost_5order_nonlinear}\includegraphics[trim=3cm 0cm 3cm 0cm, clip, width=0.5\textwidth]{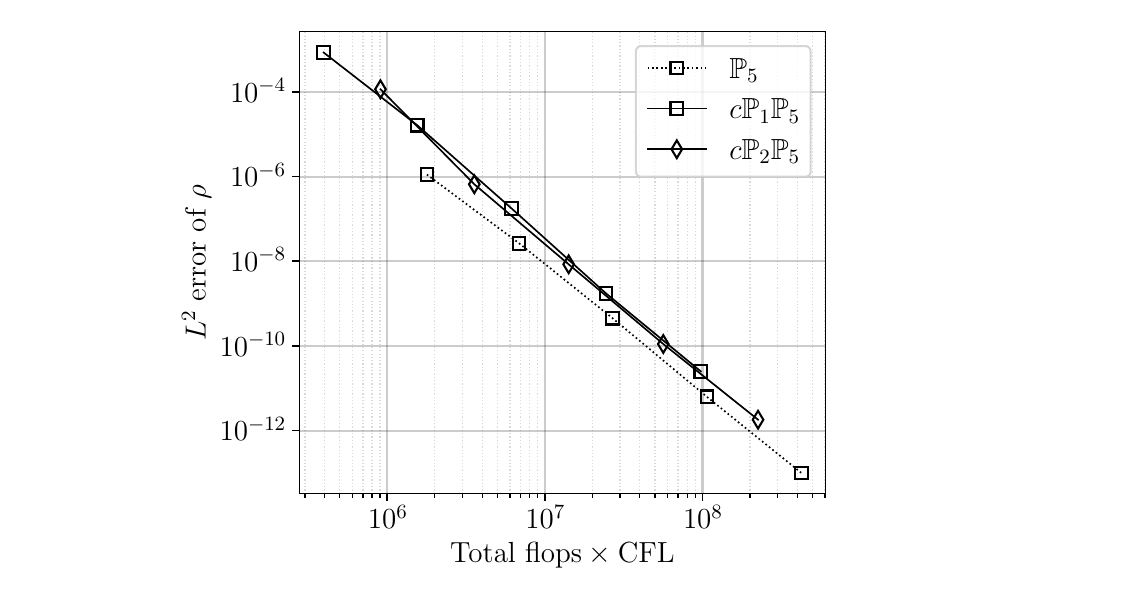}} 
    \subfloat[$6^{th}$-order methods]{\label{cost_6order_nonlinear}\includegraphics[trim=3cm 0cm 3cm 0cm, clip, width=0.5\textwidth]{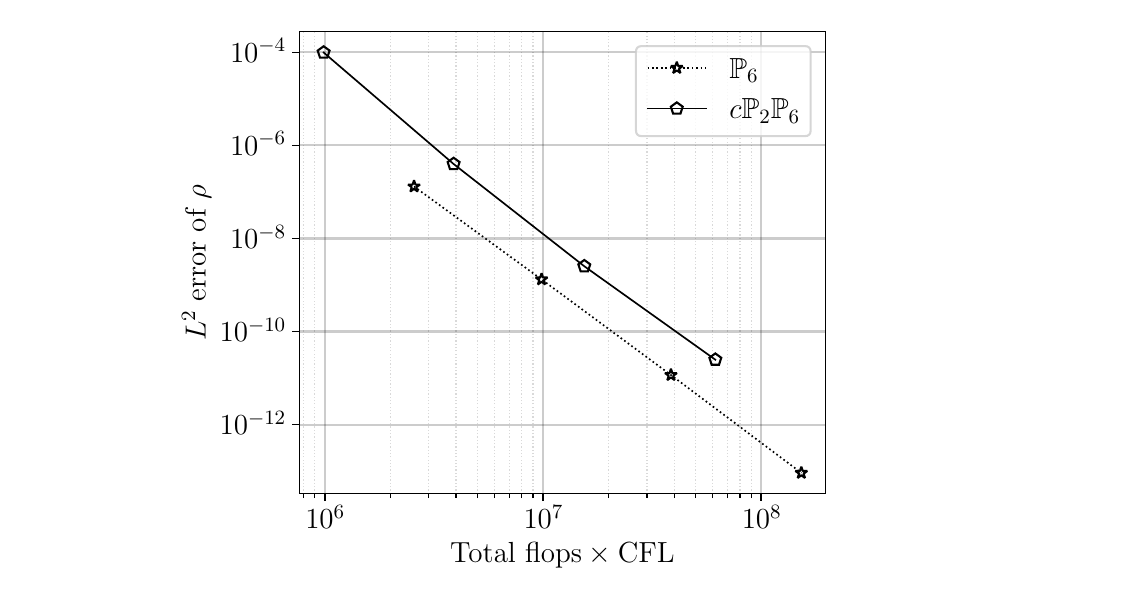}}
    \caption{The comparison of computation efficiency of different methods for solving one-dimensional Euler equation}
\label{fig:cost_comparison_euler}
\end{figure}

We solve this problem on a uniform mesh ($n_{ele}$ ranges from $4$ to $128$) and the upwind numerical flux is adopted here. \Cref{fig:convergence_rate_nonlinear_upwind} presents the $L^2$-error convergence behavior of the standard DGSEM-LGL and the proposed $c\mathbb{P}_n\mathbb{P}_m$ schemes for this problem. More specifically, all the results of the $c\mathbb{P}_1\mathbb{P}_m$ schemes are shown in \cref{fig:convergence_rate_nonlinear_upwind_p1px}, while those of the $c\mathbb{P}_2\mathbb{P}_m$ schemes are plotted in \cref{fig:convergence_rate_nonlinear_upwind_p2px}. A detailed EOC of various methods is provided in \cref{tab:convergence_order_euler}, and again, all methods demonstrate the convergence of ${m+1}^{th}$-order. Moreover, the computational efficiency of the standard DGSEM-LGL and $c\mathbb{P}_2\mathbb{P}_m$ methods is compared in \cref{fig:cost_comparison_euler}. 

It should be noted that, without viscous terms, the efficiency of $c\mathbb{P}_2\mathbb{P}_m$ is slightly lower than that of the standard method of the same order, except for the case of methods of $4^{th}$-order. A reasonable explanation can be found in \cref{fig:convergence_rate_nonlinear_upwind_p1px}, where it can be seen that the accuracy gain is considerable when  increasing $m$ from $3$ to $4$ but smaller from $4$ to $5$. This odd-even effect can also be detected if we check the error of $c\mathbb{P}_1\mathbb{P}_m,\;m=3,4,5$ for the case of $n_{ele}=4$. Increasing $m=3$ to $4$ leads to a reduction in the $L^2$-error from 5.31E-03 to 1.23E-03 (approximately x4.3). However, increasing $m=4$ to $5$ reduces the $L^2$-error from 1.23E-03 to 8.56E-04 (only about x1.4). Considering that the error is dominated by the projection-reconstruction error of the initial condition, reconstruction applied in $c\mathbb{P}_1\mathbb{P}_4$ is more efficient when accounting for the accuracy gain per computational cost for the sine wave solution \cref{eq:manufactured_solution_euler}.

\subsection{Two-dimensional isentropic vortex}

As a smooth, exact solution to the non-linear Euler equations, the isentropic vortex is an essential test case for assessing the convergence order of high-order methods. It is described by:
\begin{equation}
\begin{split}
    \rho(x,y,t) &= \left( 1-\frac{\epsilon^2(\gamma-1)M_{\infty}^2}{8\pi^2}\exp\left({f(x,y,t)}\right) \right)^{\frac{1}{\gamma-1}},\\
    u(x,y,t) &= 1 - \epsilon_v\frac{y-y_0}{2\pi}\exp\left(\frac{f(x,y,t)}{2}\right),\\
    v(x,y,t) &= 1 - \epsilon_v\frac{(x-x_0)-t}{2\pi}\exp\left(\frac{f(x,y,t)}{2}\right),\\
    p &= \frac{\rho^{\gamma}}{\gamma M_{\infty}^2},
\end{split} 
\end{equation}
where $f(x,y,t) = 1 - \left( ((x-x_0)-t)^2 + (y-y_0)^2 \right)$, $M_{\infty}$ is the Mach number, $\epsilon_v$ represents the strength of the vortex, and $(x_0,y_0)$ denotes the initial position of the vortex. Here, the values are $M_{\infty}=0.5$, $\epsilon_v=0.1$, and $(x_0,y_0)=(0,0)$. We solve this problem in the domain $(-10,10)\times(-5,5)$ and a periodic boundary condition is imposed. This rectangular domain is divided into $n_{ele}^{(x)}\times n_{ele}^{(y)}$ elements and $n_{ele}^{(x)}= 2n_{ele}^{(y)}$. Both local Lax-Friedrichs (LLF) and the Roe\cite{roe_approximate_1981} numerical fluxes are applied to study the convergence behavior of the $c\mathbb{P}_n\mathbb{P}_m$ scheme for this problem.

\begin{figure}[tbhp]
\centering
    \subfloat[LLF]{\label{fig:convergence_rate_isen_vortex_llf}\includegraphics[trim=2cm 0cm 1cm 0cm, clip, width=0.7\textwidth]{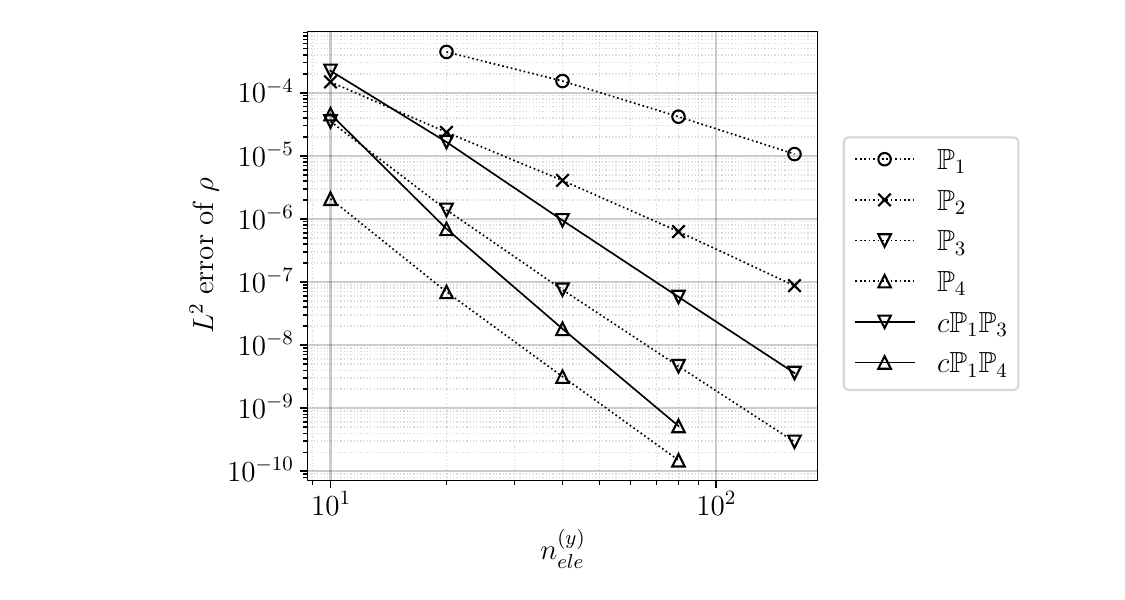}}  \\
    \subfloat[Roe]{\label{fig:convergence_rate_isen_vortex_roe}\includegraphics[trim=2cm 0cm 1cm 0cm, clip, width=0.7\textwidth]{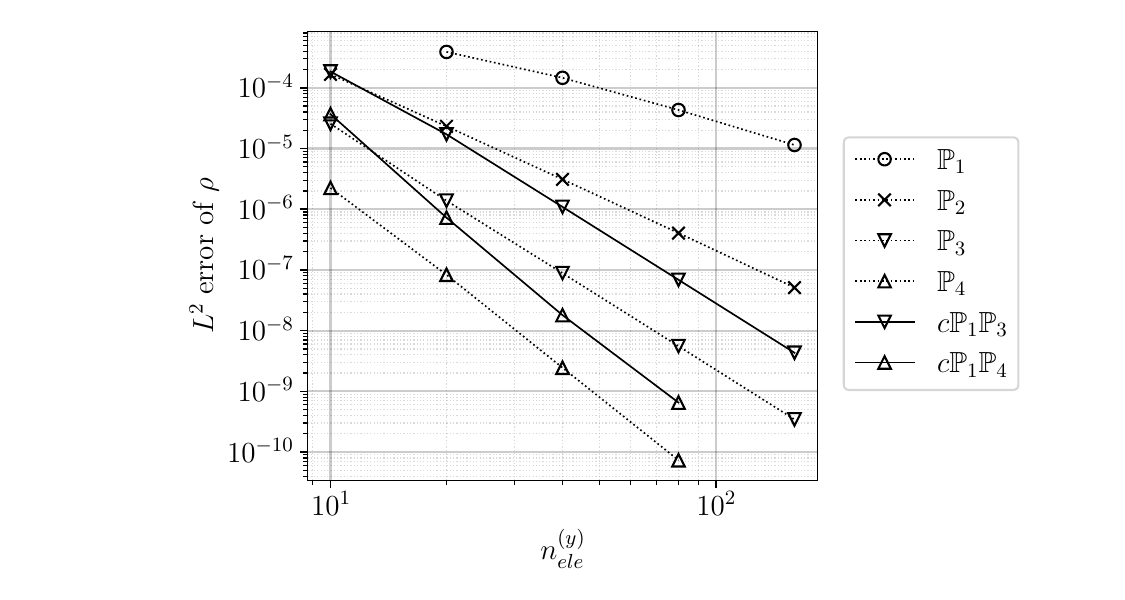}}
    \caption{The convergence order of $L^2$-error of the standard DGSEM-LGL methods and the $c\mathbb{P}_n\mathbb{P}_m$ methods for the isentropic vortex using (a)LLF and (b) Roe Riemann solvers.}
\label{fig:convergence_rate_isen_vortex}
\end{figure}

\begin{figure}[tbhp]
\centering
    \subfloat[$3^{th}$-order methods]{\label{fig:cost_p3_compare_roe_llf}\includegraphics[trim=0cm 0cm 0cm 0cm, clip, width=0.5\textwidth]{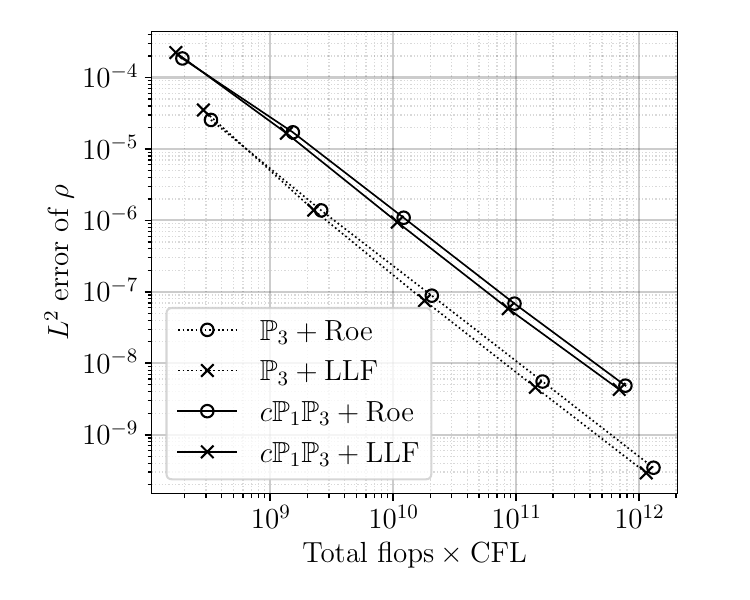}} 
    \subfloat[$4^{th}$-order methods]{\label{fig:cost_p3_compare_roe_roe}\includegraphics[trim=0cm 0cm 0cm 0cm, clip, width=0.5\textwidth]{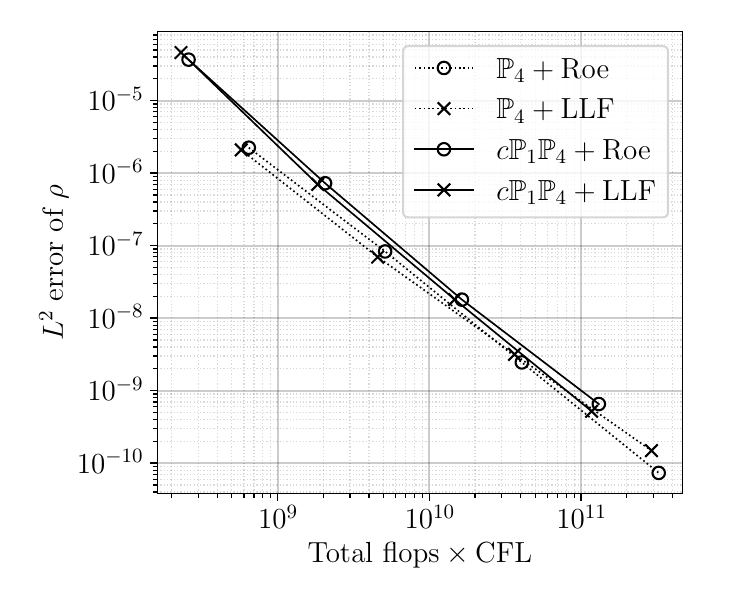}}
    \caption{The comparison of computation efficiency between the standard DGSEM-LGL methods and the $c\mathbb{P}_n\mathbb{P}_m$ methods for the isentropic vortex case.}
\label{fig:cost_isen_vortex_roe_llf}
\end{figure}

\begin{table}[htbp]
{\footnotesize
  \caption{The EOC of $c\mathbb{P}_n\mathbb{P}_m$ for two-dimensional isentropic vortex}  \label{tab:convergence_order_isentropic_vortex}
\begin{center}
  \begin{tabular}{ccccccc} \hline
   $n_{ele}^{(y)}$ & $L^1$ & $\mathcal{O}_{L^1}$ & $L^2$ & $\mathcal{O}_{L^2}$ & $L^{\infty}$ & $\mathcal{O}_{L^{\infty}}$ \\ \hline
   LLF \\
   $c\mathbb{P}_1\mathbb{P}_3$ &&&&&& \\
    10  & 9.01E-05 & -   & 2.24E-04 & -   & 2.51E-03 & - \\
    20  & 6.18E-06 & 3.9 & 1.66E-05 & 3.8 & 1.77E-04 & 3.8 \\ 
    40  & 3.34E-07 & 4.2 & 9.49E-07 & 4.1 & 1.18E-05 & 3.9 \\ 
    80  & 1.91E-08 & 4.1 & 5.79E-08 & 4.0 & 7.53E-07 & 4.0 \\
    160 & 1.14E-09 & 4.0 & 3.60E-09 & 4.0 & 4.78E-08 & 4.0 \\ 
    \\
    $c\mathbb{P}_1\mathbb{P}_4$ &&&&&& \\
    10  & 2.58E-05 & -   & 4.60E-05 & -   & 5.26E-04 & - \\
    20  & 2.55E-07 & 6.7 & 6.97E-07 & 6.0 & 1.24E-05 & 5.4 \\ 
    40  & 5.63E-09 & 5.5 & 1.81E-08 & 5.3 & 3.78E-07 & 5.0 \\ 
    80  & 1.48E-10 & 5.2 & 5.20E-10 & 5.1 & 1.10E-08 & 5.1 \\
    \\
    $\mathbb{P}_4$ &&&&&& \\
    10  & 8.29E-07 & -   & 2.09E-06 & -   & 5.52E-05 & - \\
    20  & 2.47E-08 & 5.1 & 6.96E-08 & 4.9 & 1.64E-06 & 5.1 \\ 
    40  & 9.36E-10 & 4.7 & 3.16E-09 & \fbox{4.5} & 1.03E-07 & \fbox{4.0} \\ 
    80  & 4.04E-11 & \fbox{4.5} & 1.49E-10 & \fbox{4.4} & 5.65E-09 & \fbox{4.2}\\

    \\ \hline
    Roe \\
    $c\mathbb{P}_1\mathbb{P}_3$ &&&&&& \\
    10  & 7.40E-05 & -   & 1.85E-04 & -   & 2.00E-03 & - \\
    20  & 6.24E-06 & 3.6 & 1.71E-05 & 3.4 & 2.01E-04 & 3.3 \\ 
    40  & 3.71E-07 & 4.1 & 1.09E-06 & 4.0 & 1.32E-05 & 3.9 \\ 
    80  & 2.21E-08 & 4.1 & 6.87E-08 & 4.0 & 8.26E-07 & 4.0 \\
    160 & 1.35E-09 & 4.0 & 4.32E-09 & 4.0 & 5.15E-08 & 4.0 \\
    \\
    $c\mathbb{P}_1\mathbb{P}_4$ &&&&&& \\
    10  & 2.09E-05 & -   & 3.67E-05 & -   & 3.85E-04 & - \\
    20  & 2.66E-07 & 6.3 & 7.25E-07 & 5.7 & 1.33E-05 & 4.8 \\ 
    40  & 5.72E-09 & 5.5 & 1.80E-08 & 5.3 & 3.82E-07 & 5.1 \\ 
    80  & 1.74E-10 & 5.0 & 6.54E-10 & 4.8 & 1.27E-08 & 4.9 \\

    \hline
  \end{tabular}
\end{center}
}
\end{table}

\Cref{fig:convergence_rate_isen_vortex_llf} and \cref{fig:convergence_rate_isen_vortex_roe} present the convergence of the errors $L^2$ of the standard methods and the $c\mathbb{P}_n\mathbb{P}_m$ methods using the LLF and Roe Riemann solvers, respectively. The detailed EOC can be found in \cref{tab:convergence_order_isentropic_vortex}. All $c\mathbb{P}_n\mathbb{P}_m$ methods achieve the optimal convergence order of $m+1$, as expected, which verifies the effectiveness of the two-dimensional formulation \cref{eq:2d_projected_ho_time_derivative}. However, it can be seen from \cref{fig:convergence_rate_isen_vortex_llf} that for the $4^{th}$-order standard DGSEM-LGL using the LLF Riemann solver, the convergence order degenerates as the mesh is refined. Its EOC is listed in \cref{tab:convergence_order_isentropic_vortex} and all suboptimal cases are marked with a black box. Order reduction using the LLF scheme is also reported by \cite{sherwin_order_2020}. In contrast, when coupled with the LLF scheme, the order reduction is not demonstrated for the $c\mathbb{P}_n\mathbb{P}_m$ methods. A possible explanation is that, for the standard method, the aliasing error accumulates in high-order Legendre modes as the solution evolves, while the $c\mathbb{P}_n\mathbb{P}_m$ method filters and reconstructs the high-order components, thus removing the aliasing error to some extent. It shows that the EOC of $c\mathbb{P}_n\mathbb{P}_m$ is guaranteed by the convergence-order of low-order components and the reconstructor.

As before, the computational efficiency is compared in terms of the number of normalized flops in \cref{fig:cost_isen_vortex_roe_llf} for $3^{th}$- and $4^{th}$-order methods. Both $c\mathbb{P}_1\mathbb{P}_3$ and $c\mathbb{P}_1\mathbb{P}_4$ methods are comparable to their standard counterparts, especially for $4^{th}$-order methods. It should be noted that the benefits of reconstruction are less evident for inviscid problems. In general, the more complex the original method, the more pronounced the advantages of reconstruction become. This effect will be illustrated below through simulations of two-dimensional viscous flows.

\subsection{Two-dimensional decaying homogeneous isotropic turbulence}

In this last section, we explore a compressible Navier-Stokes case. The two-dimensional decaying homogeneous isotropic turbulence (DHIT) is a classical incompressible flow problem in which the kinetic energy decays as time evolves. We simulate the flow in a square $[0,2\pi]^2$ and periodic boundary conditions are applied and follow San and Staples \cite{san_high-order_2012} for the initialization process. The vorticity distribution in the Fourier space, $\hat{\omega}(k)$, is first initialized based on the assumed initial energy spectrum $E_k$:
\begin{equation}
\label{eq:Ek}
    E(k) = \frac{a_s}{2}\frac{1}{k_p}\left( \frac{k}{k_p}\right)^{2s+1}\exp{\left[ -\left( s + \frac{1}{2}\right) \left( \frac{k}{k_p}\right)^2 \right]},
\end{equation}
where $k = \vert \bm{k}\vert = \left(k_x^2+k_y^2\right)^{1/2}$. The maximum value of the initial energy spectrum occurs at the wavenumber $k_p$ which is assumed to be $k_p = 4$ here. The coefficient $a_s$ normalizes the initial kinetic energy and is given by taking $s=3$:
\begin{equation}
    a_s = \frac{(2s + 1)^{s+1}}{2^ss!}.
\end{equation}
The initial physical velocity is recovered by using the vorticity-stream function method and taking the Fourier transform. The initial conditions of density and pressure are set based on the Mach number:
\begin{equation}
    \rho(x,y,0)=1,\; p(x,y,0)=\frac{1}{\gamma M^2_{\infty}},
\end{equation}
where $\gamma=1.4 $ and $M_{\infty}\approx0.1$. The Reynold number based on the Taylor microscale $\lambda$ is set to ${Re}_{\lambda}\approx 60$. 

\begin{figure}[tbhp]
\centering
    \includegraphics[trim=0cm 0cm 0cm 0cm, clip, width=1.0\textwidth]{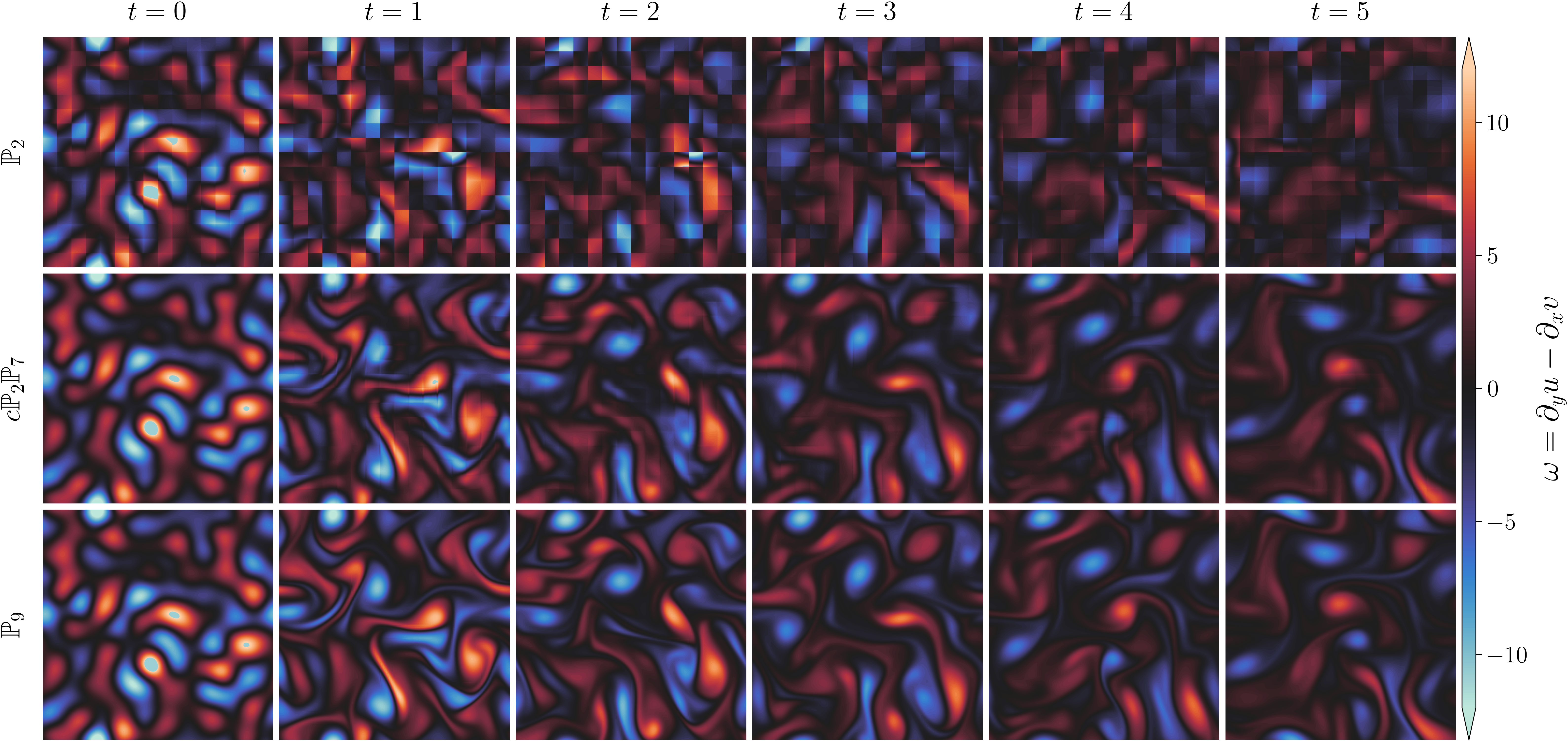}  \\
    \caption{The vorticity distributions of different methods at typical moments.}
\label{fig:vorticity_all}
\end{figure}

We run the simulation on a uniform mesh size of $16^2$ with different polynomial orders. The LLF and BR1 schemes are adopted to compute the inviscid numerical flux and viscous terms, respectively. The results of $9^{th}$- and $2^{nd}$-order standard DGSEM-LGL are considered as the high- and low-order solutions, respectively. Different $c\mathbb{P}_n\mathbb{P}_m$ methods ($n=2,m=5,6,7$) are used to simulate the same flow. To give an intuitive understanding of the enhancing effects of the $c\mathbb{P}_n\mathbb{P}_m$ method, the vorticity distributions of standard $\mathbb{P}_2$, $\mathbb{P}_9$, and $c\mathbb{P}_2\mathbb{P}_7$ at multiple moments are compared in \cref{fig:vorticity_all}. Considering that a mesh composed of $16^2$ $\mathbb{P}_2$ elements is relatively coarse for resolving the turbulent structures, the resulting evolution of vorticity is highly dissipative and inaccurate compared with the reference solution obtained using the standard $\mathbb{P}_9$ method. A large number of detailed turbulence eddies are lost. In contrast, the $c\mathbb{P}_2\mathbb{P}_7$ method accurately captures the evolution of the vorticity field over time and recovers a larger number of turbulent structures.

\begin{figure}[tbhp]
\centering
    \subfloat[$\rho u$]{\label{fig:error_evolution_rhou}\includegraphics[trim=0 0 0 10, clip, width=0.7\textwidth]{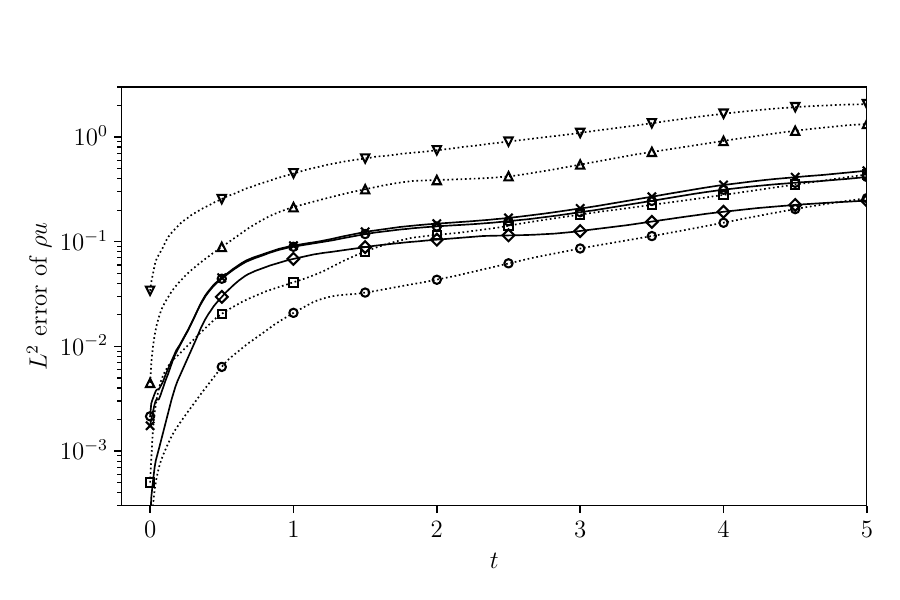}} \\
    \subfloat[$\rho v$]{\label{fig:error_evolution_rhov}\includegraphics[trim=0 0 0 10, clip, width=0.7\textwidth]{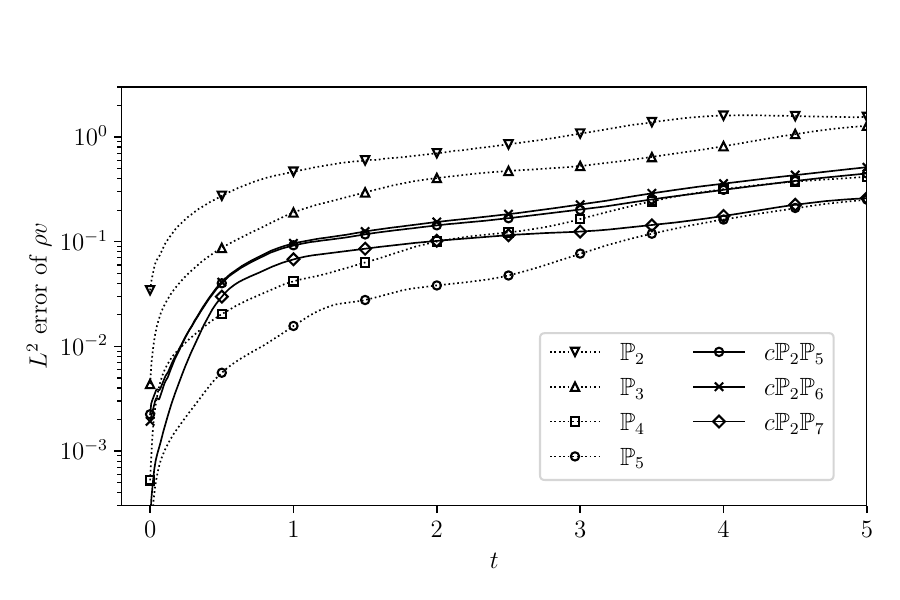}}
    \caption{The $L^2$-error evolutions of the standard DGSEM-LGL methods and the $c\mathbb{P}_n\mathbb{P}_m$ methods for the two dimensional DHIT problem (the markers are shown every time interval equals to 0.5).}
\label{fig:error_evolution_rhou_and_rhov}
\end{figure}

To quantitatively compare the accuracy of different methods, the $L^2$-error evolutions of $\rho u$ and $\rho v$ are plotted in \cref{fig:error_evolution_rhou_and_rhov}. Here, three $c\mathbb{P}_n\mathbb{P}_m$ schemes are studied ($n=2,\;m=5,6,7$). It can be seen that the $c\mathbb{P}_2\mathbb{P}_5$ and $c\mathbb{P}_2\mathbb{P}_6$ methods show an error level similar to the standard $\mathbb{P}_4$ method, and the error of $c\mathbb{P}_2\mathbb{P}_7$ is even lower and comparable to that of the standard $\mathbb{P}_5$ method. In addition, the evolutions of the kinetic energy (KE) are compared in \cref{fig:KE_t0-5_p2px}, where it shows that the $c\mathbb{P}_2\mathbb{P}_7$ method dissipates the KE at a rate similar to that of the reference solution. We also check the energy spectrum at $t=5$ in \cref{fig:Ek_p2px_t=5}, where the energy spectrum in the inertial range flattens towards the classical scaling $k^{-3}$, in agreement with the Kraichnan–Batchelor–Leith (KBL) theory of two-dimensional turbulence\cite{kraichnan_inertial_1967,batchelor_computation_1969,leith_atmospheric_1971}. The energy spectrum of the standard $\mathbb{P}_2$ method differs from that of the reference solution, while the $c\mathbb{P}_2\mathbb{P}_7$ method accurately captures the distribution. The vertical dashed line in \cref{fig:Ek_p2px_t=5} denotes the cut-off wavenumber of 4 points per wavelength (ppw) \cite{gassner_comparison_2011}.

\begin{figure}[tbhp]
\centering
    \subfloat[]{\label{fig:KE_t0-5_p2px}\includegraphics[trim=0 0 0 10, clip, width=0.5\textwidth]{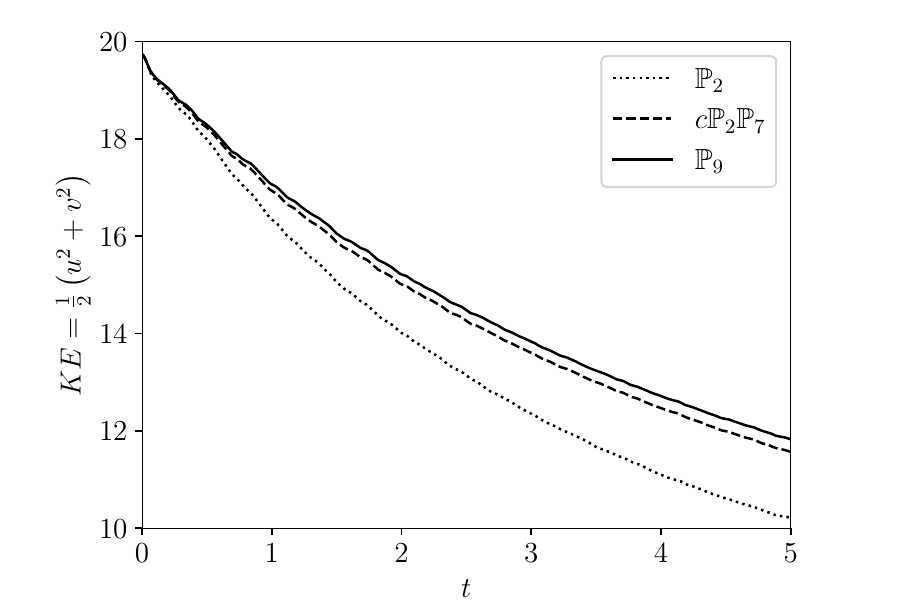}} 
    \subfloat[]{\label{fig:Ek_p2px_t=5}\includegraphics[trim=0 15 15 10, clip, width=0.5\textwidth]{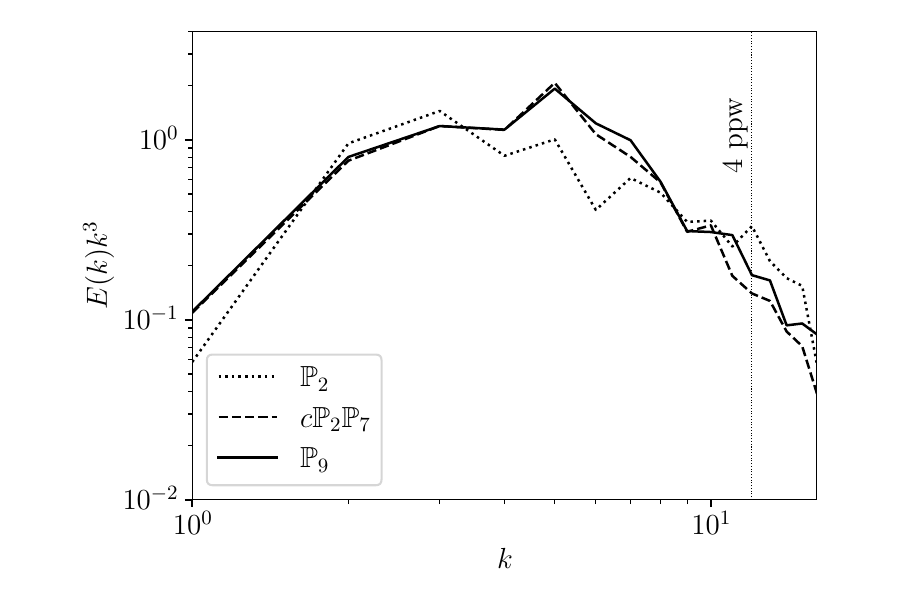}}
    \caption{(a) The evolutions of kinetic energy(KE); (b) the energy spectrum at $t=5$.  }
\label{fig:KE_and_Ek}
\end{figure}

\begin{table}[h]
{\scriptsize
  \caption{The simulation details of different methods for the two-dimensional DHIT problem}  \label{tab:info_p2px}
\centering
\begin{center}
\begin{tabular}{c|cccccc|ccc} % 12 columns total
\hline
% Row 1: Spans 2, then 5, then 2 = 9 total
\multicolumn{1}{c|}{methods} & 
\multicolumn{6}{c|}{standard DGSEM-LGL} & 
\multicolumn{3}{c}{$c\mathbb{P}_n\mathbb{P}_m$} \\ \hline

% Row 2: 9 unequal columns (content determines width)
notation & $\mathbb{P}_2$ & $\mathbb{P}_3$ & $\mathbb{P}_4$ & $\mathbb{P}_5$ & $\mathbb{P}_6$ & $\mathbb{P}_7$ & $c\mathbb{P}_2\mathbb{P}_5$ & $c\mathbb{P}_2\mathbb{P}_6$ & $c\mathbb{P}_2\mathbb{P}_7$ \\
$n$ & 2 & 3 & 4 & 5 & 6 & 7 & 2 & 2 & 2 \\ 
$m$ & - & - & - & - & - & - & 5 & 6 & 7 \\ 
$\Delta t(\times10^{-3})$ & 2.0 & 1.0 & 1.0 & 1.0 & 0.5 & 0.5 & 2.0 & 2.0 & 2.0 \\ 
CFL & 0.31 & 0.22 & 0.28 & 0.34 & 0.2 & 0.23 & 0.31 & 0.31 & 0.31 \\ 
CPU time $[s]$ & 8.63 & 21.5 & 42.1 & 50.1 & 125.4 & 151.7 & 17.4 & 19.1 & 19.6 \\
CPU time per step $[ms]$ & 3.45 & 4.30 & 8.43 & 10.0 & 12.5 & 15.2 & 6.96 & 7.65 & 9.25 \\
normalized CPU time $[s]$ & 2.67 & 4.67 & 11.8 & 17.1 & 25.24 & 35.3 & 5.39 & 5.92 & 6.08 \\ \hline
\end{tabular}
\end{center}
}
\end{table}

Finally, the computational costs of the standard methods and the $c\mathbb{P}_n\mathbb{P}_m$ methods are compared in \cref{tab:info_p2px}, where the sizes of $\Delta t$ are also given. Since the CFL number is slightly different, the normalized CPU time (CPU time $\times$ CFL) is used to measure the efficiency of different methods. It can be seen that, for the standard DGSEM-LGL methods, the computational cost rises significantly as the  polynomial order increases. Regarding the $c\mathbb{P}_n\mathbb{P}_m$ methods, increasing $m$ will not lead to considerable overhead in cost. As mentioned above, the advantage of $c\mathbb{P}_n\mathbb{P}_m$ is more pronounced when viscous terms are involved in the simulation. From \cref{fig:error_evolution_rhou_and_rhov} and \cref{tab:info_p2px}, we can find that the $c\mathbb{P}_2\mathbb{P}_5$ methods achieve the same error level as the standard $\mathbb{P}_4$ method with only about half the normalized CPU time ($5.39s$ and $11.8s$, respectively). Similarly, the $c\mathbb{P}_2\mathbb{P}_7$ method obtains a solution close to that of the standard $\mathbb{P}_5$ method, only using $6.08s$ instead of the latter's $17.1s$. Therefore, we can conclude that the standard $\mathbb{P}_4$ and $\mathbb{P}_5$ methods are accelerated by about $\times2.2$ and $\times2.8$, respectively, and these results demonstrate the great potential of the $c\mathbb{P}_n\mathbb{P}_m$ method when simulating viscous flow on under-resolved meshes.

\section{Conclusion}
This work introduced a corrected $\mathbb{P}_n\mathbb{P}_m$ (shortened to $c\mathbb{P}_n\mathbb{P}_m$) framework for DGSEM-LGL discretizations that recovers high-order accuracy from low-order evolved degrees of freedom for sufficiently smooth solutions. The method is based on the separation of scales between components of low- and high-order. At the continuous level, the evolution of the low-order modes was shown to depend on the projected flux and the interfacial numerical flux. At the discrete level, applying the projection to the DGSEM-LGL semi-discretization revealed an additional correction term associated with the highest retained Legendre mode. This term compensates for the inexactness of LGL quadrature in that mode and is essential to preserve the designed order of accuracy.

The corrected evolution was coupled with a compact projection-based reconstruction operator, which yielded the proposed cPnPm scheme. In contrast to reconstruction strategies that solve enlarged constrained least-squares problems, the present reconstruction computes the high-order components from a smaller local system and can be precomputed in matrix form. Except for the reconstruction step, the scheme evaluates the spatial operators on the low-order LGL nodes, reducing the number of evolved degrees of freedom and lowering memory requirements.

The convergence analysis, formulated through an auxiliary reconstructed scheme, shows that the $c\mathbb{P}_n\mathbb{P}_m$ approximation achieves the optimal $m+1^{th}$ order under the stated smoothness, stability and reconstruction assumptions. The  error estimates derived allow for the identification of the main contributions from the projected evolution and the reconstruction error, and we include  numerical experiments to verify these estimates. Tests ranging from one-dimensional linear advection and nonlinear scalar problems to two-dimensional Euler and viscous flow simulations confirm the expected convergence rates. In the 2D decaying homogeneous isotropic turbulence case, the $c\mathbb{P}_n\mathbb{P}_m$ schemes reproduce accuracy levels comparable to higher-order DGSEM-LGL discretizations while requiring substantially less normalized CPU time in the reported setup.

These results indicate that the $c\mathbb{P}_n\mathbb{P}_m$ framework is a promising strategy for smooth, under-resolved flow simulations in which memory bandwidth and high-order operator costs are limiting factors. Future work should extend the framework to nonlinear reconstruction for shock-capturing, positivity-preserving and entropy-stable formulations, curved and complex geometries, and three-dimensional large-scale simulations.

\section*{Acknowledgments}
Xukun Wang acknowledges the financial support of the China Scholarship Council(CSC, project number: 202406290060).

Esteban Ferrer and Oscar Marino acknowledge the funding from the European Union (ERC, Off-coustics, project number 101086075). Views and opinions expressed are, however, those of the authors only and do not necessarily reflect those of the European Union or the European Research Council. Neither the European Union nor the granting authority can be held responsible for them.

Esteban Ferrer acknowledges the funding received by the Grant DeepCFD (Project No. PID2022-137899OB-I00) funded by MICIU/AEI/10.13039/501100011033 and by ERDF, EU.

\bibliographystyle{siamplain}
\bibliography{references}

\newpage
\appendix

\section{Detailed spatio-temporal discretisation}
\label{sec: AppendixA}

Generally, for the conservation law in \cref{eq:1d_conservation_law}, DG method approximates $\vec{q}$ in a piecewise polynomial space:
\begin{equation}
    \mathcal{V}_h^n := \left\{ v \in L^2(\Omega)\; \vert \; v\vert_{\Omega_i} \in \mathbb{P}_n(\Omega_i), \forall \Omega_i \in \mathcal{T}_h \right\},
\end{equation}
where $\mathcal{T}_h$ is a tessellation of the domain $\Omega$ into non-overlapping $K$ elements: $\Omega_i = [x_{i-1},x_{i}],i=1,2,\dots,K$(mesh); $\mathbb{P}_n(\Omega_i)$ is the space of polynomials of degree $\leq n$ on $\Omega_i$, and $\{x_i\}_{i=0}^K$ denote the location of faces of elements.

In practice, DG solves the equation in a reference element, $E = [-1,1]$, which is geometrically transformed from a physical element $\Omega_{i}= [x_{i-1},x_i]$ through a transfinite mapping:
\begin{equation}
\label{eq: geo_mapping}
    \xi = \xi(\Omega_i,x) = \frac{2}{\Delta x_i}\left( x - \frac{x_{i-1}-x_{i}}{2}\right),
\end{equation}
where $\Delta x_i = x_i - x_{i-1}$ is the length of the element $\Omega_{i}$. The transformation is applied to \cref{eq:1d_conservation_law} in $\Omega_i$, with the result that the following:
\begin{equation}
    \label{eq:1dConservationTransformed}
    \mathcal{J}_i\partial_t\vec{q} + \partial_{\xi}\vec{f}(\vec{q})=0,
\end{equation}
where $\mathcal{J}_i = \partial x/\partial \xi$ is the Jacobian of the inverse transfinite mapping and equals $\Delta x_i/2$.

Following the standard process of DG, we transfer the problem of solving the strong differential form of equation to the weak form by taking the inner product of \cref{eq:1dConservationTransformed} with the test function $\phi$, applying the integration by parts and replacing the boundary flux by a numerical one $\vec{f}^{\ast}$, which leads to:
\begin{equation}
\label{eq:weak_form}
    \mathcal{J}_i\langle \partial_t\vec{q}\,,\phi \rangle + \vec{f}^{\ast}\phi\Big\vert_{-1}^1-\langle\vec{f}\,,\phi_{\xi}\rangle = 0,
\end{equation}
where $\langle \cdot\,,\cdot\rangle$ denotes the inner product of two functions defined in $E$. Equivalently, the strong form can be obtained by integrating the last term in Equation \cref{eq:weak_form} by parts:

In Gassner's nodal DGSEM\cite{gassner_comparison_2011,gassner_skew_symmetric_2013}, the test functions are taken as the Lagrange interpolates on $n+1$ Legendre-Gauss-Lobatto (LGL) nodes, $\{\ell_j\}_{j=0}^n$, and the integrations are approximated using $n^{th}$-order LGL quadrature:
\begin{equation}
\label{eq:inner_product_form}
    \mathcal{J}_i\langle \partial_t\vec{q}\,,\ell_j \rangle_{n} + \left(\vec{f}^{\ast}-\vec{f}\right)\ell_j\Big\vert_{-1}^1+\langle\vec{f}_{\xi}\,,\ell_j\rangle_n = 0,\;j=0,1,\dots,n,
\end{equation}
where $\langle \cdot\,,\cdot\rangle_n$ denotes the discrete inner product approximated by the $n^{th}$-order LGL quadrature. Taking one component of the conservative system, $u$, for example, the $n^{th}$-order Gassner's nodal DGSEM-LGL can be written in the matrix form:
\begin{equation}  
\mathcal{J}_i\underline{\dot{u}}+\underline{\underline{M}}^{-1}\underline{\underline{B}}(\underline{f}^{\ast} - \underline{f})+\underline{\underline{D}}\,\underline{f} = 0,
\end{equation}
where $\underline{u}=[u_0,\,u_1,\dots,u_n]^{\top}\in \mathbb{R}^{n+1}$ and $\underline{f}=[f_0,\,f_1,\dots,f_n]^{\top}\in \mathbb{R}^{n+1}$ are the vectors of the nodal values of the solution and the corresponding flux in the element of $n^{th}$-order. $\underline{\underline{M}} = diag([\omega_0,\omega_1,\dots,\omega_n])$ is the so-called mass matrix (diagonal matrix of GL quadrature weights $\omega_i$). $\underline{\underline{B}} = diag([-1,\,0,\dots,0,\,1])$, $\underline{f}^{\ast} = [f^{\ast}_L,\, 0,\dots,\,0,\,f^{\ast}_R]^{\top}$ and $\underline{\underline{D}}$ is the derivative matrix defined as
\begin{equation}
\label{eq:D}
    \left[\underline{\underline{D}}\right]_{ij} = \frac{\partial \ell_j(\xi_i)}{\partial\xi}.
\end{equation}
$f^{\ast}_L$ and $f^{\ast}_R$ denote the numerical flux on the left and right boundaries, respectively. Generally, the second term is referred to as the surface contribution, while the third term is considered as the volume contribution.  

Alternatively, the surface term, $\underline{\underline{M}}^{-1}\underline{\underline{B}}(\underline{f}^{\ast} - \underline{f})$, can be reformulated as the scalar product between a real number and a vector:
\begin{equation}
    \underline{\underline{M}}^{-1}\underline{\underline{B}}(\underline{f}^{\ast} - \underline{f}) = \frac{(f^{\ast}_L - f_0)}{\omega_0}\underline{b}_L + \frac{(f^{\ast}_R - f_n)}{\omega_n}\underline{b}_R,
\end{equation}
where $\underline{b}_L = [-1,\,0,\dots,0]^{\top}$ and $\underline{b}_R = [0,\dots,0,\,1]^{\top}$, which leads to a more explicit form of Equation \cref{eq:DGSEM-LGL_0}:
\begin{equation}
    \mathcal{J}_i\underline{\dot{u}}+\frac{(f^{\ast}_L - f_0)}{\omega_0}\underline{b}_L + \frac{(f^{\ast}_R - f_n)}{\omega_n}\underline{b}_R+\underline{\underline{D}}\,\underline{f} = 0.
\end{equation}

\section{Details on discrete operators}
\label{sec: AppendixB}

\subsection{Transformation operators}
\label{sec: AppendixB1}
Considering that $\underline{u}$ can be expanded by series of Legendre polynomials $\{L_i \}_{i=0}^n$:
\begin{equation}
\label{eq:modal_index}
    u_i = \sum_{j=0}^n\hat{u}_jL_j(\xi_i),
\end{equation}
the transformation operator between nodal and modal space can be represented by the Vandermonde matrix of Legendre polynomials $\underline{\underline{V}}\in \mathbb{R}^{(n+1)\times (n+1)}$ and its inverse:
\begin{align}
    \underline{u} &= \underline{\underline{V}}\,\underline{\hat{u}}\\
    \underline{\hat{u}} &= \underline{\underline{V}}^{-1}\,\underline{u},
\end{align}
where $\underline{\hat{u}} = [\hat{u}_0,\,\hat{u}_1,\dots,\hat{u}_n]^{\top}\in \mathbb{R}^{n+1}$ is the vector of the modal coefficients of the Legendre polynomials. The Vandermonde matrix $\underline{\underline{V}}$ can be written as:
\begin{equation}
    \underline{\underline{V}} = [\underline{v}_0,\,\underline{v}_1,\dots, \underline{v}_n],\;\underline{v}_i = [L_i(\xi_0),\,L_i(\xi_1),\dots, L_i(\xi_n)]^{\top},
\end{equation}
where $\underline{v}_i$ denotes the column vector of the nodal values of $i^{th}$ Legendre modes. Alternatively, the Vandermonde matrix can also be written using the composition of row vectors, $\underline{r}_i^{\top}$
\begin{equation}
\underline{\underline{V}} = [\underline{r}_0,\,\underline{r}_1,\dots, \underline{r}_n]^{\top},\;
    \underline{r}_i^{\top} = [L_0(\xi_i),\,L_1(\xi_i),\dots, L_n(\xi_i)]^{\top}.
\end{equation}

However, the inverse operation is inconvenient to compute, and it can be written more explicitly. Considering that the Legendre polynomials are orthogonal to each other,
\begin{equation}
    \langle L_i\,,L_j\rangle = \delta_{ij}\hat{\omega}_i,
\end{equation}
where $\hat{\omega}_i := \Vert L_i\Vert_{L^2}^2 = {2}/{(2i+1)}$ denotes the square of the $L^2$-norm of the $i^{th}$-order Legendre polynomial, and the discrete orthogonality can be approximated by Gauss quadrature rules
\begin{equation}
\underline{\underline{V_G}}^{\top}\underline{\underline{M_G}}\,\underline{\underline{V_G}} = \underline{\underline{\hat{M}}}
\end{equation}
where $\underline{\underline{\hat{M}}} = diag([\hat{\omega}_0,\,\hat{\omega}_1,\dots,\hat{\omega}_n])$ is the so-called modal mass matrix (diagonal matrix of the square of the $L^2$-norm) and the sub-scripts $G$ denote the matrices corresponding to the Gauss quadrature. However, the LGL quadrature is exact only for polynomials up to $2n-1$ order, and thus the discrete integral of the square of the highest order Legendre polynomial is no longer equal to the analytic value
\begin{equation}
\label{eq:vpTWvp}
    \underline{v}_n^{\top}\underline{\underline{M}}\,\underline{v}_n \neq \hat{\omega}_n = \frac{2}{2n+1}.
\end{equation}
Noting that the weights of LGL-quadrature are
\begin{equation}
    \omega_i = \frac{2}{n(n+1)L_n^2(\xi_i)},\;i=0,1,\dots,n,
\end{equation}
substituting this expression into \cref{eq:vpTWvp} leads to
\begin{equation}
    \underline{v}_n^{\top}\underline{\underline{M}}\,\underline{v}_n = \sum_{i=0}^nL_n^2(\xi_i)\omega_i = \frac{2}{n}.
\end{equation}
Therefore, if we define the modified modal mass matrix $\underline{\underline{\hat{M}^{\prime}}}=diag([\hat{\omega}_0^{\prime},\hat{\omega}_1^{\prime},\dots,\hat{\omega}_n^{\prime}])$ by copying $\underline{\underline{\hat{M}}}$ except that the last element $\hat{w}_n = 2/(2n+1)$ is replaced by $\hat{\omega}_n^{\prime} = 2/n$:
\begin{equation}
    \hat{\omega}_i^{\prime}=\left\{ \begin{array}{rcl}
\hat{\omega}_i & \mbox{for}
& \leq i \leq n-1 \\ 
2/n & \mbox{for} & i=n,
\end{array}\right.
\end{equation}
then the following relation holds in the LGL nodes
\begin{equation}
\label{eq:VTWV=tildeN}\underline{\underline{V}}^{\top}\underline{\underline{M}}\,\underline{\underline{V}} = \underline{\underline{\hat{M}^{\prime}}},
\end{equation}
and the inverse of the Legendre Vandermonde matrix, $\underline{\underline{V}}$, can be formulated as:
\begin{equation}
\label{eq: inverse_Vandermonde}
    \underline{\underline{V}}^{-1} = \underline{\underline{\hat{M}^{\prime}}}^{-1}\underline{\underline{V}}^{\top}\underline{\underline{M}}.
\end{equation}
So far, the transform matrices in both directions are well defined.

\subsection{Projection and interpolation operators}\label{sec:appendixB}
Once we clarify the high- and low-order notations, we can define the projection and interpolation operators using the transform matrices defined in \cref{sec: AppendixB1}. Since $\underline{\underline{V}}^{(m)} = [\underline{v}_0^{(m)},\underline{v}_1^{(m)},\dots,\underline{v}_m^{(m)}]$, we can split it into two parts:
\begin{equation}
    \underline{\underline{V}}^{(m)} = \left[\overline{\underline{\underline{V}}^{(m)}} \;\Big| \;\widetilde{\underline{\underline{V}}^{(m)}}\right],
\end{equation}
where
\begin{equation}
\begin{split}
        \overline{\underline{\underline{V}}^{(m)}} :&= \left[ \underline{v}_0^{(m)},\underline{v}_1^{(m)},\dots, \underline{v}_n^{(m)}\right] \in \mathbb{R}^{(m+1)\times(n+1)}, \\
        \widetilde{\underline{\underline{V}}^{(m)}} :&= \left[ \underline{v}_{n+1}^{(m)},\underline{v}_{n+2}^{(m)},\dots, \underline{v}_{m}^{(m)}\right] \in \mathbb{R}^{(m+1)\times r}.
\end{split}
\end{equation}
Now, the projection operator from $\mathbb{R}^{m+1}$ to $\mathbb{R}^{n+1}$ can be defined as:
\begin{equation}
\label{eq:projection}    \underline{\underline{\Pi}}:=\underline{\underline{V}}^{(n)}\overline{\left[ \underline{\underline{V}}^{(m)} \right]^{-1}} = \underline{\underline{V}}^{(n)}\overline{\underline{\underline{V}}^{(m)}}^{+}\in \mathbb{R}^{(n+1)\times (m+1)},
\end{equation}
where $+$ denotes the Moore–Penrose inverse (or pseudoinverse) and we have:
\begin{equation}
    \overline{\underline{u}^{(m)}}:=\underline{\underline{\Pi}}\;\underline{u}^{(m)} = \underline{u}^{(n)}.
\end{equation}
Similarly, we can define the inverse operator (interpolation) as follows:
\begin{equation}    \underline{\underline{\Pi}}^{\ast}:=\overline{\underline{\underline{V}}^{(m)}}\left[ \underline{\underline{V}}^{(n)} \right]^{-1} \in \mathbb{R}^{(m+1)\times (n+1)}
\end{equation}
with the property: $\underline{\underline{\Pi}}\,\underline{\underline{\Pi}}^{\ast} = \underline{\underline{I}}^{(n)}$, where $\underline{\underline{I}}^{(n)}$ denotes the identity matrix of $n^{th}$-order. Thus, \cref{eq:u_decomposition} can be expressed in high-order nodes:
\begin{equation*}
    \underline{u}^{(m)} = \underline{\underline{\Pi}}^{\ast}\underline{u}^{(n)}+\underline{\tilde{u}}.
\end{equation*}

\section{Details on the derivative matrix}
\label{sec: AppendixC}
\subsection{The block structure of modal derivative matrix}
\label{sec: AppendixC1}
The $m^{th}$-order $\underline{\underline{\hat{D}}}^{(m)}$ and $\underline{\underline{\hat{Q}}}^{(m)}$ have the following block structure:
\begin{equation}
\label{eq:splitD}
    \underline{\underline{\hat{D}}}^{(m)} =\left[ \begin{array}{cc}
         \underline{\underline{\hat{D}}}^{(n)}& 2\left[\underline{\underline{\hat{M}}}^{(n)}\right]^{-1}\underline{\underline{\hat{R}}} \\
            %\hdashline
         & \underline{\underline{\widetilde{\hat{D}}}}   \\
    \end{array} \right],
\end{equation}
and
\begin{equation}
    \label{eq:splitQ}
    \underline{\underline{\hat{Q}}}^{(m)} =\left[ \begin{array}{ll}
         \underline{\underline{\hat{Q}}}^{(n)}& \underline{\underline{\hat{R}}} \\
            %\hdashline
         & \underline{\underline{\hat{Q}}}^{(r-1)}   \\
    \end{array} \right],
\end{equation}
where 
\begin{equation}
\label{eq:Rhat}
    \left[ \underline{\underline{\hat{R}}}\right]_{ij} = \left\{\begin{array}{cl}
         1&,n+i+j\; \text{is even};  \\ 
         0&, \text{else.}
    \end{array}
    \right.,\;0\leq i\leq n,\;0\leq j\leq r-1.
\end{equation}

\subsection{The structure of $\underline{\underline{\hat{R}}}$}
\label{sec: AppendixC2}
Considering that $\underline{\underline{\hat{R}}}$ can be formulated as a combination of the unit vector and the alternating sign vector
\begin{equation}
\label{eq:2R}
    2\underline{\underline{\hat{R}}} = \underline{1}^{(n)}{\underline{1}^{{(r-1)}}}^{\top} - \underline{\check{1}}^{(n)}{\underline{\tilde{\check{1}}}}^{\top},
\end{equation}
where the four ingredients come from the splitting of $\underline{1}^{(m)}$ and $\underline{\check{1}}^{(m)}$
\begin{equation*}
    \underline{1}^{(m)} = \left[ \begin{array}{l}
         \underline{1}^{(n)}  \\
         \underline{1}^{(r-1)} 
    \end{array} \right],\;
    \underline{\check{1}}^{(m)} = \left[ \begin{array}{l}
         \underline{\check{1}}^{(n)}  \\
         \underline{\tilde{\check{1}}}
    \end{array} \right],\,\underline{\tilde{\check{1}}}:=\underline{\check{1}}^{{(m)}}_{\setminus
    \underline{\check{1}}^{(n)}},
\end{equation*}
and 
\begin{align}
    \underline{1}^{(m)} &= [1,1,\dots,1]\in \mathbb{R}^{m+1},\\
    \underline{\check{1}}^{(m)} &= [1,-1,1,\dots]\in \mathbb{R}^{m+1}.
\end{align}
Replacing $\underline{1}^{(n)}$ and $\underline{\check{1}}^{(n)}$ with $\left[\underline{\underline{V}}^{(n)}\right]^{\top}{\underline{b}_R^{(n)}}$ and  $-\left[\underline{\underline{V}}^{(n)}\right]^{\top}{\underline{b}_L^{(n)}}$, respectively, and considering ${\underline{1}^{{(r-1)}}}^{\top}\underline{\tilde{\hat{f}}} = \tilde{f}_n$ and $\underline{\tilde{\check{1}}}^{\top}\underline{\tilde{\hat{f}}} = \tilde{f}_0$, we can obtain \cref{eq: 2Rf}.

\section{Two-dimensional formulation}
\label{sec: Appendix_2d_formulation}
The simplified formulation of two-dimensional projected high-order time-derivatives for variable $u$ reads:
\begin{equation}
\label{eq:2d_projected_ho_time_derivative}
\begin{split}
    \mathcal{J}_{ij}\overline{\dot{u}^{(m)}_{ij}} &+ \delta f^{(n,m)}\left(1,\eta_j^{(n)}\right)\frac{\ell_i^{(\xi,n)}(1)}{\omega_i^{(\xi,n)}} - \delta f^{(n,m)}\left(-1,\eta_j^{(n)}\right)\frac{\ell_i^{(\xi,n)}(-1)}{\omega_i^{(\xi,n)}}+\sum_{k=0}^nD^{(\xi,n)}_{ik}\overline{f^{(m)}_{kj}}\\
    &+\delta g^{(n,m)}\left(\xi_i^{(n)},1\right)\frac{\ell_j^{(\eta,n)}(1)}{\omega_j^{(\eta,n)}} - \delta g^{(n,m)}\left(\xi_i^{(n)},-1\right)\frac{\ell_j^{(\eta,n)}(-1)}{\omega_j^{(\eta,n)}}+\sum_{k=0}^nD^{(\eta,n)}_{jk}\overline{g^{(m)}_{ik}}\\
    &+ \frac{n+1}{2}\left[\delta f^{(n,m)}\left(1,\eta_j^{(n)}\right)-(-1)^n\delta f^{(n,m)}\left(-1,\eta_j^{(n)}\right)\right]L_{n}^{(\xi)}\left(\xi_i^{(n)}\right)\\
    &+ \frac{n+1}{2}\left[\delta g^{(n,m)}\left(\xi_i^{(n)},1\right)-(-1)^n\delta g^{(n,m)}\left(\xi_i^{(n)},-1\right)\right]L_{n}^{(\eta)}\left(\eta_j^{(n)}\right) = 0,
\end{split}  
\end{equation}
where
\begin{equation}
    \begin{split}
        \delta f^{(n,m)}\left(\pm1,\eta\right) &= \Pi^{(\eta)}\left(f^{\ast(m)}\left(  \pm1,\eta\right) - f^{(m)}\left(  \pm1,\eta\right)\right) \\
        \delta g^{(n,m)}\left(\xi,\pm1\right) &= \Pi^{(\xi)}\left(g^{\ast(m)}\left(  \xi,\pm1\right) - g^{(n,m)}\left( \xi,\pm1\right)\right),
    \end{split}
\end{equation}
where $f$ and $g$ denote the $\xi-$ and $\eta-$ directional flux, respectively; $\Pi^{(\xi)}$ and $\Pi^{(\eta)}$ denote the projections along the $\xi-$ and $\eta-$ directions, respectively; and the two dimensional projection is defined as $\overline{(\cdot)} := \Pi^{(\xi)} \otimes \Pi^{(\eta)}$.

\section{Validation of the formulations}

\subsection{One-dimensional case}

We verify the correction term in Equation\cref{eq:correction} using the numerical test where the simplest problem of linear advection with unit wave speed ($f=u$) is solved with periodic boundary conditions:

To simplify the problem, we solve it on one element with $m$- and $n^{th}$-order, respectively. The upwind Riemann solver is used to compute the numerical flux at the boundaries. Since the entire domain is a single element, it leads to the special case that $f^{\ast(m)}_L = f^{\ast(m)}_R=u_m^{(m)}$ and $f^{\ast(n)}_L = f^{\ast(n)}_R=u_n^{(n)}$. According to \cref{eq:correction}, the correction is reduced to
\begin{equation}
\label{eq:correction_linear_advection}
\begin{split}
    \underline{c}^{(n,m)} = &- \frac{\left(u_m^{(m)}-u_n^{(n)}\right)}{{\omega_0^{(n)}}}{\underline{b}_L^{(n)}} - \frac{\left(u_m^{(m)}-u_n^{(n)}\right)}{{\omega_n^{(n)}}}{\underline{b}_R^{(n)}} \\
    &- \frac{n+1}{2}\left[\left(u_m^{(m)}-{u_n^{(n)}}\right)-(-1)^{n}\left(u_m^{(m)}-{u_0^{(n)}}\right)\right] {\underline{v}_n^{(n)}}.
\end{split} 
\end{equation}

\begin{figure}[tbhp]
\centering
    \subfloat[$n=4,m=9$]{\label{fig:c_4_9}\includegraphics[width=0.45\textwidth]{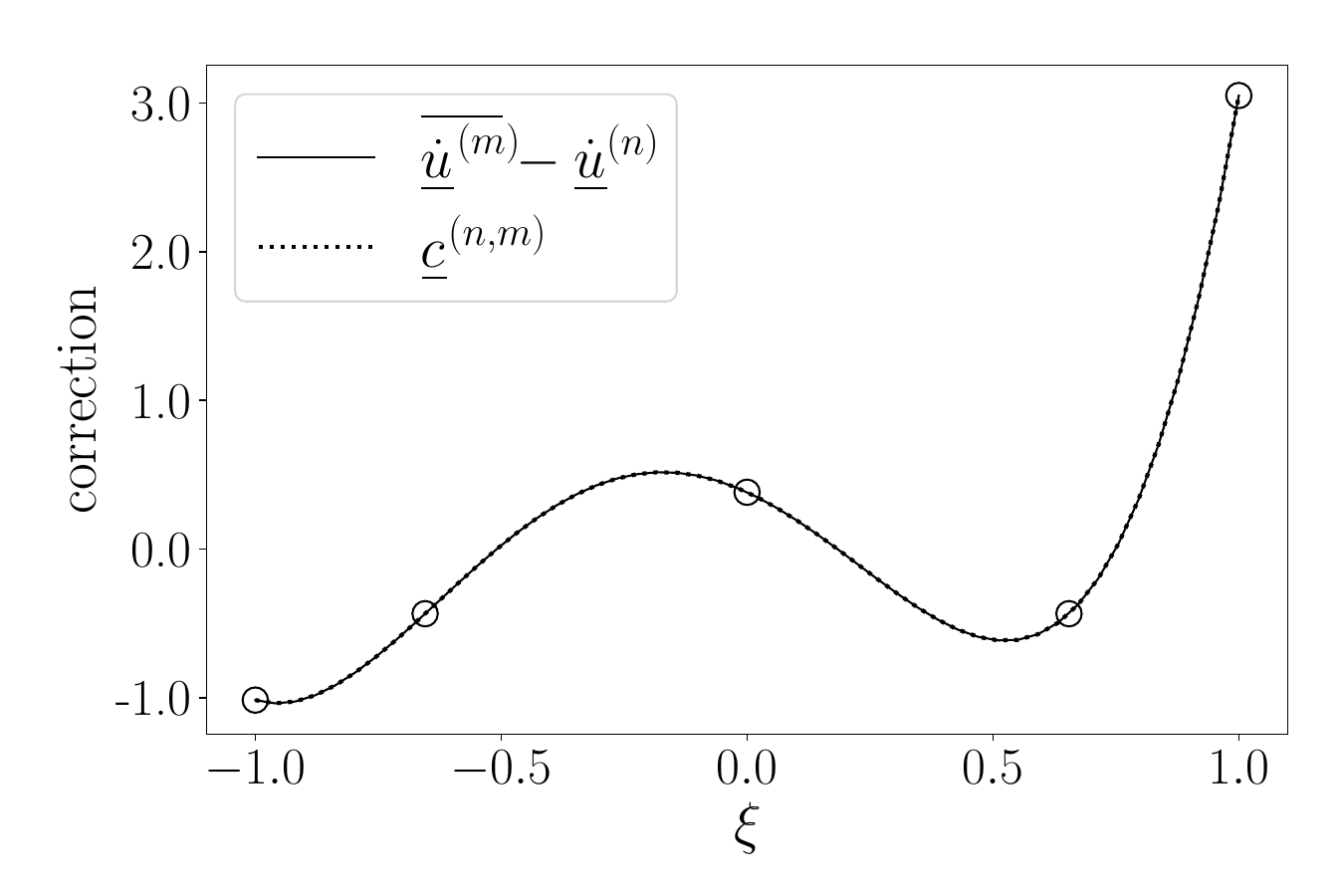}}
    \subfloat[$n=5,m=9$]{\label{fig:c_5_9}\includegraphics[width=0.45\textwidth]{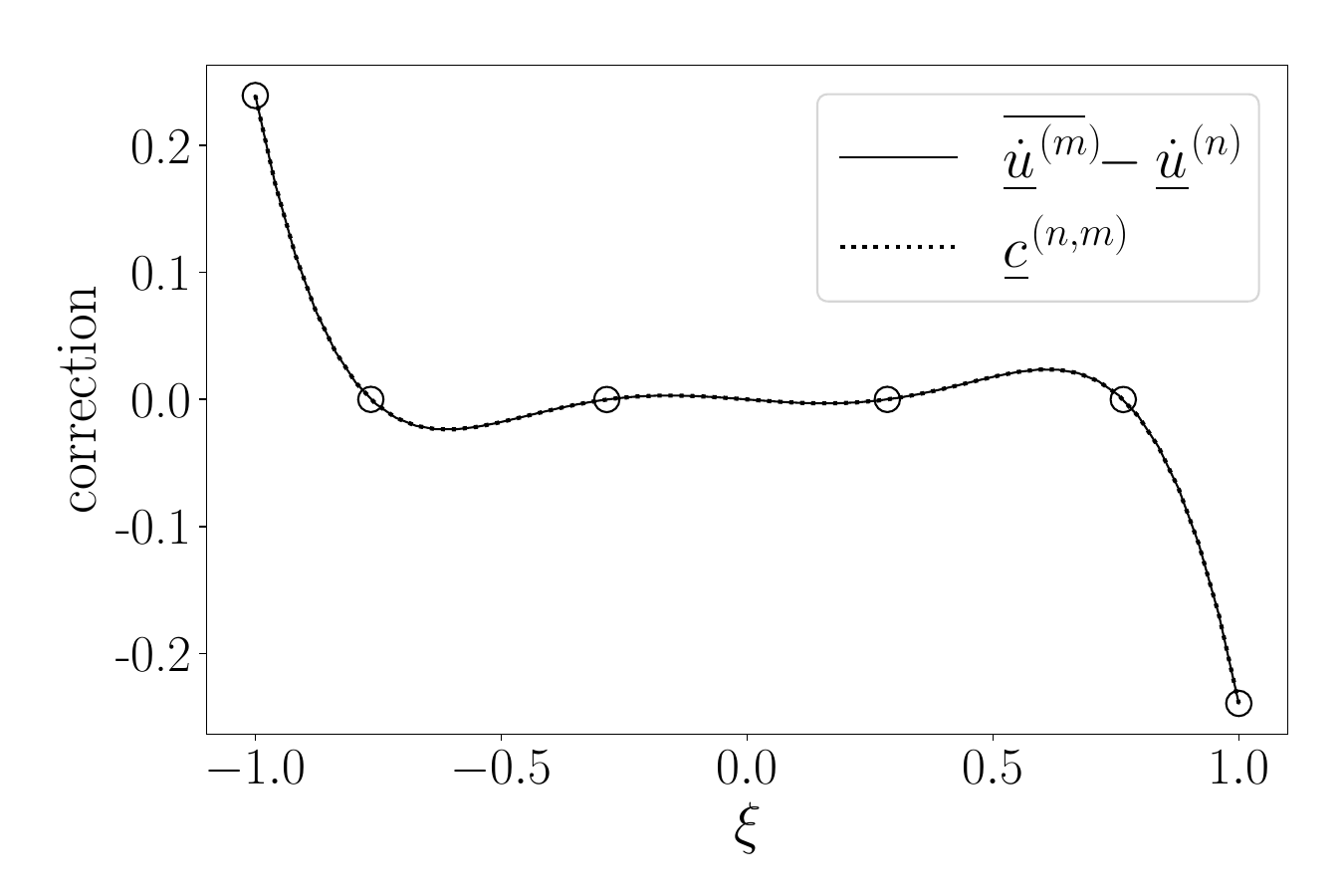}}
    \caption{The comparison between the corrections calculated from filtered high-order time derivative, $\overline{\underline{\dot{u}}^{(m)}}-\underline{\dot{u}}^{(n)}$, and the derived Equation\cref{eq:correction_linear_advection} for $\underline{c}^{(n,m)}$ at $t=0$. The high polynomial order $m$ is fixed to 9 while the low polynomial order $n$ is set to (a) $n=4$ and (b) $n=5$. The circle labels (o) denote the nodal values of the correction on LGL nodes.}
\label{fig:c_4_9_and_c_5_9}
\end{figure}

To verify the correctness of \cref{eq:correction_linear_advection} and thus \cref{eq:correction}, we compare the corrections at $t=0$ calculated from the projected high-order time derivative $\overline{\underline{\dot{u}}^{(m)}}-\underline{\dot{u}}^{(n)}$, and the Equations we derived before. The high polynomial order is fixed to $m=9$, and the corrections on low polynomial order $n=4$ and $n=5$, i.e., $\underline{c}^{(4,9)}$ and $\underline{c}^{(5,9)}$, are shown in \cref{fig:c_4_9_and_c_5_9}. It shows that the result of our derived \cref{eq:correction_linear_advection} is identical to the true one, proving that our formulations are correct. In addition, the concentration of correction near the boundaries is also clearly observed.

\subsection{Two-dimensional case}
Similarly, we verify the correctness of \cref{eq:2d_projected_ho_time_derivative} using a manufactured two-dimensional Euler solution:
\begin{equation}
    \begin{split}
        \rho(x,y,t) &= 2 +\frac{1}{10}\sin\left( \pi (x+y-t)\right),\\
        u(x,y,t) &= 1,\\
        v(x,y,t) &= 1,\\
        E(x,y,t) &= \left(2 +\frac{1}{10}\sin\left( \pi (x+y-t)\right)\right)^2,
    \end{split}
\end{equation}
with the corresponding forcing terms:
\begin{equation}
    \begin{split}
        s_{\rho}(x,y,t) &= c_1 \cos\left( \pi (x+y-t)\right),\\
        s_{\rho u }(x,y,t) &= c_2 \cos\left( \pi (x+y-t)\right) + c_3 \sin\left( 2\pi (x+y-t)\right),\\
        s_{\rho v }(x,y,t) &= c_2 \cos\left( \pi (x+y-t)\right) + c_3 \sin\left( 2\pi (x+y-t)\right),\\
        s_{E }(x,y,t) &= c_4 \cos\left( \pi (x+y-t)\right) + c_5 \sin\left( 2\pi (x+y-t)\right),
    \end{split}
\end{equation}
where $c_1 = \frac{1}{10}\pi$, $c_2 = \frac{1}{10}(3\gamma-2)\pi$, $c_3 = \frac{1}{100}(\gamma-1)\pi$, $c_4 = \frac{1}{5}(3\gamma-2)\pi$ and $c_5 = \frac{1}{100}(2\gamma-1)\pi$. 

The problem is solved on a single element using the upwind Riemann solver, and the periodic boundary conditions are applied. The high and low-orders are set equal to $9$ and $5$, respectively. The projected high-order and low-order time derivatives of density are shown in \cref{fig:dQdt_9_5c} and \cref{fig:dQdt_9_5_2d}, respectively. By subtracting $\dot{\rho}^{(n)}$ from $\overline{\dot{\rho}^{(m)}}$, the true correction can be computed, and it is plotted in \cref{fig:correction_true_2d}. It can be clearly seen that the correction is concentrated on the element boundary. Alternatively, the correction derived from \cref{eq:2d_projected_ho_time_derivative} is shown in \cref{fig:c_nm_2d}, which is exactly identical to the true correction.

\begin{figure}[htbp]
    \centering
    % Subfigure (a)
    \subfloat[$\overline{\dot{\rho}^{(m)}}$]{\label{fig:dQdt_9_5c}\includegraphics[trim=0 10 25 0, clip, width=0.4\textwidth]{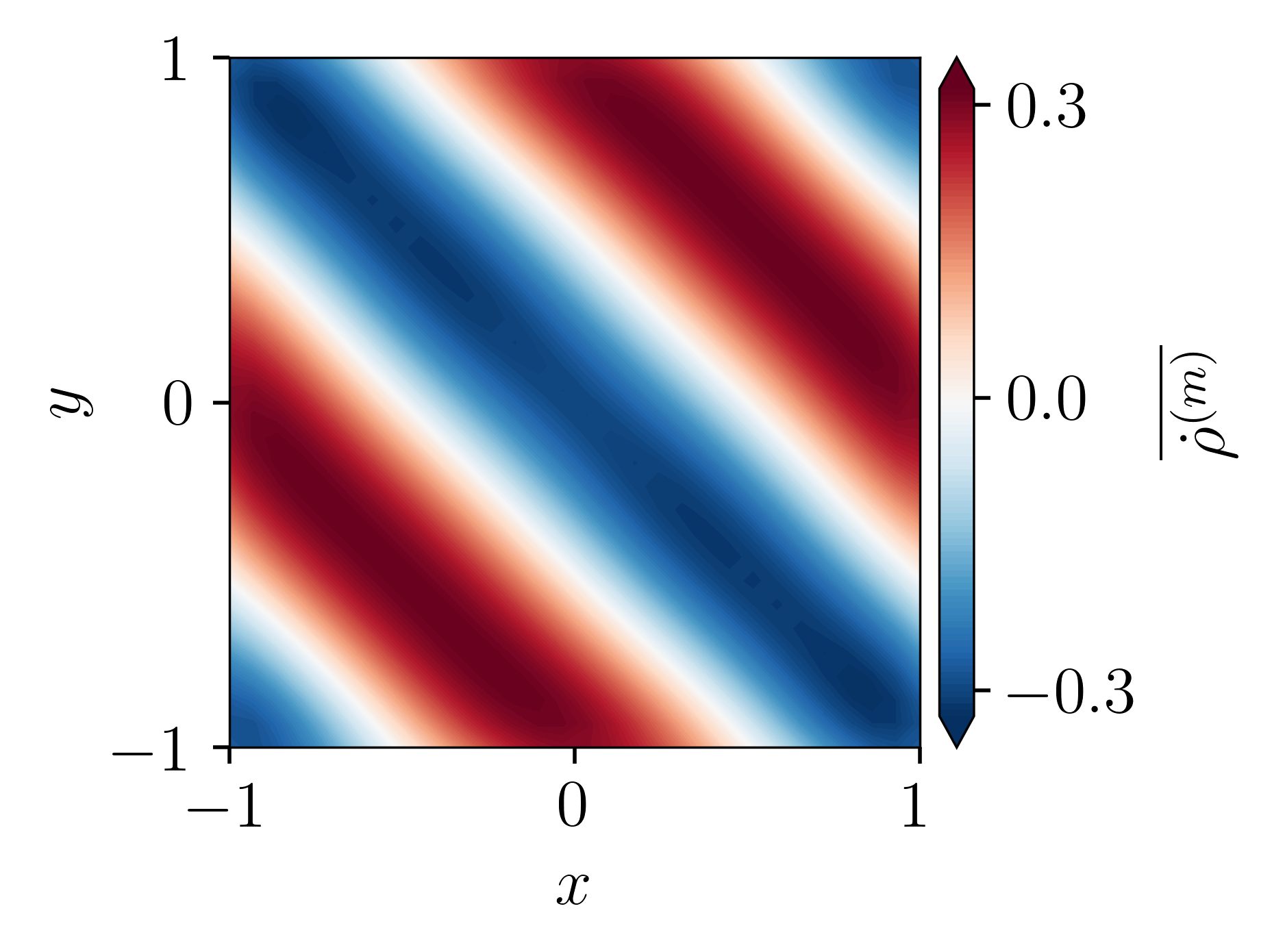}}
    % Subfigure (b) 
    \subfloat[$\dot{\rho}^{(n)}$]{\label{fig:dQdt_9_5_2d}\includegraphics[trim=0 10 25 0, clip, width=0.4\textwidth]{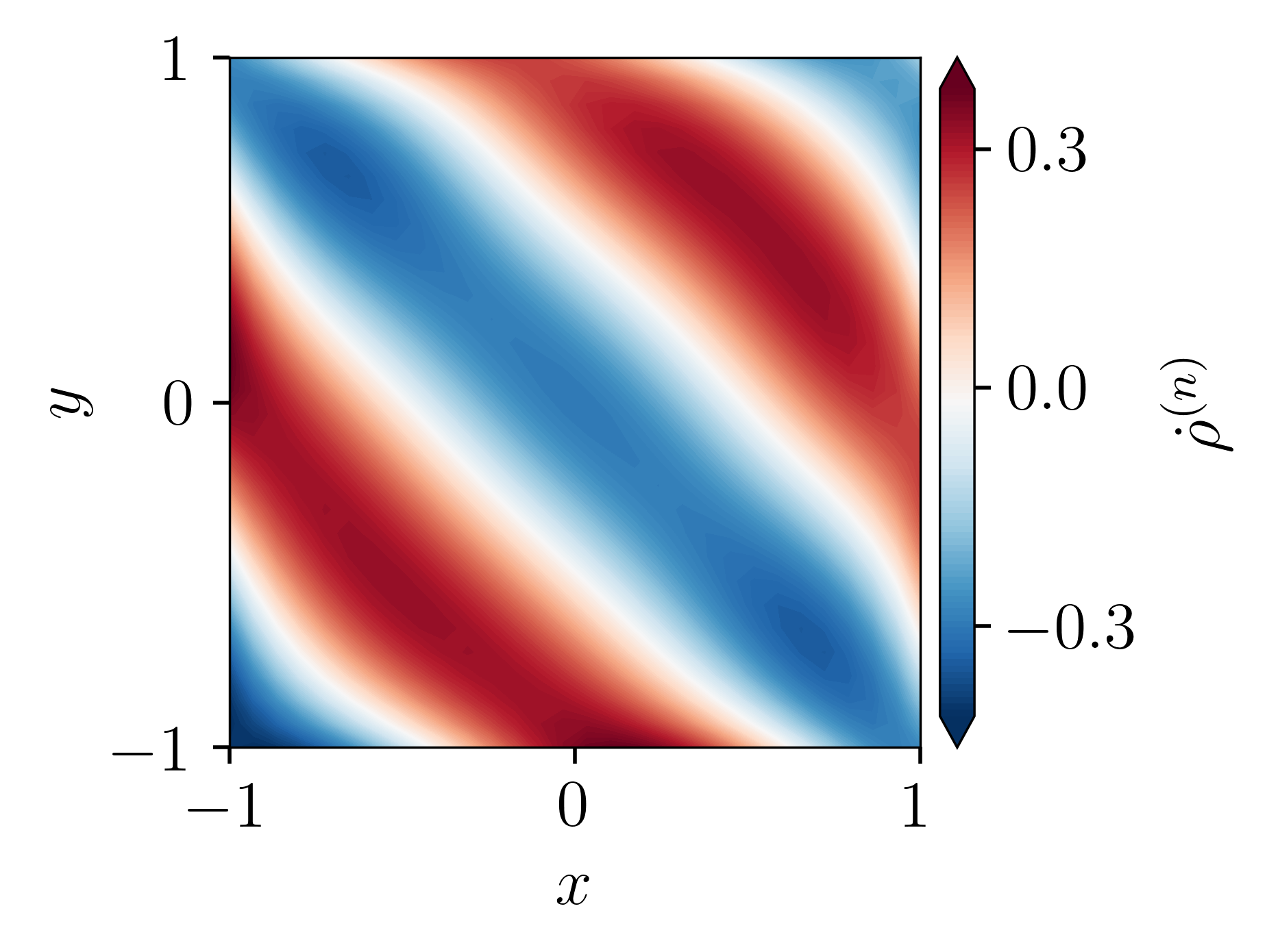}}\\
    % Subfigure (c)
    \subfloat[$\overline{\dot{\rho}^{(m)}} - \dot{\rho}^{(n)}$]{\label{fig:correction_true_2d}\includegraphics[trim=0 10 25 0, clip, width=0.4\textwidth]{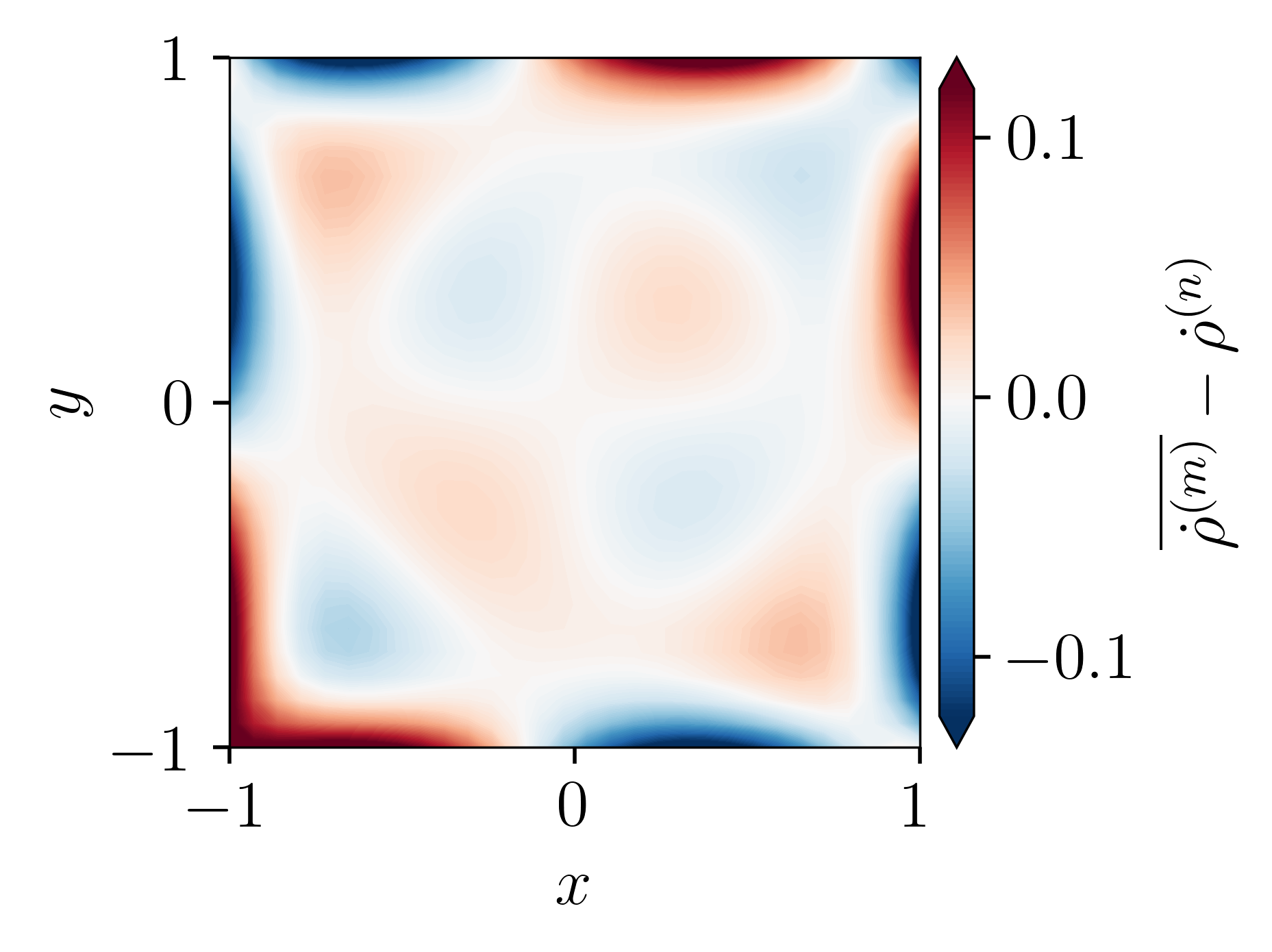}}
    % Subfigure (d)
    \subfloat[$c^{(n,m)}$]{\label{fig:c_nm_2d}\includegraphics[trim=0 10 25 0, clip, width=0.4\textwidth]{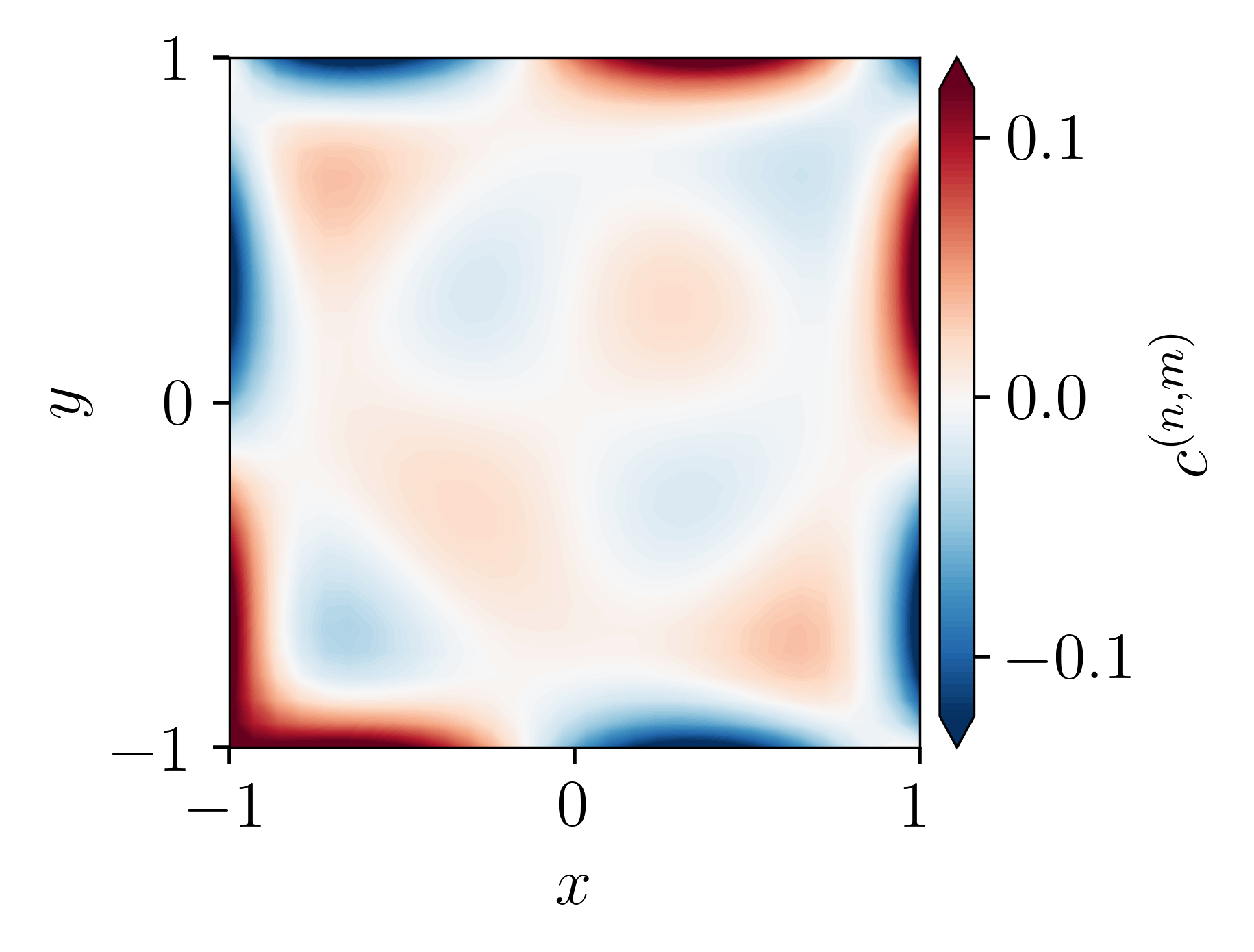}}

    \caption{The comparison between (a) projected high-order and (b) low-order time derivative for variable $\rho$; (c) true correction and (d) computed correction using formulation derived above. Here $m=9$ and $n=5$.}
    \label{fig:dQdt_correction_2d}
\end{figure}

\end{document}